\documentclass[11pt,a4paper,reqno,a4paper]{amsart}
\usepackage{amsmath,amssymb,amsthm, comment,graphicx}
\usepackage{tikz}
\usepackage{fancyhdr}
\usepackage{epsfig}
\usepackage{mathrsfs}
\newtheorem{theorem}{Theorem}[section]

\newtheorem{lemma}{Lemma}[section]

\newtheorem{proposition}{Proposition}[section]
\newtheorem{corollary}{Corollary}[section]

\numberwithin{equation}{section}
\newcommand{\dd}{\mathrm{d}}

\makeatletter \@addtoreset{figure}{section} \makeatother

\begin{document}
\title[Hypersonic Similarity for Potential Flow with Large Data]
{Convergence Rate of the Hypersonic Similarity for Two-Dimensional Steady Potential Flows with Large Data}

\author{Gui-Qiang G. Chen}
\address{ Mathematical Institute, University of Oxford,
 Radcliffe Observatory Quarter, Woodstock Road, Oxford, OX2 6GG, UK;
 School of Mathematical Sciences, Fudan University, Shanghai 200433, China}
 \email{\tt  chengq@maths.ox.ac.uk}

\author{Jie Kuang}
\address{ Innovation Academy for Precision
Measurement Science and Technology, and Wuhan Institute of Physics and Mathematics,
Chinese Academy of Sciences, Wuhan 430071, China}
\email{\tt jkuang@apm.ac.cn}

\author{Wei Xiang}
\address{Department of Mathematics
City University of Hong Kong
Kowloon, Hong Kong, China}
\email{\tt  weixiang@cityu.edu.hk}

\author{Yongqian Zhang}
\address{School of Mathematical Sciences, Fudan University, Shanghai 200433, China}
\email{\tt  yongqianz@fudan.edu.cn}

\keywords{Hypersonic similarity, optimal convergence rate, large data, straight wedge, isothermal hypersonic small disturbance equations,
BV solutions, standard Riemann semigroup.}

\subjclass[2010]{35B07, 35B20, 35D30; 76J20, 76L99, 76N10}
\date{\today}

\begin{abstract}
We establish the optimal convergence rate of the hypersonic similarity
for two-dimensional steady potential flows with {\it large data} past over a straight wedge
in the $BV\cap L^1$ framework, provided that the total variation of the large data multiplied
by $\gamma-1+\frac{a_{\infty}^2}{M_\infty^2}$
is uniformly bounded with respect to the adiabatic exponent $\gamma>1$, the Mach number $M_\infty$ of the incoming steady flow,
and the hypersonic similarity parameter $a_\infty$.
Our main approach in this paper is first to establish the Standard Riemann Semigroup
of the initial-boundary value problem
for the isothermal hypersonic small disturbance equations with large data
and then to compare the Riemann solutions between two systems with boundary locally case by case.
Based on them, we derive the global $L^1$--estimate between the two solutions by employing the Standard Riemann Semigroup
and the local $L^1$--estimates.
We further construct an example to show that the convergence rate is optimal.
\end{abstract}
\maketitle

\tableofcontents

\section{Introduction and Main Theorems}\setcounter{equation}{0}

We are concerned with the optimal convergence rate of the hypersonic similarity
for two-dimensional steady potential flows with {\it large data} past over a straight wedge
in the $BV\cap L^1$ framework.
In gas dynamics, hypersonic flows are the flows with a large Mach number (at least larger than five).
One of the important properties of the hypersonic flows is the hypersonic similarity, which was first developed
by Tsien in \cite{Tsien} for the two-dimensional potential flow
and the three-dimensional axis-symmetric steady potential flow in 1940s.
The convergence without a rate of the hypersonic similarity
for two-dimensional steady potential flows past over a straight wedge
was rigorously verified in \cite{Kuang-Xiang-Zhang-1}.
In this paper, we further develop mathematical analysis
on the hypersonic similarity for steady hypersonic potential flow over a two-dimensional straight wedge with
large data (see Fig. \ref{fig1.1}) to establish the {\it optimal} convergence rate rigorously.

\vspace{5pt}
\begin{figure}[ht]
\begin{center}
\begin{tikzpicture}[scale=0.8]
\draw [line width=0.05cm] (-2,1) --(2.5,2);
\draw [line width=0.05cm] (-2,1) --(2.5,0);
\draw [line width=0.02cm][red] (-2,1) --(1.8,2.7);
\draw [line width=0.02cm][red] (-2,1) --(1.8,-0.7);
\draw [thin][->](-4.5,1.8)--(-2.2,1.8);
\draw [thin][->](-4.5,1.3)--(-2.2,1.3);
\draw [thin][->](-4.5,0.8)--(-2.2,0.8);
\draw [thin][->](-4.5,0.3)--(-2.2,0.3);
\draw [thin](-1.5,0.9)--(-1.2,1.2);
\draw [thin](-1.3,0.85)--(-0.9,1.25);
\draw [thin](-1.1,0.80)--(-0.6,1.3);
\draw [thin](-0.9,0.75)--(-0.3,1.36);
\draw [thin](-0.6,0.7)--(0.2,1.5);
\draw [thin](-0.3,0.6)--(0.7,1.6);
\draw [thin](0,0.55)--(1.15,1.70);
\draw [thin](0.3,0.5)--(1.5,1.75);
\draw [thin](0.6,0.40)--(2.1,1.9);
\draw [thin](0.9,0.35)--(2.2,1.65);
\draw [thin](1.3,0.3)--(2.5,1.5);
\draw [thin](1.7,0.20)--(2.5,1);

\node at (2.1,2.8) {$\textrm{Shock}$};
\node at (2.1,-0.8) {$\textrm{Shock}$};
\node at (-3.5,2.2) {${\textrm{M}_{\infty}}\geq5$};
\end{tikzpicture}		
\end{center}
\caption{Hypersonic flow over a two-dimensional slender straight wedge}\label{fig1.1}
\end{figure}
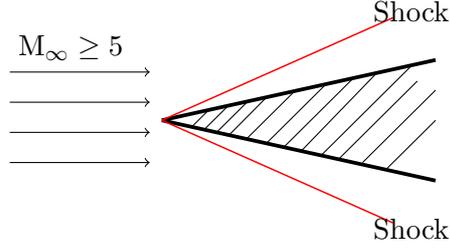

Physically, the law of the hypersonic similarity is also called the
Van Dyke similarity law \cite{Dyke},
which states that, for the steady flow around a slender wedge,
the flow structures are similar under some scaling if the Mach number of the incoming flow is sufficiently large.
More precisely, after scaling, the governed equations of the flow with the same hypersonic similarity parameter
can be approximated by the same hypersonic small-disturbance system.

Mathematically,
consider a uniform hypersonic flow with velocity $(u_{\infty}, 0)$ past over a two-dimensional straight wedge
with boundaries $\bar{y}=\pm\tau b_{0}\bar{x}$, for a fixed constant $b_{0}$ and a sufficiently small parameter $\tau>0$.
The two-dimensional steady isentropic irrotational Euler flows are governed by the following equations:
\begin{equation}\label{eq:1.1}
\begin{cases}
\partial_{\bar{x}}(\bar{\rho} \bar{u})+\partial_{\bar{y}}(\bar{\rho} \bar{v})=0,\\[3pt]
\partial_{\bar{x}}\bar{v}-\partial_{\bar{y}}\bar{u}=0,
\end{cases}
\end{equation}
where density $\bar{\rho}$ and velocity $(\bar{u},\bar{v})$ satisfy the following Bernoulli law:
\begin{equation}\label{eq:1.2}
\frac{1}{2}(\bar{u}^2+\bar{v}^2)+\frac{\bar{\rho}^{\gamma-1}}{\gamma-1}=\frac{1}{2}u_{\infty}^2+\frac{\rho_{\infty}^{\gamma-1}}{\gamma-1}.
\end{equation}
Due to the symmetry of the initial-boundary value problem, we constrain ourself to consider the lower half-plane
in $\mathbb{R}^{2}$ with wedge boundary $\bar{y}=\tau b_{0}\bar{x}$ for $b_0<0$ (see Fig. \ref{fig1.2}).
Then, along the wedge boundary, the flow satisfies the impermeable slip boundary condition:
\begin{equation}\label{eq:1.3}
(\bar{u},\bar{v})\cdot \mathbf{n}=0,
\end{equation}
where $\mathbf{n}=\frac{(\tau b_{0},-1)}{\sqrt{1+\tau^2b^2_{0}}}$ is the {unit} inner normal of the wedge boundary.

\vspace{5pt}
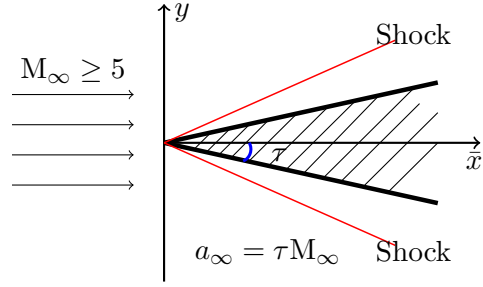
\begin{figure}[ht]
\begin{center}
\begin{tikzpicture}[scale=0.8]
\draw [thick][->](-2,1)--(3.2,1);
\draw [thick][->](-2,-1.3)--(-2,3.3);
\draw [line width=0.06cm] (-2,1) --(2.5,2);
\draw [line width=0.06cm] (-2,1) --(2.5,0);
\draw [line width=0.02cm][red] (-2,1) --(1.8,2.7);
\draw [line width=0.02cm][red] (-2,1) --(1.8,-0.7);
\draw [thin][->](-4.5,1.8)--(-2.5,1.8);
\draw [thin][->](-4.5,1.3)--(-2.5,1.3);
\draw [thin][->](-4.5,0.8)--(-2.5,0.8);
\draw [thin][->](-4.5,0.3)--(-2.5,0.3);
\draw [thin](-1.5,0.9)--(-1.2,1.2);
\draw [thin](-1.3,0.85)--(-0.9,1.25);
\draw [thin](-1.1,0.80)--(-0.6,1.3);
\draw [thin](-0.9,0.75)--(-0.3,1.36);
\draw [thin](-0.6,0.7)--(0.2,1.5);
\draw [thin](-0.3,0.6)--(0.7,1.6);
\draw [thin](0,0.55)--(1.15,1.70);
\draw [thin](0.3,0.5)--(1.5,1.75);
\draw [thin](0.6,0.40)--(2.1,1.9);
\draw [thin](0.9,0.35)--(2.2,1.65);
\draw [thin](1.3,0.3)--(2.5,1.5);
\draw [thin](1.7,0.20)--(2.5,1);
\draw [line width=0.04cm][blue] (-0.6,1)to [out=-60, in=20](-0.7,0.7);
\node at (2.1,2.8) {$\textrm{Shock}$};
\node at (2.1,-0.8) {$\textrm{Shock}$};
\node at (-3.5,2.2) {${\textrm{M}_{\infty}}\geq5$};
\node at (-0.1,0.8) {$\tau$};
\node at (3.1,0.7) {$\bar{x}$};
\node at (-1.7,3.2) {$\bar{y}$};
\node at (-0.3, -0.8) {$a_{\infty}=\tau {\textrm{M}_{\infty}}$};
\end{tikzpicture}		
\end{center}
\caption{Hypersonic similarity law}\label{fig1.2}
\end{figure}

Define the hypersonic similarity parameter:
\begin{eqnarray}\label{eq:1.4}
a_{\infty}:=\tau {\textrm{M}_{\infty}}=\tau u_{\infty}\rho^{-\frac{\gamma-1}{2}}_{\infty}.
\end{eqnarray}
Following the arguments in \cite{Anderson, Tsien}, we define the scaling:
\begin{eqnarray}\label{eq:1.5}
\bar{x}=x,\,\,\, \bar{y}=\tau y,\,\,\, \bar{u}=u_{\infty}(1+\tau^2 u),\,\,\,
\bar{v}=u_{\infty}\tau v,\,\,\, \bar{\rho}=\rho_{\infty}\rho,
\end{eqnarray}
where $\rho_{\infty}=\lim_{\bar{y}\rightarrow -\infty} \bar{\rho}(\bar{y})$ so that $\lim_{y\rightarrow-\infty}\rho(y)=1$.

Substituting \eqref{eq:1.5} into \eqref{eq:1.1}--\eqref{eq:1.2}, we obtain
\begin{eqnarray}\label{eq:1.6}
\begin{cases}
\partial_{x}(\rho(1+\tau^2 u))+\partial_{y}(\rho v)=0,\\[3pt]
\partial_{x}v-\partial_{y}u=0,
\end{cases}
\end{eqnarray}
and 
\begin{eqnarray}\label{eq:1.7}
u+\frac{1}{2}(v^2+\tau^2u^2)+\frac{\rho^{\gamma-1}-1}{(\gamma-1)a^2_{\infty}}=0.
\end{eqnarray}

Now 
the fluid domain and its boundaries (see Fig. \ref{fig1.3}) are given by
\begin{eqnarray*}
\Omega_{\textrm{w}}=\{(x,y)\,:\,x>0, y<b_{0}x \}
\end{eqnarray*}
and
\begin{eqnarray*}
\Gamma_{\textrm{w}}=\{(x,y)\,:\, x>0, y=b_{0}x \}, \quad {\Sigma_{0}}=\{(x,y)\,:\, x=0, y<0\}.
\end{eqnarray*}
Let $\mathbf{n}_{\textrm{w}}=\frac{(b_{0},-1)}{\sqrt{1+b^{2}_{0}}}$ be the unit inner normal vector
of $\Gamma_{\textrm{w}}$. Then
the boundary condition \eqref{eq:1.3} becomes
\begin{eqnarray}\label{eq:1.8}
(1+\tau^2 u^2, v)\cdot \mathbf{n}_{\textrm{w}}=0 \qquad \mbox{ on $\Gamma_{\textrm{w}}$}.
\end{eqnarray}
In addition, we impose the initial data on $\Sigma_{0}$ as
\begin{eqnarray}\label{eq:1.9}
(\rho, u, v)=(\rho_{0}, \textrm{u}_{0}, v_{0})(y)\qquad\,\, \mbox{on $\Sigma_{0}$},
\end{eqnarray}
where $\rho_{0}, \textrm{u}_{0}$, and $v_{0}$ satisfy \eqref{eq:1.7}.

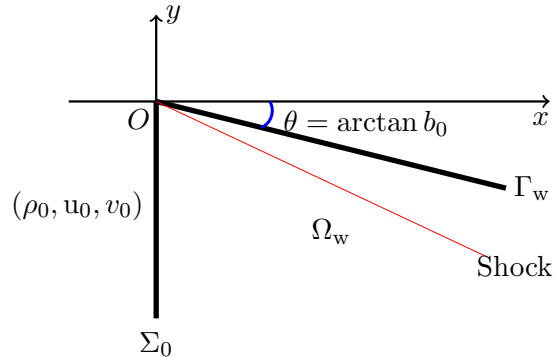
\begin{figure}[ht]
\begin{center}
\begin{tikzpicture}[scale=1.15]
\draw [line width=0.03cm][->] (-2,0) --(3.5,0);
\draw [line width=0.03cm][->] (-1,-2.5) --(-1,1);
\draw [line width=0.07cm](-1,0) --(3,-1);
\draw [line width=0.07cm](-1,0) --(-1,-2.5);
\draw [line width=0.01cm][red] (-1,0) --(2.8,-1.8);
\draw [line width=0.04cm][blue] (0.3,0)to [out=-60, in=20](0.2,-0.3);
\node at (3.4,-0.2) {$x$};
\node at (-0.8,1) {$y$};
\node at (-1.2,-0.2) {$O$};
\node at (3.3,-1) {$\Gamma_{\textrm{w}}$};
\node at (3.1,-1.9) {$\textrm{Shock}$};
\node at (-1.0,-2.8) {$\Sigma_{0}$};
\node at (1.4,-0.25) {$\theta=\arctan b_{0}$};
\node at (1.0,-1.5) {$\Omega_{\textrm{w}}$};
\node at (-1.9,-1.2) {$(\rho_{0}, \textrm{u}_{0},v_{0})$};
\end{tikzpicture}
\end{center}
\caption{The initial-boundary value problem \eqref{eq:1.6}--\eqref{eq:1.9}}\label{fig1.3}
\end{figure}

Mathematically, the hypersonic similarity means that,
for a fixed hypersonic similarity parameter $a_{\infty}$,
the structure of the solution of \eqref{eq:1.6}--\eqref{eq:1.9} is persistent if {$\textrm{M}_{\infty}$} is large (\emph{i.e.}, $\tau$ is small).
In practice, when {$\textrm{M}_{\infty}$} is sufficiently large, $\gamma$ is expected to be near $1$.
{Therefore}, if the hypersonic similarity is {valid}, when $\tau$ and $\gamma-1$ are sufficiently small, the solution of
the initial-boundary value problem \eqref{eq:1.6}--\eqref{eq:1.9} should be approximated
by the problem via taking {$\tau=0$} and $\gamma=1$:
\begin{equation}\label{eq:1.10}
\begin{cases}
\partial_{x}\rho+\partial_{y}(\rho v)=0,\\[3pt]
\partial_{x}v-\partial_{y}u=0,\\[3pt]
u+\frac{1}{2}v^2+\frac{\ln \rho}{a^2_{\infty}}=0,
\end{cases}
\end{equation}
with the boundary condition:
\begin{eqnarray}\label{eq:1.11}
v=b_0 \qquad \mbox{on $\Gamma_{\textrm{w}}$},
\end{eqnarray}
and the initial data:
\begin{eqnarray}\label{eq:1.12}
(\rho, u,v)=(\rho_{0}, u_{0},v_{0})(y) \qquad\,\, \mbox{on $\Sigma_{0}$},
\end{eqnarray}
where $\rho_0$ and $(u_{0}, v_{0})$ satisfy $\eqref{eq:1.10}_{3}$.

System \eqref{eq:1.10} is called the hypersonic small-disturbance system.
The hypersonic similarity was established in  \cite{Kuang-Xiang-Zhang-1}
by proving the existence of global entropy solutions
of problem \eqref{eq:1.6}--\eqref{eq:1.9} with large data,
provided that $(\gamma-1+\tau^2)(T.V.\{(\rho_{0},v_{0});\Sigma_0\}+|b_{0}|)<\infty$,
and then showing that the solutions converge pointwise to the solution of
problem \eqref{eq:1.10}--\eqref{eq:1.12}
as $\gamma-1+\tau^2$ tends to zero.
Therefore, a next natural question is what the convergence rate
with respect to parameter $\gamma-1+\tau^2$ should be.
The main purpose of this paper is to establish the optimal convergence rate of the solutions of
 problem \eqref{eq:1.6}--\eqref{eq:1.9} to
the solution of problem \eqref{eq:1.10}--\eqref{eq:1.12} in $L^1$
as $\gamma-1+\tau^2\rightarrow0$
with large initial data.
To this end, we set
\begin{eqnarray}\label{eq:1.13a}
\boldsymbol{\mu}=(\epsilon, \tau^2):=(\gamma-1, \tau^2).
\end{eqnarray}
{Denoted by $(\rho^{(\boldsymbol{\mu})}, u^{(\boldsymbol{\mu})}, v^{(\boldsymbol{\mu})})$}
the solution of problem \eqref{eq:1.6}--\eqref{eq:1.9},
and {denoted by} $(\rho, u,v)$ the solution of problem \eqref{eq:1.10}--\eqref{eq:1.12} (\emph{i.e.}, corresponding to the case: $\boldsymbol{\mu}=\boldsymbol{0}$).
Since the flow moves from the left to the right, then $1+\tau^2 u^{(\boldsymbol{\mu})}>0$ so that
$u^{(\boldsymbol{\mu})}$ can be solved from equation
\eqref{eq:1.7}:
\begin{eqnarray}\label{eq:1.13}
u^{(\boldsymbol{\mu})}=\frac{1}{\tau^2}\Big(\sqrt{1-\tau^2 B^{(\epsilon)}(\rho^{(\boldsymbol{\mu})},v^{(\boldsymbol{\mu})},\epsilon)}-1\Big),
\end{eqnarray}
where $B^{(\epsilon)}(\rho,v,\epsilon)$ is given by
\begin{eqnarray}\label{eq:1.14}
B^{(\epsilon)}(\rho,v,\epsilon):=\frac{2\big(\rho^{\epsilon}-1\big)}{a^{2}_{\infty}\epsilon}
+ v^2.
\end{eqnarray}

Substituting \eqref{eq:1.13} with \eqref{eq:1.14} into equations \eqref{eq:1.6},
we reformulate problem \eqref{eq:1.6}--\eqref{eq:1.9} as
\begin{eqnarray}\label{eq:1.15}
\begin{cases}
\partial_{x}\big(\rho^{(\boldsymbol{\mu})}\sqrt{1-\tau^2 B^{(\epsilon)}(\rho^{(\boldsymbol{\mu})},v^{(\boldsymbol{\mu})},\epsilon)}\big)
+\partial_{y}(\rho^{(\boldsymbol{\mu})} v^{(\boldsymbol{\mu})})=0
&\qquad \mbox{in $\Omega_{\textrm{w}}$},\\[4pt]
\partial_{x}v^{(\boldsymbol{\mu})}-\partial_{y}\Big(\frac{\sqrt{1-\tau^2 B^{(\epsilon)}(\rho^{(\boldsymbol{\mu})},v^{(\boldsymbol{\mu})},\epsilon)}-1}{\tau^2}\Big)=0
&\qquad \mbox{in $\Omega_{\textrm{w}}$},
\end{cases}
\end{eqnarray}
together with the initial condition:
\begin{eqnarray}\label{eq:1.16}
(\rho^{(\boldsymbol{\mu})}, v^{(\boldsymbol{\mu})})=(\rho_{0},v_{0})(y) \qquad \mbox{on $\Sigma_{0}$},
\end{eqnarray}
and the boundary condition:
\begin{eqnarray}\label{eq:1.17}
\Big(\sqrt{1-\tau^2 B^{(\epsilon)}(\rho^{(\boldsymbol{\mu})},v^{(\boldsymbol{\mu})},\epsilon)},
v^{(\boldsymbol{\mu})}\Big)\cdot \mathbf{n}_{\textrm{w}}=0
\qquad \mbox{on $\Gamma_{\textrm{w}}$}.
\end{eqnarray}

Similarly, from the third equation of \eqref{eq:1.10}, we obtain
\begin{equation}\label{eq:1.18}
u=-\frac{1}{2}v^2-\frac{\ln \rho}{a^{2}_{\infty}}.
\end{equation}
Then, substituting \eqref{eq:1.18} into the other two equations of \eqref{eq:1.10},
we reformulate problem \eqref{eq:1.10}--\eqref{eq:1.12} as
\begin{equation}\label{eq:1.19}
\begin{cases}
\partial_{x}\rho+\partial_{y}(\rho v)=0 &\qquad \mbox{in $\Omega_{\textrm{w}}$},\\[3pt]
\partial_{x}v+\partial_{y}\big(\frac{1}{2}v^2+\frac{\ln \rho}{a^{2}_{\infty}}\big)=0&\qquad \mbox{in $\Omega_{\textrm{w}}$},
\end{cases}
\end{equation}
with the initial condition:
\begin{equation}\label{eq:1.20}
(\rho, v)=(\rho_{0}, v_{0})(y) \qquad\,\, \mbox{on $\Sigma_{0}$},
\end{equation}
and the boundary condition:
\begin{equation}\label{eq:1.21}
v=b_0 \qquad\,\, \mbox{on $\Gamma_{\textrm{w}}$}.
\end{equation}

Our main results in this paper are stated as follows:
\begin{theorem}[Main Theorem]\label{thm:1.1}
Assume that there exist $\rho^*>\rho_*>0$ so that $\rho_0\in[\rho_{*},\rho^*]$.
Assume that $(\rho_{0}-1, v_{0})\in (L^{1}\cap BV)(\Sigma_{0})$.
Moreover, assume that there exists $C_{0}>0$ independent of $\boldsymbol{\mu}$ such that
$$
\|\boldsymbol{\mu}\|\big(T.V.\{(\rho_0, v_0); \Sigma_{0}\}+|b_0|\big)<C_{0}
$$
for $\|\boldsymbol{\mu}\|:=\epsilon+\tau^2$.
Let $(\rho^{(\boldsymbol{\mu})}, v^{(\boldsymbol{\mu})})$ and $(\rho, v)$ be the entropy solutions
of problem \eqref{eq:1.15}--\eqref{eq:1.17} and problem \eqref{eq:1.19}--\eqref{eq:1.21}, respectively.
Then there exists $\boldsymbol{\mu}_{0}=(\epsilon_0, \tau^{2}_{0})$ with $\epsilon_0=\gamma_0-1>0$
and $\tau_0>0$ such that, when $\|\boldsymbol{\mu}\|<\|\boldsymbol{\mu}_0\|:=\epsilon_0+\tau^2_0$,
\begin{eqnarray}\label{eq:1.22}
\|({\rho^{(\boldsymbol{\mu})}}-\rho, {v^{(\boldsymbol{\mu})}}-v)\|_{L^{1}((-\infty, b_{0}x))}
\leq C x\|\boldsymbol{\mu}\|
\qquad \mbox{for every $x>0$},
\end{eqnarray}
where
$C>0$ is independent on $\boldsymbol{\mu}$ and $x$.
{Moreover, the convergence rate for $\boldsymbol{\mu}$ in \eqref{eq:1.22} is optimal.}
\end{theorem}

With Theorem \ref{thm:1.1} in hand, we can further show the convergence rate between the entropy solutions
$(\rho^{(\boldsymbol{\mu})},u^{(\boldsymbol{\mu})}, v^{(\boldsymbol{\mu})})$
of problems \eqref{eq:1.6}--\eqref{eq:1.9} and the entropy solution $(\rho, u,v)$ of problem \eqref{eq:1.10}--\eqref{eq:1.12} below.

\begin{theorem}\label{coro:1.1}
Under the assumptions in {\rm Theorem \ref{thm:1.1}},
let $(\rho^{(\boldsymbol{\mu})},u^{(\boldsymbol{\mu})},v^{(\boldsymbol{\mu})})$ and $(\rho, u,v)$
be the entropy solutions of problem \eqref{eq:1.6}--\eqref{eq:1.9}
and problem \eqref{eq:1.10}--\eqref{eq:1.12}, respectively.
Let {$\boldsymbol{\mu}_{0}$} be defined in {\rm Theorem \ref{thm:1.1}}.
Then, for any $\|\boldsymbol{\mu}\|<\|\boldsymbol{\mu}_0\|$,
the following optimal convergence rate holds{\rm :}
\begin{eqnarray}\label{eq:1.23}
\|(\rho^{(\boldsymbol{\mu})}-\rho,u^{(\boldsymbol{\mu})}-u, v^{(\boldsymbol{\mu})}-v)\|_{L^{1}((-\infty, b_{0}x))}\leq C (1+ x)\|\boldsymbol{\mu}\|
\qquad\,\, \mbox{for every $x>0$},
\end{eqnarray}
where $C>0$ is a constant independent of $\boldsymbol{\mu}$ and $x$.
\end{theorem}

To complete the proof, our main strategy is to further develop the methods used
in \cite{Chen-Christoforou-Zhang-1, Chen-Christoforou-Zhang-2, Chen-Xiang-Zhang} for the Cauchy problem
into the initial-boundary value problem with large data by requiring that one of the two problems can generate
a Standard Riemann Semigroup, denoted by {\emph{SRS}}, while the other admits approximate solutions
constructed by the wave-front tracking scheme.

Since there is no theory on the \emph{SRS} for the initial-boundary value problem
in general, we identify an affine transformation to transfer the initial-boundary value
problem \eqref{eq:1.19}--\eqref{eq:1.21} to be in a quarter region with the unchanged equations \eqref{eq:1.19}.
Then we can apply the results in \cite{Colombo-Risebro} to show that the transformed problem admits a unique \emph{SRS}.
After that, by applying the inverse transformation, we can establish the $L^{1}$-stability
and the existence of the \emph{SRS} of the initial-boundary value problem \eqref{eq:1.19}--\eqref{eq:1.21}.
Moreover, a new semigroup formula is also established.

On the other hand, for the initial-boundary value problem \eqref{eq:1.15}--\eqref{eq:1.17} with large data,
by employing the path decomposition technique developed in \cite{Asakura} and following the argument
in \cite{Kuang-Xiang-Zhang-1}, we can also construct the approximate solutions via the wave-front tracking scheme.
Based on them, employing the new semigroup formula obtained in this paper,
we establish the global $L^1$--difference estimate between two approximate solutions
and obtain estimate \eqref{eq:1.22} by taking the corresponding limits.
{Finally, we construct a simple example to illustrate that the convergence rate obtained
in Theorem \ref{thm:1.1} is optimal.}

We remark that, recently, the law of the hypersonic similarity without a convergence rate was rigorously justified
for the steady potential flow past a straight wedge with large data in \cite{Kuang-Xiang-Zhang-1}
and the optimal convergence rate for {\it small data} was obtained in \cite{Kuang-Xiang-Zhang-2} over a Lipschitz curved wedge,
as well as for the full Euler flows with small data in \cite{Chen-Kuang-Xiang-Zhang}.
Meanwhile, a similar but different problem on the hypersonic limit was considered in \cite{Qu-Wang-Yuan, Qu-Yuan-Zhao}
as the Mach number of the upcoming flow $\textrm{M}_{\infty}$ tends to infinity
with the obstacle being fixed, for which the Radon measure valued solutions were constructed
as the limit of the solutions of the steady full Euler flows past a two-dimensional obstacle.

There are also some results on the existence of global entropy solutions with large data in $BV$ for
the one-dimensional gas dynamics equations in Lagrange coordinates;
see \cite{Nishida-Smoller-1, Nishida-Smoller-2,Temple} for more details.
There are also some
results on the steady supersonic flow problems
that involve the structural stability of shock waves,
rarefaction waves, and contact discontinuities;
see \cite{Chen-Kuang-Zhang,Chen-Kuang-Zhang2,Chen-Li,Chen-Zhang-Zhu-1, Courant-Friedrichs, Zhang-1, Zhang-2}
and the references cited therein. See also \cite{ChenGironSchulz,ChenSlemrodWang}.

The remaining context of this paper is organized as follows:
In \S 2, we study the elementary wave curves for systems \eqref{eq:1.15}
and \eqref{eq:1.19} globally, and then compare the Riemann solvers between these two systems with a boundary.
In \S 3, we construct the approximate solutions of the initial-boundary value problem \eqref{eq:1.15}--\eqref{eq:1.17}
via the wave-front tracking scheme and then establish the $L^1$-stability estimates and the {properties}
of the Standard Riemann Semigroup (\emph{SRS}) for the initial-boundary value problem \eqref{eq:1.19}--\eqref{eq:1.21}.
Based on them, a new semigroup formula is derived. In \S 4, we complete the proof of Theorem \ref{thm:1.1}
by first establishing the local $L^1$-estimates between two approximate solutions and then applying the semigroup formula
and the properties of the approximate solutions.
Finally, we present an example to illustrate that the convergence rate obtained in Theorem \ref{thm:1.1} is optimal.
In \S 5, we complete the proof of Theorem \ref{coro:1.1}.
{In the appendix, we show some basic estimates, which are used for establishing the optimal convergence rate in \S4.3.}

\section{Riemann Solvers for Systems \eqref{eq:1.15} and \eqref{eq:1.19}}\setcounter{equation}{0}
In this section, we construct the elementary wave curves for system \eqref{eq:1.15} and system \eqref{eq:1.19}, respectively.
Then we make the comparison of the Riemann solvers with a boundary
between the initial-boundary value problems \eqref{eq:1.15}--\eqref{eq:1.17} and
\eqref{eq:1.19}--\eqref{eq:1.21}.

\subsection{Wave curves for system \eqref{eq:1.15}}
For simplicity,
we rewrite $(\rho^{(\boldsymbol{\mu})},v^{(\boldsymbol{\mu})})$ as $(\rho,v)$
and $B^{(\epsilon)}(\rho^{(\boldsymbol{\mu})},v^{(\boldsymbol{\mu})},\epsilon)$ as $B^{(\epsilon)}$.
Denote $U:=(\rho, v)^{\top}$.
Then, by direct computation, the characteristic polynomial of system \eqref{eq:1.15} is
\begin{equation}\label{eq:2.1}
\big(1-\tau^2(B^{(\epsilon)}+a^{-2}_{\infty}\rho^{\epsilon})\big)(\lambda^{(\boldsymbol{\mu})})^2
-2v\sqrt{1-\tau^2B^{(\epsilon)}}\lambda^{(\boldsymbol{\mu})}+v^2-a^{-2}_{\infty}\rho^{\epsilon}=0.
\end{equation}
It admits two roots that are the two eigenvalues of system \eqref{eq:1.15}:
\begin{equation}\label{eq:2.2}
\lambda^{(\boldsymbol{\mu})}_{j}(U,\boldsymbol{\mu})
=\frac{a^2_{\infty}v\sqrt{1-\tau^2B^{(\epsilon)}}
  +(-1)^{j}\rho^{\frac{\epsilon}{2}}\sqrt{a^2_{\infty}-\tau^2\epsilon^{-1}\big((\epsilon+2)\rho^{\epsilon}-2\big)}}
{a^{2}_{\infty}\big(1-\tau^2(B^{(\epsilon)}+a^{-2}_{\infty}\rho^{\epsilon})\big)} \qquad\,\, \mbox{for $j=1,2$}.
\end{equation}
The corresponding right-eigenvectors are
\begin{equation}\label{eq:2.3}
{\boldsymbol{\rm{r}}}^{(\boldsymbol{\mu})}_{j}(U, \boldsymbol{\mu})
=(-1)^{j}(\rho, \, \frac{\rho^{\epsilon}}{a^2_{\infty}(\sqrt{1-\tau^2B^{(\epsilon)}}\lambda^{(\boldsymbol{\mu})}_{j}-v)})^{\top}
\qquad \mbox{for $j=1,2$}.
\end{equation}

\begin{lemma}\label{lem:2.1}
For any $U\in D=\{(\rho,v): \rho\in (\rho_{*}, \rho^{*}),\ |v|<K \}$ with constants $\rho^{*}>\rho_{*}>0$ and $K>0$
independent of $\boldsymbol{\mu}$, then
\begin{align}
&\lambda^{(\boldsymbol{\mu})}_{j}(U,\boldsymbol{\mu})\Big|_{\boldsymbol{\mu}=\boldsymbol{0}}
=v+ (-1)^{j}a_{\infty}^{-1} \qquad\,\, &&\mbox{for $j=1,2$}, \label{eq:2.4}\\[1pt]
&{\boldsymbol{\rm{r}}}^{(\boldsymbol{\mu})}_{j}(U, \boldsymbol{\mu})\Big|_{\boldsymbol{\mu}=\boldsymbol{0}}
=((-1)^{j}\rho,\, {a}_{\infty}^{-1})^{\top}
\qquad &&\mbox{for $j=1,2$}. \label{eq:2.5}
\end{align}
\end{lemma}

Since $a_{\infty}>0$, we deduce from Lemma \ref{lem:2.1} that system \eqref{eq:1.15} is strictly hyperbolic
for any $U\in D$ if $\epsilon>0$
and $\tau>0$ are sufficiently small.
Moreover, we have

\begin{lemma}\label{lem:2.2}
For any $U\in D$ with $D$ defined in {\rm Lemma \ref{lem:2.1}},
there exists a constant vector $\bar{\boldsymbol{\mu}}_0=(\bar{\epsilon}_0,\bar{\tau}^2_0)$
with $\bar{\epsilon}_0>0$ and $\bar{\tau}_0>0$ such that, for $\|\boldsymbol{\mu}\|\leq \|\bar{\boldsymbol{\mu}}_0\|$,
\begin{eqnarray}\label{eq:2.6}
\nabla_{U}\lambda^{(\boldsymbol{\mu})}_{j}(U,\boldsymbol{\mu})\cdot {\boldsymbol{\rm{r}}}^{(\boldsymbol{\mu})}_{j}(U,\boldsymbol{\mu})>0
\qquad \mbox{for $j=1,2$},
\end{eqnarray}
where $\|\boldsymbol{\mu}\|=\epsilon+\tau^2$ and  $\|\boldsymbol{\bar{\mu}}_0\|=\bar{\epsilon}_0+\bar{\tau}^2_0$.
\end{lemma}

\begin{proof}
For $j=1$, taking the derivatives on both sides of \eqref{eq:2.1} with respect to $\rho$ to obtain
\begin{eqnarray}\label{eq:2.7}
&&
2\Big(\big(1-\tau^2(B^{(\epsilon)}+a^{-2}_{\infty}\rho^{\epsilon})\big)\lambda^{(\boldsymbol{\mu})}_1-v\sqrt{1-\tau^2B^{(\epsilon)}}\Big)
\frac{\partial \lambda^{(\boldsymbol{\mu})}_1}{\partial \rho}\nonumber\\[3pt]
&&\,\,+a^{-2}_{\infty}\big((2v-(\epsilon+2)\lambda^{(\boldsymbol{\mu})}_1)\lambda^{(\boldsymbol{\mu})}_1\tau^2-\epsilon\big)
\rho^{\epsilon-1}=0.
\end{eqnarray}
Substituting \eqref{eq:2.2} for $j=1$ into \eqref{eq:2.7}, we deduce
\begin{eqnarray}\label{eq:2.8}
\frac{\partial \lambda^{(\boldsymbol{\mu})}_1}{\partial \rho}
=\frac{\big(2v-(\epsilon+2)\lambda^{(\boldsymbol{\mu})}_1\big)\lambda^{(\boldsymbol{\mu})}_1\tau^2-\epsilon}
{2\rho^{1-\frac{\epsilon}{2}}\sqrt{a^2_{\infty}-\tau^2\epsilon^{-1}\big((\epsilon+2)\rho^{\epsilon}-2\big)}}.
\end{eqnarray}
Similarly, we also obtain
\begin{eqnarray}\label{eq:2.9}
\frac{\partial \lambda^{(\boldsymbol{\mu})}_1}{\partial v}
=-\frac{a^2_{\infty}\Big(\big(\sqrt{1-\tau^2B^{(\epsilon)}}-2\tau^2v^2\big)\lambda^{(\boldsymbol{\mu})}_1-v
+\tau^2v(\lambda^{(\boldsymbol{\mu})}_1)^2\Big)}
{2\rho^{\frac{\epsilon}{2}}\sqrt{a^2_{\infty}-\tau^2\epsilon^{-1}\big((\epsilon+2)\rho^{\epsilon}-2\big)}}.
\end{eqnarray}

It follows from \eqref{eq:2.3} and  \eqref{eq:2.8}--\eqref{eq:2.9} that
\begin{small}
\begin{align*}
&\quad \nabla_{U}\lambda^{(\boldsymbol{\mu})}_{1}(U,\boldsymbol{\mu})\cdot {\textbf{r}}^{(\boldsymbol{\mu})}_{1}(U,\boldsymbol{\mu})\\[3pt]
&=\frac{2\big((\sqrt{1-\tau^2B^{(\epsilon)}}-2\tau^2v)\lambda^{(\boldsymbol{\mu})}_1-v+\tau^2 v(\lambda^{(\boldsymbol{\mu})}_{1})^2\big)
  -\big(\sqrt{1-\tau^2B^{(\epsilon)}}\lambda^{(\boldsymbol{\mu})}_1-v\big)
\big((2v-(\epsilon+2)\lambda^{(\boldsymbol{\mu})}_{1})\lambda^{(\boldsymbol{\mu})}_{1}\tau^2-\epsilon\big)}
{2\rho^{-\frac{\epsilon}{2}}\sqrt{a^2_{\infty}-\tau^2\epsilon^{-1}\big((\epsilon+2)\rho^{\epsilon}-2\big)}
\big(\sqrt{1-\tau^2B^{(\epsilon)}}\lambda^{(\boldsymbol{\mu})}_1-v\big)},
\end{align*}
\end{small}
which, by Lemma \ref{lem:2.1}, implies that
\begin{align*}
\nabla_{U}\lambda^{(\boldsymbol{\mu})}_{1}(U,\boldsymbol{\mu})\cdot {\textbf{r}}^{(\boldsymbol{\mu})}_{1}(U,\boldsymbol{\mu})
\Big|_{\boldsymbol{\mu}=\boldsymbol{0}}=a_{\infty}^{-1}>0.
\end{align*}
\smallskip

In the same way, for $j=2$, we have
\begin{small}
\begin{align*}
\begin{split}
&\nabla_{U}\lambda^{(\boldsymbol{\mu})}_{2}(U,\boldsymbol{\mu})\cdot {\textbf{r}}^{(\boldsymbol{\mu})}_{2}(U,\boldsymbol{\mu})\\[3pt]
&=\frac{2\big((\sqrt{1-\tau^2B^{(\epsilon)}}-2\tau^2v)\lambda^{(\boldsymbol{\mu})}_2-v+\tau^2 v(\lambda^{(\boldsymbol{\mu})}_{2})^2\big)
  -(\sqrt{1-\tau^2B^{(\epsilon)}}\lambda^{(\boldsymbol{\mu})}_2-v)
 \big((2v-(\epsilon+2)\lambda^{(\boldsymbol{\mu})}_{2})\lambda^{(\boldsymbol{\mu})}_{2}\tau^2-\epsilon\big)}
{2\rho^{-\frac{\epsilon}{2}}\sqrt{a^2_{\infty}-\tau^2\epsilon^{-1}((\epsilon+2)\rho^{\epsilon}-2)}
(\sqrt{1-\tau^2B^{(\epsilon)}}\lambda^{(\boldsymbol{\mu})}_2-v)}.
\end{split}
\end{align*}
\end{small}
Then, by Lemma \ref{lem:2.1} again, we obtain
\begin{align*}
\nabla_{U}\lambda^{(\boldsymbol{\mu})}_{2}(U,\boldsymbol{\mu})\cdot {\textbf{r}}^{(\boldsymbol{\mu})}_{2}(U,\boldsymbol{\mu})
\Big|_{\boldsymbol{\mu}=\boldsymbol{0}}=a_{\infty}^{-1}>0.
\end{align*}

Therefore, we can choose $\boldsymbol{\bar{\mu}}_0=(\bar{\epsilon}_0, \bar{\tau}^2_0)$
with small $\bar{\epsilon}_0>0$ and $\bar{\tau}_0>0$ such that, when
$\|\boldsymbol{\mu}\|\leq \|\boldsymbol{\bar{\mu}}_0\|$,
$\nabla_{U}\lambda^{(\boldsymbol{\mu})}_{2}(U,\boldsymbol{\mu})\cdot {\textbf{r}}^{(\boldsymbol{\mu})}_{2}(U,\boldsymbol{\mu})>0$ for $j=1,2$
and $U\in D$. This completes the proof.
\end{proof}

Lemma \ref{lem:2.2} implies that both characteristic fields of system \eqref{eq:1.15} are genuinely nonlinear.
Thus, the elementary waves are either shock waves $S^{(\boldsymbol{\mu})}=S^{(\boldsymbol{\mu})}_1\cup S^{(\boldsymbol{\mu})}_2$
or rarefaction waves $R^{(\boldsymbol{\mu})}=R^{(\boldsymbol{\mu})}_1\cup R^{(\boldsymbol{\mu})}_2$.
Next, we study the shock wave curves and rarefaction wave curves of system \eqref{eq:1.15} in the $(\rho,v)$--plane.

For a given left-state $U_L=(\rho_L,v_L)^{\top}$,
the shock solutions $U=(\rho,v)$ are the Riemann solutions satisfying the following Rankine-Hugoniot conditions
on the shock with shock speed
$\sigma_{j}(\boldsymbol{\mu})$:
\begin{equation}\label{eq:2.10}
\begin{cases}
\rho v-\rho_L v_L=\sigma^{(\boldsymbol{\mu})}_{j}
\Big(\rho\sqrt{1-\tau^2B^{(\epsilon)}(\rho, v, \epsilon)}-\rho_L\sqrt{1-\tau^2B^{(\epsilon)}(\rho_L, v_L, \epsilon)}\Big), \\[6pt]
\frac{1}{\tau^2}\Big(\sqrt{1-\tau^2B^{(\epsilon)}(\rho, v, \epsilon)}-\sqrt{1-\tau^2B^{(\epsilon)}(\rho_L, v_L, \epsilon)}\Big)
=\sigma^{(\boldsymbol{\mu})}_{j}(v-v_L),
\end{cases}
\end{equation}
and the following Lax geometry entropy conditions:
\begin{equation}\label{eq:2.11}
\lambda^{(\boldsymbol{\mu})}_{1}(U,\boldsymbol{\mu})<\sigma^{(\boldsymbol{\mu})}_{1}<\lambda^{(\boldsymbol{\mu})}_{1}(U_{L}, \boldsymbol{\mu}),
\quad\,\,
or\quad\,\,   \lambda^{(\boldsymbol{\mu})}_{2}(U,\boldsymbol{\mu})<\sigma^{(\boldsymbol{\mu})}_{2}<\lambda_{2}(U,\boldsymbol{\mu}).
\end{equation}
Then, for sufficiently small $\|\boldsymbol{\mu}\|$, if $U_L$ and $U\in D$, conditions \eqref{eq:2.11} imply that
\begin{equation}\label{eq:2.12}
\rho>\rho_{L},\,  v<v_{L}, \quad\,\,\,   \mbox{or} \quad\,\,\, \rho<\rho_{L},\,  v<v_{L}.
\end{equation}
Therefore, it follows from \eqref{eq:2.10} that
\begin{align}\label{eq:2.13}
(v-v_{L})^2&=\frac{2(\rho^{\epsilon}-\rho_{L})(\rho-\rho_L)}{a^2_{\infty}\epsilon (\rho+\rho_L)}
+\tau^2B^{(\epsilon)}(\rho_L,v_L, \epsilon)B^{(\epsilon)}(\rho,v, \epsilon)\nonumber\\[3pt]
&\quad +\Big(\sqrt{(1-\tau^2B^{(\epsilon)}(\rho_L,v_L, \epsilon))(1-\tau^2B^{(\epsilon)}(\rho,v, \epsilon))}-1\Big)\nonumber\\[3pt]
&\qquad\,\,\times\Big(v_L v+\frac{2(\rho^{\epsilon}-1)\rho+2(\rho^{\epsilon}_{L}-1)\rho_L}{a^{2}_{\infty}\epsilon(\rho+\rho_L)}\Big).
\end{align}

Set $\alpha:=\frac{\rho}{\rho_{L}}$.
Define 
\begin{align}\label{eq:2.14}
&H^{(\boldsymbol{\mu})}_{S}(v-v_L, \alpha, U_{L}, \boldsymbol{\mu})\nonumber\\[3pt]
&=(v-v_{L})^2-\frac{2\rho^{\epsilon}_{L}(\alpha^{\epsilon}-1)(\alpha-1)}{a^2_{\infty}\epsilon (\alpha+1)}
  -\tau^2B^{(\epsilon)}(\rho_L,v_L, \epsilon)B^{(\epsilon)}(\rho_L\alpha,v, \epsilon)\nonumber\\[3pt]
&\quad -\Big(\sqrt{(1-\tau^2B^{(\epsilon)}(\rho_L,v_L, \epsilon))(1-\tau^2B^{(\epsilon)}(\rho_{L}\alpha, v, \epsilon))}-1\Big)\nonumber\\[3pt]
&\qquad\,\,\times\Big(v_L v+\frac{2\big(\rho^{\epsilon}_{L}(\alpha^{\epsilon+1}+1)-\alpha-1\big)}{a^{2}_{\infty}\epsilon(\alpha+1)}\Big).
\end{align}
Thus, solving $v-v_L$ from equation \eqref{eq:2.13} is equivalent to solving the equation:
\begin{eqnarray}\label{eq:2.14b}
H^{(\boldsymbol{\mu})}_{S}(v-v_L, \alpha, U_{L}, \boldsymbol{\mu})=0,
\end{eqnarray}
where $H^{(\boldsymbol{\mu})}$ is defined by \eqref{eq:2.14}. Its solvability is given by the following lemma:

\begin{lemma}\label{lem:2.3}
Let $D$ be defined in {\rm Lemma \ref{lem:2.1}}. There exist both a constant $\delta_{0}\in(0,\frac{1}{2})$
and a constant vector $\boldsymbol{\bar{\mu}}'_{0}=(\bar{\epsilon}'_0,\bar{\tau}'^{2}_{0})$
with $\bar{\epsilon}'_0<\bar{\epsilon}_0$ and $\bar{\tau}'_0<\bar{\tau}_0$ such that,
for $\|\boldsymbol{\mu}\|\leq \|\boldsymbol{\bar{\mu}}'_0\|$,

\begin{enumerate}
\item[\rm (i)]
If $\alpha\in[1, \delta^{-1}_{0}]$, then equation $H^{(\boldsymbol{\mu})}(v-v_L, \alpha, U_{L}, \boldsymbol{\mu})=0$ admits a unique solution
$v-v_{L}=\varphi^{(\boldsymbol{\mu})}_{S_{1}}(\alpha; U_{L},\boldsymbol{\mu})
\in C^{2}\big([1, \delta^{-1}_{0}]\times \bar{D}\times(0,\bar{\epsilon}'_0)\times(0,\bar{\tau}'^{2}_0)\big)$ satisfying
\begin{eqnarray}\label{eq:2.15}
\varphi^{(\boldsymbol{\mu})}_{S_{1}}\Big|_{\boldsymbol{\mu}=\boldsymbol{0}}=-\frac{\sqrt{2}}{a_{\infty}}\sqrt{\frac{(\alpha-1)\ln\alpha}{\alpha+1}},
\qquad\,\,
\frac{\partial\varphi^{(\boldsymbol{\mu})}_{S_{1}}}{\partial \alpha}<0;
\end{eqnarray}

\item[\rm (ii)]
If $\alpha\in[\delta_{0}, 1)$, then equation $H^{(\boldsymbol{\mu})}(v-v_L, \alpha, U_{L}, \boldsymbol{\mu})=0$ admits a unique solution
$v-v_{L}=\varphi^{(\boldsymbol{\mu})}_{S_{2}}(\alpha; U_{L},\boldsymbol{\mu})\in C^{2}\big([\delta_{0}, 1)\times \bar{D}\times(0,\bar{\epsilon}'_0)\times(0,\bar{\tau}'^{2}_0)\big)$ satisfying
\begin{eqnarray}\label{eq:2.16}
\varphi^{(\boldsymbol{\mu})}_{S_{2}}\Big|_{\boldsymbol{\mu}=\boldsymbol{0}}=-\frac{\sqrt{2}}{a_{\infty}}\sqrt{-\frac{(1-\alpha)\ln\alpha}{1+\alpha}},
\qquad\,\,
\frac{\partial\varphi^{(\boldsymbol{\mu})}_{S_{2}}}{\partial \alpha}>0.
\end{eqnarray}
\end{enumerate}
\end{lemma}

\begin{proof}
If $\alpha=1$, then $\rho=\rho_{L}$.

\smallskip
For $\alpha\neq1$, it follows from  \eqref{eq:2.12}--\eqref{eq:2.14} and Lemma \ref{lem:2.1} that,
when $\boldsymbol{\mu}=\boldsymbol{0}$,
\begin{eqnarray}\label{eq:2.18}
v-v_{L}=-\frac{\sqrt{2}}{a_{\infty}}\sqrt{\frac{(\rho-\rho_{L})(\rho-\rho_{L})}{(\rho+\rho_{L})}}=-\frac{\sqrt{2}}{a_{\infty}}
\sqrt{\frac{(\alpha-1)\ln\alpha}{(\alpha+1)}}\qquad \mbox{for $\alpha>1$},
\end{eqnarray}
or
\begin{eqnarray}\label{eq:2.19}
v-v_{L}=-\frac{\sqrt{2}}{a_{\infty}}
\sqrt{-\frac{(1-\alpha)\ln\alpha}{(1+\alpha)}}\qquad \mbox{for $0<\alpha<1$}.
\end{eqnarray}

Now, we first consider the case that $\alpha>1$. It follows from \eqref{eq:2.18} that
\begin{eqnarray}
&&H^{(\boldsymbol{\mu})}_{S}(-\frac{\sqrt{2}}{a_{\infty}}
\sqrt{\frac{(\alpha-1)\ln\alpha}{(\alpha+1)}},\alpha, U_{L}, \boldsymbol{0})=0,\label{eq:2.20}\\
&& \frac{\partial H^{(\boldsymbol{\mu})}_{S}}{\partial(v-v_{L})}\bigg|_{\boldsymbol{\mu}=\boldsymbol{0}}=2(v-v_{L})=-\frac{2\sqrt{2}}{a_{\infty}}
\sqrt{\frac{(\alpha-1)\ln\alpha}{(\alpha+1)}}.\label{eq:2.21}
\end{eqnarray}

Next, we set
\begin{eqnarray*}
\tilde{H}^{(\boldsymbol{\mu})}_{S}(v-v_L, \alpha, U_{L}, \boldsymbol{\mu})
:=\frac{H^{(\boldsymbol{\mu})}(v-v_L, \alpha, U_{L}, \boldsymbol{\mu})}{\alpha-1} \qquad \mbox{for $\alpha>1$}.
\end{eqnarray*}
Then, by \eqref{eq:2.20}, 
\begin{eqnarray*}
\tilde{H}^{(\boldsymbol{\mu})}_{S}(-\frac{2\sqrt{2}}{a_{\infty}}
\sqrt{\frac{(\alpha-1)\ln\alpha}{(\alpha+1)}}, \alpha, U_{L}, \boldsymbol{0})=0,
\end{eqnarray*}
and, by \eqref{eq:2.21},
\begin{eqnarray*}
\frac{\partial\tilde{H}^{(\boldsymbol{\mu})}_{S}(v-v_L, \alpha, U_{L}, \boldsymbol{\mu})}{\partial (v-v_{L})}\bigg|_{\boldsymbol{\mu}
=\boldsymbol{0}}=-\frac{2\sqrt{2}}{a_{\infty}}\sqrt{\frac{\ln \alpha}{\alpha^2-1}}.
\end{eqnarray*}

Since $\lim_{\alpha\rightarrow 1+}\frac{\ln \alpha}{\alpha-1}=1$, we have
\begin{eqnarray*}
\lim_{\alpha\rightarrow 1+}\frac{\partial\tilde{H}^{(\boldsymbol{\mu})}_{S}(v-v_L, \alpha, U_{L}, \boldsymbol{\mu})}{\partial (v-v_{L})}\bigg|_{\boldsymbol{\mu}
=\boldsymbol{0}}=-\frac{2}{a_{\infty}}<0.
\end{eqnarray*}

Note that $\frac{\ln \alpha}{\alpha^2-1}$ is monotonically decreasing on $[1,\infty)$ and
$\lim_{\alpha\rightarrow\infty}\frac{\ln \alpha}{\alpha^2-1}=0$.
Then we can choose a small constant $\delta_{0}\in (0,\frac{1}{2})$ and a
constant $C_{\delta_0}\in(0,\frac{\sqrt{2}}{2})$ such that, for $\alpha\in [1,\delta^{-1}_0)$,
\begin{eqnarray*}
-\frac{2}{a_{\infty}}\leq\frac{\partial \tilde{H}^{(\boldsymbol{\mu})}}{\partial(v-v_{L})}\bigg|_{\boldsymbol{\mu}=\boldsymbol{0}}
<-\frac{2\sqrt{2}}{a_{\infty}}C_{\delta_0}<0.
\end{eqnarray*}

Therefore, by the implicit function theorem and the compactness of $\bar{D}$,
we deduce that there exists a constant $\boldsymbol{\mu}'_{0}=(\bar{\epsilon}'_{0}, \bar{\tau}'^2_{0})$
with $\bar{\epsilon}'_{0}<\bar{\epsilon}_{0}$ and $\bar{\tau}'_{0}<\bar{\tau}_{0}$
such that
there exists a unique solution:
$$
v-v_{L}=\varphi^{(\boldsymbol{\mu})}_{S_{1}}(\alpha; U_{L},\boldsymbol{\mu})
\in C^{1}\big([1,\delta^{-1}_{0})\times \bar{D}\times (0,\bar{\epsilon}'_{0})\times(0,\bar{\tau}'^2_{0})\big)
$$
so that $H^{(\boldsymbol{\mu})}(\varphi^{(\boldsymbol{\mu})}_{S_{1}},\alpha,U_{L},\boldsymbol{\mu} )=0$.
Moreover, it follows from \eqref{eq:2.18} that the first identity in \eqref{eq:2.15} holds.

Taking the derivative with respect to $\alpha$ on both sides of equation \eqref{eq:2.14b} to obtain
\begin{eqnarray}\label{eq:2.22}
\frac{\partial H^{(\boldsymbol{\mu})}_{S}}{\partial \alpha}
+\frac{\partial H^{(\boldsymbol{\mu})}_{S}}{\partial \varphi^{(\boldsymbol{\mu})}_{S_1}}
\frac{\partial\varphi^{(\boldsymbol{\mu})}_{S_1}}{\partial\alpha}=0.
\end{eqnarray}
Taking $\boldsymbol{\mu}=\boldsymbol{0}$ in \eqref{eq:2.22}, we see that
\begin{align}\label{eq:2.23}
\frac{\partial\varphi^{(\boldsymbol{\mu})}_{S_1}}{\partial\alpha}\big|_{\boldsymbol{\mu}=\boldsymbol{0}}
&=-\frac{\partial H^{(\boldsymbol{\mu})}_{S}}{\partial \alpha}\bigg|_{\boldsymbol{\mu}=\boldsymbol{0}}
\bigg(\frac{\partial H^{(\boldsymbol{\mu})}_{S}}{\partial \varphi^{(\boldsymbol{\mu})}_{S_1}}\bigg|_{\boldsymbol{\mu}=\boldsymbol{0}}\bigg)^{-1}\nonumber\\[1mm]
&=-\frac{\sqrt{2}}{2a_{\infty}}\frac{\alpha^2+2\alpha\ln\alpha-1}{\alpha(\alpha+1)^{\frac{3}{2}}\sqrt{(\alpha-1)\ln\alpha}}<0
\qquad\mbox{for $\,\alpha>1$}.
\end{align}
Moreover,
we have
\begin{eqnarray*}
\lim_{\alpha\rightarrow1+}\frac{\partial \varphi^{(\boldsymbol{\mu})}_{S_1}}{\partial \alpha}=-a_{\infty}^{-1},
\qquad \lim_{\alpha\rightarrow \infty}\frac{\partial \varphi^{(\boldsymbol{\mu})}_{S_1}}{\partial \alpha}=0.
\end{eqnarray*}

Thus, by choosing $\delta_0>0$, $\bar{\epsilon}'_{0}>0$, and $\bar{\tau}'_{0}>0$ sufficiently small, it follows that
\begin{eqnarray*}
\frac{\partial \varphi^{(\boldsymbol{\mu})}_{S_1}}{\partial \alpha}<0 \qquad\,\, \mbox{for $\alpha\in [1,\delta^{-1}_{0})$,
$\|\boldsymbol{\mu}\| \leq \|\boldsymbol{\bar{\mu}}'_{0}\|$, and $U_{L}\in \bar{D}$}.
\end{eqnarray*}

Furthermore, we take the derivative with respect to $\alpha$ on both sides of \eqref{eq:2.22}
and set $\boldsymbol{\mu}=\boldsymbol{0}$ to obtain
\begin{eqnarray*}
\frac{\partial^2 \varphi^{(\boldsymbol{\mu})}_{S_1}}{\partial \alpha^2}\bigg|_{\boldsymbol{\mu}=\boldsymbol{0}}
=-\bigg(\frac{\partial^{2}H^{(\boldsymbol{\mu})}_{S}}{\partial\alpha^2 }
+2\frac{\partial^{2}H^{(\boldsymbol{\mu})}_{S}}{\partial\varphi^{(\boldsymbol{\mu})}_{S_1}\partial\alpha }
\frac{\partial\varphi^{(\boldsymbol{\mu})}_{S_1}}{\partial \alpha}
  +\frac{\partial^{2}H^{(\boldsymbol{\mu})}_{S}}{\partial^{2}\varphi^{(\boldsymbol{\mu})}_{S_1} }
   \Big(\frac{\partial\varphi^{(\boldsymbol{\mu})}_{S_1}}{\partial \alpha}\Big)^2\bigg)\bigg|_{\boldsymbol{\mu}=\boldsymbol{0}}\,
\bigg(\frac{\partial H^{(\boldsymbol{\mu})}_{S}}{\partial \varphi^{(\boldsymbol{\mu})}_{S_1}}\bigg|_{\boldsymbol{\mu}=\boldsymbol{0}}\bigg)^{-1}.
\end{eqnarray*}
By direct calculation, we have
\begin{eqnarray*}
\frac{\partial^{2}H^{(\boldsymbol{\mu})}_{S}}{\partial\alpha^2 }\bigg|_{\boldsymbol{\mu}=\boldsymbol{0}}
=\frac{2\big((\alpha^2-4\alpha-1)(\alpha+1)+4\alpha^2\ln\alpha\big)}{a^2_{\infty}\alpha^2(\alpha+1)^3},
\quad \frac{\partial^{2}H^{(\boldsymbol{\mu})}_{S}}{\partial\varphi^{(\boldsymbol{\mu})}_{S_1}\partial\alpha }\bigg|_{\boldsymbol{\mu}=\boldsymbol{0}}=0,\quad \frac{\partial^{2}H^{(\boldsymbol{\mu})}_{S}}{\partial^{2}\varphi^{(\boldsymbol{\mu})}_{S_1} }\bigg|_{\boldsymbol{\mu}=\boldsymbol{0}}=2.
\end{eqnarray*}
Then
\begin{eqnarray*}
\frac{\partial^2 \varphi^{(\boldsymbol{\mu})}_{S_1}}{\partial \alpha^2}\bigg|_{\boldsymbol{\mu}=\boldsymbol{0}}
=\frac{2(\alpha-1)\big((\alpha^2-4\alpha+1)(\alpha+1)+4\alpha^2\ln \alpha\big)\ln \alpha
+(\alpha^2+2\alpha\ln\alpha-1)^{2}}{2\sqrt{2}a_{\infty}\alpha^2(\alpha+1)^{\frac{5}{2}}\big((\alpha-1)\ln\alpha\big)^{\frac{3}{2}}}.
\end{eqnarray*}

Note that $\lim_{\alpha\rightarrow1+}\frac{\partial^2 \varphi^{(\boldsymbol{\mu})}_{S_1}}{\partial \alpha^2}\Big|_{\boldsymbol{\mu}=\boldsymbol{0}}=a_{\infty}^{-1}$.
Thus, for $\alpha\in[1,\delta^{-1}_{0})$
and $\|\boldsymbol{\mu}\|\leq \|\boldsymbol{\bar{\mu}}_{0}\|$,
we see that
$\frac{\partial \varphi^{(\boldsymbol{\mu})}_{S_1}}{\partial \alpha^2}\in C\big([1,\delta^{-1}_{0})\times\bar{D}\times(0,\bar{\epsilon}'_0)\times(0,\bar{\tau}'^{2}_{0})\big)$.
This completes the proof of \rm(i), \emph{i.e.}, for the case that $\alpha\geq1$.
In the same way, by \eqref{eq:2.19}, we can also prove (ii),
\emph{i.e.}, for the case that $\alpha\leq1$.
\end{proof}

Now, we study the rarefaction wave curves with $U_{L}=(\rho_{L},v_{L})^{\top}$ as the left-state.
If these curves are parameterized as $U(\alpha)=(\rho(\alpha),v(\alpha))^{\top}$,
then the $1$-rarefaction wave satisfies
\begin{equation}\label{eq:2.24}
\begin{cases}
\frac{{\rm d}\rho}{\rho}=\frac{{\rm d}\alpha}{\alpha}, \\[3pt]
{\rm d}v=\frac{\rho^{\epsilon-1}{\rm d}\alpha}{a^{2}_{\infty}\big(\sqrt{1-\tau^2B^{(\epsilon)}(\rho(\alpha),v(\alpha),\epsilon)}
\lambda^{(\boldsymbol{\mu})}_{1}(U(\alpha),\boldsymbol{\mu})-v(\alpha)\big)},
\end{cases}
\qquad\mbox{when }\alpha\in(0,1],
\end{equation}
with $(\rho,u)\big|_{\alpha=1}=(\rho_{L},v_{L})$, or the $2$-rarefaction wave satisfies
\begin{equation}\label{eq:2.25}
\begin{cases}
\frac{{\rm d}\rho}{\rho}=\frac{{\rm d}\alpha}{\alpha}, \\[3pt]
{\rm d}v=\frac{\rho^{\epsilon-1}{\rm d}\alpha}
{a^{2}_{\infty}\big(\sqrt{1-\tau^2B^{(\epsilon)}(\rho(\alpha),v(\alpha),\epsilon)}\lambda^{(\boldsymbol{\mu})}_{2}(U(\alpha),\boldsymbol{\mu})-v(\alpha)\big)},
\end{cases}
\qquad\mbox{when }\alpha\in[1,\infty),
\end{equation}
with $(\rho,u)\big|_{\alpha=1}=(\rho_{L},v_{L})$.
Then we have the following lemma.

\begin{lemma}\label{lem:2.4}
Let $D$ be given as in {\rm Lemma \ref{lem:2.1}}.
Then there exists a constant vector $\boldsymbol{\bar{\mu}}''_{0}=(\bar{\epsilon}''_{0}, \bar{\tau}''^{2}_{0})$
with $\bar{\epsilon}''_{0}<\bar{\epsilon}_{0}$
and $\bar{\tau}''_{0}<\bar{\tau}_{0}$ such that, for $\|\boldsymbol{\mu}\|\leq \|\boldsymbol{\bar{\mu}}''_{0}\|$,

\begin{enumerate}
\item[(i)] when $\alpha\in(0,1]$, equation \eqref{eq:2.24} admits a unique solution
$v-v_{L}=\varphi^{(\boldsymbol{\mu})}_{R_1}(\alpha, U_{L},\boldsymbol{\mu})\in C^2\big((0,1]\times \bar{D}
\times (0,\bar{\epsilon}''_{0})\times (0, \bar{\tau}''^{2}_{0})\big)$ satisfying
\begin{eqnarray}\label{eq:2.26}
\varphi^{(\boldsymbol{\mu})}_{R_1}\Big|_{\alpha=1}=0,\qquad
\varphi^{(\boldsymbol{\mu})}_{R_1}\Big|_{\boldsymbol{\mu}=\boldsymbol{0}}=-\frac{1}{a_{\infty}}\ln \alpha, \qquad
\frac{\partial \varphi^{(\boldsymbol{\mu})}_{R_1}}{\partial \alpha}\leq 0;
\end{eqnarray}

\item[(ii)] when $\alpha\in[1,\infty)$, equation \eqref{eq:2.25} admits a unique solution
$v-v_{L}=\varphi^{(\boldsymbol{\mu})}_{R_2}(\alpha, U_{L},\boldsymbol{\mu})\in C^2\big([1,\infty)\times \bar{D}
\times (0,\bar{\epsilon}''_{0})\times (0, \bar{\tau}''^{2}_{0})\big)$ satisfying
\begin{eqnarray}\label{eq:2.27}
\varphi^{(\boldsymbol{\mu})}_{R_2}\Big|_{\alpha=1}=0,\qquad
\varphi^{(\boldsymbol{\mu})}_{R_1}\Big|_{\boldsymbol{\mu}=\boldsymbol{0}}=\frac{1}{a_{\infty}}\ln \alpha,
\frac{\partial \varphi^{(\boldsymbol{\mu})}_{R_1}}{\partial \alpha}\geq 0.
\end{eqnarray}
\end{enumerate}
\end{lemma}

\begin{proof}
We give the proof of \rm{(i)} only, since the argument for $\alpha\in[1,\infty)$ is the same.
First, by equation $\eqref{eq:2.24}_{1}$, we see that $\rho=\rho_{L}\alpha$.
Then we substitute it into equation $\eqref{eq:2.24}_{2}$ and integrate the resulted equation to derive
\begin{eqnarray}\label{eq:2.28}
\int^{v}_{v_{L}}\Big(\sqrt{1-\tau^{2}B^{(\epsilon)}(\rho_{L}\alpha, \zeta, \epsilon)}
\,\lambda^{(\boldsymbol{\mu})}_{1}(U(\zeta),\boldsymbol{\mu} )-\zeta\Big)
{\rm d}\zeta=\frac{\rho_L^{\epsilon}(\alpha^{\epsilon}-1)}{a^{2}_{\infty}\epsilon},
\end{eqnarray}
where $U(\zeta)=(\rho_{L}\alpha, \zeta)^{\top}$.
It follows from \eqref{eq:2.28} that $v=v_{L}$ when $\alpha=1$.

When $\alpha\in(0,1)$, set
\begin{eqnarray*}
H^{(\boldsymbol{\mu})}_{R_{1}}(v-v_{L},\alpha,\rho_{L}, \boldsymbol{\mu})
:=\int^{v}_{v_{L}}\Big(\sqrt{1-\tau^{2}B^{(\epsilon)}(\rho_{L}\alpha, \zeta, \epsilon)}
\,\lambda^{(\boldsymbol{\mu})}_{1}(U(\zeta),\boldsymbol{\mu} )-\zeta\Big){\rm d}\zeta
-\frac{\rho_L^{\epsilon}(\alpha^{\epsilon}-1)}{a^{2}_{\infty}\epsilon}.
\end{eqnarray*}
Therefore, solving equation \eqref{eq:2.28} is equivalent to solving the following equation:
\begin{eqnarray}\label{eq:2.29}
H^{(\boldsymbol{\mu})}_{R_{1}}(v-v_{L},\alpha,\rho_{L}, \boldsymbol{\mu})=0.
\end{eqnarray}

Notice that $H^{(\boldsymbol{\mu})}_{R_{1}}\in C^{2}$,
$H^{(\boldsymbol{\mu})}_{R_{1}}(-\frac{\ln \alpha}{a_{\infty}},\alpha,\rho_{L}, \boldsymbol{0})=0$, and
\begin{eqnarray}\label{eq:2.30}
\frac{\partial H^{(\boldsymbol{\mu})}_{R_{1}}(v-v_{L},\alpha,\rho_{L}, \boldsymbol{\mu}) }{\partial(v-v_{L})}\bigg|_{\boldsymbol{\mu}=\boldsymbol{0}}
=-a_{\infty}^{-1}<0.
\end{eqnarray}
Then, by the implicit function theorem, equation \eqref{eq:2.29} has a unique solution
$v-v_{L}=\varphi^{(\boldsymbol{\mu})}_{R_{1}}(\alpha, \rho_{L}, \boldsymbol{\mu})\in C^{2}$
satisfying $\varphi^{(\boldsymbol{\mu})}_{R_{1}}(1, \rho_{L}, \boldsymbol{\mu})=0$ 
and $\varphi^{(\boldsymbol{\mu})}_{R_{1}}\Big|_{\boldsymbol{\mu}=\boldsymbol{0}}=-\frac{\ln \alpha}{a_{\infty}}$.

To find $\frac{\partial \varphi^{(\boldsymbol{\mu})}_{R_{1}}}{\partial\alpha}$,
we take the derivatives with respect to $\alpha$ on both sides of \eqref{eq:2.29} and then set $\boldsymbol{\mu}=\boldsymbol{0}$
to obtain
\begin{eqnarray}\label{eq:2.31}
\frac{\partial H^{(\boldsymbol{\mu})}}{\partial \alpha}\bigg|_{\boldsymbol{\mu}=\boldsymbol{0}}
+\frac{\partial H^{(\boldsymbol{\mu})}}{\partial\varphi^{(\boldsymbol{\mu})}_{R_{1}}}\bigg|_{\boldsymbol{\mu}=\boldsymbol{0}}
\frac{\partial \varphi^{(\boldsymbol{\mu})}_{R_{1}}}{\partial \alpha}\bigg|_{\boldsymbol{\mu}=\boldsymbol{0}}=0.
\end{eqnarray}
Inserting the identity that
$\frac{\partial H^{(\boldsymbol{\mu})}}{\partial \alpha}\bigg|_{\boldsymbol{\mu}=\boldsymbol{0}}=-\frac{1}{a^2_{\infty}\alpha}$
into \eqref{eq:2.31} and using \eqref{eq:2.30},
we obtain
\begin{eqnarray*}
\frac{\partial \varphi^{(\boldsymbol{\mu})}_{R_{1}}}{\partial \alpha}\bigg|_{\boldsymbol{\mu}=\boldsymbol{0}}
=-\frac{\partial H^{(\boldsymbol{\mu})}}{\partial \alpha}\bigg|_{\boldsymbol{\mu}=\boldsymbol{0}}
\bigg(\frac{\partial H^{(\boldsymbol{\mu})}}{\partial\varphi^{(\boldsymbol{\mu})}_{R_{1}}}\bigg|_{\boldsymbol{\mu}=\boldsymbol{0}}\bigg)^{-1}
=-\frac{1}{a_{\infty}\alpha}<0.
\end{eqnarray*}
This completes the proof.
\end{proof}

For $\delta_{0}\in (0,\frac{1}{2})$ and $U_{L}\in \bar{D}$, we set
\begin{eqnarray}\label{eq:2.32}
\varphi^{(\boldsymbol{\mu})}_{1}(\alpha; U_{L},\boldsymbol{\mu})=
\begin{cases}
\varphi^{(\boldsymbol{\mu})}_{S_{1}}(\alpha; U_{L},\boldsymbol{\mu}) \qquad \mbox{for  $\alpha\in[1,\delta^{-1}_{0})$}, \\[3pt]
\varphi^{(\boldsymbol{\mu})}_{R_{1}}(\alpha; U_{L},\boldsymbol{\mu}) \qquad \mbox{for $\alpha\in(0,1]$},
\end{cases}
\end{eqnarray}
and
\begin{eqnarray}\label{eq:2.33}
\varphi^{(\boldsymbol{\mu})}_{2}(\alpha; U_{L},\boldsymbol{\mu})=
\begin{cases}
\varphi^{(\boldsymbol{\mu})}_{S_{2}}(\alpha; U_{L},\boldsymbol{\mu})\qquad \mbox{for $\alpha\in(\delta_{0},1]$}, \\[3pt]
\varphi^{(\boldsymbol{\mu})}_{R_{2}}(\alpha; U_{L},\boldsymbol{\mu}) \qquad \mbox{for $\alpha\in[1,\infty)$},
\end{cases}
\end{eqnarray}
where $\varphi^{(\boldsymbol{\mu})}_{S_{j}}$ and $\varphi^{(\boldsymbol{\mu})}_{R_{j}}$, $j=1,2$, are given by Lemma \ref{lem:2.3} and Lemma \ref{lem:2.4},
 respectively.

\smallskip
Using Lemmas \ref{lem:2.3}--\ref{lem:2.4}, we have
\begin{lemma}\label{lem:2.5}
Let $D$ be given in {\rm Lemma \ref{lem:2.1}} and $\delta_{0}\in(0,\frac{1}{2})$.
Let $\boldsymbol{\bar{\mu}}'_{0}$ be given in {\rm Lemma \ref{lem:2.3}},
and let $\boldsymbol{\bar{\mu}}''_{0}$ be given in {\rm Lemma \ref{lem:2.4}}.
Then, for $\|\boldsymbol{\mu}\|\leq \min\{\|\boldsymbol{\bar{\mu}}'_{0}\|,\|\boldsymbol{\bar{\mu}}''_{0}\| \}$
and $U_{L}\in \bar{D}$, the following statements hold{\rm :}

\begin{enumerate}
\item[(i)] $\varphi^{(\boldsymbol{\mu})}_{1}\in C^{2}\big((0,\delta^{-1}_{0})\times \bar{D}\times (0,\min\{\|\boldsymbol{\bar{\mu}}'_{0}\|,\|\boldsymbol{\bar{\mu}}''_{0}\| \})\big)$
satisfies that $\varphi^{(\boldsymbol{\mu})}_{1}\Big|_{\alpha=1}=0$ and $\frac{\partial\varphi^{(\boldsymbol{\mu})}_{1} }{\partial \alpha}\leq 0$ for $\alpha\in(0,\delta^{-1}_{0})$, and
\begin{eqnarray}\label{eq:2.34}
\varphi^{(\boldsymbol{\mu})}_{1}\Big|_{\boldsymbol{\mu}=\boldsymbol{0}}=
\begin{cases}
-\frac{\ln \alpha}{a_{\infty}} &\qquad \mbox{for $\alpha\in(0,1]$},\\[3pt]
-\frac{\sqrt{2}}{a_{\infty}}\sqrt{\frac{(\alpha-1)\ln \alpha}{\alpha+1}} &\qquad \mbox{for $\alpha\in[1,\delta^{-1}_{0})$};
\end{cases}
\end{eqnarray}

\item[(ii)] $\varphi^{(\boldsymbol{\mu})}_{2}\in C^{2}\big((\delta_{0},\infty)\times \bar{D}\times (0,\min\{\|\boldsymbol{\bar{\mu}}'_{0}\|,\|\boldsymbol{\bar{\mu}}''_{0}\| \})\big)$
satisfies that $\varphi^{(\boldsymbol{\mu})}_{2}\Big|_{\alpha=1}=0$ and $\frac{\partial\varphi^{(\boldsymbol{\mu})}_{2} }{\partial \alpha}\geq 0$ for $\alpha\in(\delta_{0},\infty)$, and
\begin{eqnarray}\label{eq:2.35}
\varphi^{(\boldsymbol{\mu})}_{2}\Big|_{\boldsymbol{\mu}=\boldsymbol{0}}=
\begin{cases}
-\frac{\sqrt{2}}{a_{\infty}}\sqrt{-\frac{(1-\alpha)\ln \alpha}{1+\alpha}}&\qquad \mbox{for $\alpha\in(\delta_{0},1]$}, \\[3pt]
\frac{\ln \alpha}{a_{\infty}} &\qquad \mbox{for $\alpha\in[1,\infty)$}.
\end{cases}
\end{eqnarray}
\end{enumerate}
\end{lemma}

Based on Lemma \ref{lem:2.5}, we define
\begin{eqnarray}
&&\Phi^{(\boldsymbol{\mu})}_{1}(\alpha_1; U_{L}, \boldsymbol{\mu})=\big(\rho_{L}\alpha_{1}, v_{L}+\varphi^{(\boldsymbol{\mu})}_{1}(\alpha_1; U_{L},\boldsymbol{\mu})\big)
\qquad \mbox{for $\alpha_1\in (0,\delta^{-1}_{0})$},\label{eq:2.36}\\[3pt]
&&\Phi^{(\boldsymbol{\mu})}_{2}(\alpha_2; U_{L}, \boldsymbol{\mu})=\big(\rho_{L}\alpha_{2}, v_{L}+\varphi^{(\boldsymbol{\mu})}_{2}(\alpha_2; U_{L},\boldsymbol{\mu})\big)
\qquad \mbox{for $\alpha_2\in (\delta_{0},\infty)$}.\label{eq:2.37}
\end{eqnarray}
Denote
\begin{eqnarray}\label{eq:2.38}
\begin{split}
\Phi^{(\boldsymbol{\mu})}(\boldsymbol{\alpha}; U_{L}, \boldsymbol{\mu}):&=\Phi^{(\boldsymbol{\mu})}_{2}\big(\alpha_2;\Phi^{(\boldsymbol{\mu})}_{1}(\alpha_1; U_{L},\boldsymbol{\mu}),\boldsymbol{\mu}\big)\\[3pt]
&=(\rho_{L}\alpha_{2}\alpha_{1}, v_{L}+\varphi^{(\boldsymbol{\mu})}_{1}(\alpha_1; U_{L},\boldsymbol{\mu})+\varphi^{(\boldsymbol{\mu})}_{2}(\alpha_2;U^{(\boldsymbol{\mu})}_{M},\boldsymbol{\mu}))^{\top},
\end{split}
\end{eqnarray}
where $\boldsymbol{\alpha}=(\alpha_1,\alpha_2)$ and
$U^{(\boldsymbol{\mu})}_{M}=(\rho_{L}\alpha_1, v_{L}+\varphi^{(\boldsymbol{\mu})}_{1}(\alpha_1; U_{L}, \boldsymbol{\mu}))^{\top}$.

Next, we consider the elementary wave curves of system \eqref{eq:1.19} for $U:=(\rho, v)^{\top}$.
The eigenvalues of system \eqref{eq:1.19} are
\begin{eqnarray}\label{eq:2.39}
\lambda_{1}(U)=v-a_{\infty}^{-1}, \qquad \lambda_{2}(U)=v+a_{\infty}^{-1},
\end{eqnarray}
and the corresponding two right-eigenvectors are
\begin{eqnarray}\label{eq:2.40}
\textbf{r}_{1}(U)=(-\rho, a_{\infty}^{-1})^{\top},\qquad \textbf{r}_{2}(U)=(\rho, a_{\infty}^{-1})^{\top}.
\end{eqnarray}

Notice that, for any $U\in D$, by {\rm Lemma \ref{lem:2.1}} and \eqref{eq:2.39}--\eqref{eq:2.40}, we know that
\begin{eqnarray}\label{eq:2.41}
\lambda^{\boldsymbol{\mu}}_{j}(U,\boldsymbol{\mu})\Big|_{\boldsymbol{\mu}=\boldsymbol{0}}=\lambda_{j}(U),\quad
\boldsymbol{\rm{r}}^{(\boldsymbol{\mu})}_{j}(U,\boldsymbol{\mu})\Big|_{\boldsymbol{\mu}=\boldsymbol{0}}=\boldsymbol{\rm{r}}_{j}(U)
\qquad\,\, \mbox{for $j=1,2$}.
\end{eqnarray}

For $U\in D$, system \eqref{eq:1.19} is strictly hyperbolic. Moreover, a direct computation shows that
\begin{eqnarray*}
\nabla_{U}\lambda_{j}(U)\cdot\textbf{r}_{j}(U)=a_{\infty}^{-1}>0 \qquad\,\, \mbox{for $j=1,2$}.
\end{eqnarray*}
It implies that the two characteristics families are genuinely nonlinear in $D$.
Then,
for any two constant states $U, U_{L}\in D$,
we can parameterize the first  elementary wave curve (including 1-shock {${S}_1$} and 1-rarefaction wave {${R}_1$})
and the second elementary wave curve
(including 2-shock {${S}_2$} and 2-rarefaction wave {${R}_2$}) of system \eqref{eq:1.19} which connects $U_{L}$ to $U$ as
\begin{eqnarray}
&&U=\Phi_{1}(\alpha_1; U_{L})=(\rho_{L}\alpha_1, v_{L}+\varphi_{1}(\alpha_1))^{\top},\label{eq:2.42}\\[3pt]
&&U=\Phi_{2}(\alpha_2; U_{L})=(\rho_{L}\alpha_2, v_{L}+\varphi_{2}(\alpha_2))^{\top},\label{eq:2.43}
\end{eqnarray}
respectively, where
\begin{eqnarray}
&&\varphi_{1}(\alpha_1)=
\begin{cases}
-\frac{\ln \alpha_1}{a_{\infty}} &\qquad \mbox{for $\alpha_1\in(0,1]$}, \\[3pt]
-\frac{\sqrt{2}}{a_{\infty}}\sqrt{\frac{(\alpha_1-1)\ln \alpha_1}{\alpha_1+1}} &\qquad \mbox{for $\alpha_1\in[1,\infty)$},
\end{cases}\label{eq:2.44}\\[4pt]
&&\varphi_{2}(\alpha_2)=
\begin{cases}
-\frac{\sqrt{2}}{a_{\infty}}\sqrt{-\frac{(1-\alpha_2)\ln \alpha_2}{1+\alpha_2}} &\quad\,\,\, \mbox{for $\alpha_2\in(0,1]$}, \\[3pt]
\frac{\ln \alpha_2}{a_{\infty}} &\quad\,\,\, \mbox{for $\alpha_2\in[1,\infty)$}.\label{eq:2.45}
\end{cases}
\end{eqnarray}
Finally, we set
\begin{eqnarray}\label{eq:2.46}
\Phi(\boldsymbol{\alpha}; U_{L})=(\rho_{L}\alpha_2\alpha_1, v_{L}+\varphi_{1}(\alpha_1)+\varphi_{2}(\alpha_2))^{\top}
\qquad \mbox{for $\boldsymbol{\alpha}=(\alpha_1,\alpha_2)$}.
\end{eqnarray}

Then, by direct computation, we have

\begin{lemma}\label{lem:2.6}
$\varphi_{k}(\alpha)$, $k=1,2$, satisfy
\begin{enumerate}
\item[(i)]  $\varphi_{k}(\alpha)\in C^2(\mathbb{R}_{+})$ for $k=1,2;$

\smallskip
\item[(ii)] $\varphi'_{1}(\alpha)<0$ and $\varphi'_{2}(\alpha)>0$ for $\alpha\in\mathbb{R}_{+};$

\smallskip
\item[(iii)] $\varphi_{k}(1)=0$ and $\varphi'_{k}(1)=(-1)^{k}a_{\infty}^{-1}$ for $k=1, 2;$

\smallskip
\item[(iv)]
For any $U_{L}\in D$,
\begin{eqnarray}\label{eq:2.47}
\varphi^{(\boldsymbol{\mu})}_{j}\Big|_{ \boldsymbol{\mu}= \boldsymbol{0}}=\varphi_{j},\quad
\frac{\partial \varphi^{(\boldsymbol{\mu})}_{j}}{\partial \alpha}\bigg|_{ \boldsymbol{\mu}= \boldsymbol{0}}
= \varphi'_{j}(\alpha) \qquad\,\, \mbox{for $j=1,2$}.
\end{eqnarray}
\end{enumerate}
\end{lemma}

\subsection{Comparison of the Riemann solvers between systems \eqref{eq:1.15} and \eqref{eq:1.19}}
In this subsection, we consider the comparison of the Riemann solvers between system \eqref{eq:1.15} and \eqref{eq:1.19}
with/without a boundary.

First, we are concerned with the Riemann problem
for system \eqref{eq:1.15} with $(\tilde{x},\tilde{y})\in \Omega$ and the following data:
\begin{eqnarray}\label{eq:2.48}
U(\tilde{x},y)=
\begin{cases}
U_{R}:=(\rho_{R}, v_{R})^{\top} \quad & \mbox{for $y>\tilde{y}$},\\[3pt]
U_{L}:=(\rho_{L}, v_{L})^{\top} \quad & \mbox{for $y<\tilde{y}$}.
\end{cases}
\end{eqnarray}

\begin{lemma}\label{lem:2.7}
Let domain $D$ be given as in {\rm Lemma \ref{lem:2.1}}.
For any two given constant states $U_{L}, U_{R}\in D$,
there exist small constants $\delta'_{0}\in (\delta_{0}, \frac{1}{2})$ and $\boldsymbol{\bar{\mu}}'_{1}=(\bar{\epsilon}'_{1}, \bar{\tau}'^{2}_{1})$
with $\bar{\epsilon}'_{1}<\min\{\bar{\epsilon}'_{0},\bar{\epsilon}''_{0} \}$ and $\bar{\tau}'_{1}<\min\{\bar{\tau}'_{0}, \bar{\tau}''_{0}\}$ such that,
for $\|\boldsymbol{\mu}\|\leq \|\boldsymbol{\bar{\mu}}'_{1}\|$, the Riemann problem \eqref{eq:1.15} and \eqref{eq:2.48}
admits a unique constant state $U^{(\boldsymbol{\bar{\mu}})}_{M}$ satisfying
\begin{eqnarray}\label{eq:2.49}
U^{(\boldsymbol{\mu})}_{M}=\Phi^{(\boldsymbol{\mu})}_{1}(\alpha_1; U_{L}, \boldsymbol{\mu}), \qquad
U_{R}=\Phi^{(\boldsymbol{\mu})}_{2}(\alpha_2; U^{(\boldsymbol{\mu})}_{M}, \boldsymbol{\mu}),
\end{eqnarray}
where $\alpha_1, \alpha_2\in(\delta'_{0}, \frac{1}{\delta'_{0}})$.
\end{lemma}
\begin{proof}
To obtain the solution of the Riemann problem \eqref{eq:1.15} and \eqref{eq:2.48}, it suffices
to solve the following equations for $\boldsymbol{\alpha}=(\alpha_1, \alpha_2)$:
\begin{eqnarray*}
U_{R}=\Phi^{(\boldsymbol{\mu})}(\boldsymbol{\alpha}; U_{L}, \boldsymbol{\mu}).
\end{eqnarray*}

More precisely, by \eqref{eq:2.38}, it can be rewritten as
\begin{eqnarray}\label{eq:2.50}
\begin{cases}
\rho_{R}=:\Phi^{(\boldsymbol{\mu}),1}(\boldsymbol{\alpha}; U_{L}, \boldsymbol{\mu})=\rho_{L}\alpha_2\alpha_1,\\[3pt]
v_{R}=:\Phi^{(\boldsymbol{\mu}),2}(\boldsymbol{\alpha}; U_{L}, \boldsymbol{\mu})=v_{L}+\varphi^{(\boldsymbol{\mu})}_{1}(\alpha_1; U_{L},\boldsymbol{\mu})+\varphi^{(\boldsymbol{\mu})}_{2}\big(\alpha_2;\Phi^{(\boldsymbol{\mu})}_{1}(\alpha_1; U_{L}, \boldsymbol{\mu}),\boldsymbol{\mu} \big).
\end{cases}
\end{eqnarray}
For $\boldsymbol{\mu}=\boldsymbol{0}$, by \eqref{eq:2.34}--\eqref{eq:2.35},
equations \eqref{eq:2.50} admits a unique solution $\boldsymbol{\alpha}=(\alpha_1,\alpha_2)$ for $U_{L}, U_{R}\in D$.
Next, by \eqref{eq:2.47},
\begin{eqnarray*}
\det\bigg(\frac{\partial (\Phi^{(\boldsymbol{\mu}),1}, \Phi^{(\boldsymbol{\mu}),2})}{\partial(\alpha_1, \alpha_2)}\Big|_{\boldsymbol{\mu}=\boldsymbol{0}}\bigg)
=\rho_{L}\big(\alpha_{2}\varphi'_{2}(\alpha_{2})-\alpha_{1}\varphi'_{1}(\alpha_{1})\big).
\end{eqnarray*}
Thus, it follows from Lemma \ref{lem:2.6} that there exists a small constant $\delta'_{0}\in (\delta_{0},\frac{1}{2})$
such that, for $\alpha_1, \alpha_2\in(\delta'_{0}, \frac{1}{\delta'_{0}})$ and $\rho_{L}\in D$,
\begin{eqnarray*}
\rho_{L}\big(\alpha_{2}\varphi'_{2}(\alpha_{2})-\alpha_{1}\varphi'_{1}(\alpha_{1})\big)>C_{\delta'_{0}}>0.
\end{eqnarray*}

Then, by the implicit function theorem, there exist constants $\boldsymbol{\bar{\mu}}'_{1}=(\bar{\epsilon}'_{1}, \bar{\tau}'^{2}_{1})$
with $\bar{\epsilon}'_{1}<\min\{\bar{\epsilon}'_{0},\bar{\epsilon}''_{0} \}$
and $\bar{\tau}'_{1}<\min\{\bar{\tau}'_{0}, \bar{\tau}''_{0}\}$ such that,
for $\|\boldsymbol{\mu}\|\leq \|\boldsymbol{\bar{\mu}}'_{1}\|=\bar{\epsilon}'_{1}+\bar{\tau}'^2_{1}$, equations \eqref{eq:2.50} admit a
unique solution $(\alpha_1, \alpha_2)\in (\delta'_{0}, \frac{1}{\delta'_{0}})^2$.
\end{proof}

We now make the comparison of the Riemann solutions between system \eqref{eq:1.15} and system \eqref{eq:1.19}
with the initial-boundary condition \eqref{eq:2.48}.
\begin{proposition}\label{prop:2.1}
For a given number $k=1,2$, assume that two constant states $U_{L}, U_{R}\in D$ satisfy
\begin{eqnarray}\label{eq:2.51}
U_{R}=\Phi(\boldsymbol{\beta}; U_{L}), \qquad U_{R}=\Phi^{(\boldsymbol{\mu})}_{k}(\alpha_{k}; U_{L}, \boldsymbol{\mu}),
\end{eqnarray}
where $\boldsymbol{\beta}=(\beta_1, \beta_2)$, $\alpha_{k}\in (\delta'_{0},\frac{1}{\delta'_{0}})$, and $\delta'_{0}>0$
is given in {\rm Lemma \ref{lem:2.7}}.
Then, for $\|\boldsymbol{\mu}\|\leq \|\boldsymbol{\bar{\mu}}'_{1}\|$,
\begin{eqnarray}\label{eq:2.52}
\beta_{k}=\alpha_{k}+O(1)|\alpha_{k}-1|\|\boldsymbol{\mu}\|,
\qquad \beta_{j}=1+O(1)|\alpha_{k}-1|\|\boldsymbol{\mu}\|,
\end{eqnarray}
where $ j\neq k$, $ j=1,2$, and $\boldsymbol{\bar{\mu}}'_{1}=(\bar{\epsilon}'_{1}, \bar{\tau}'_{1})$ is given
in {\rm Lemma \ref{lem:2.7}}.
Moreover, if $|U_{R}-U_{L}|=\alpha_{NP}$,
then
\begin{eqnarray}\label{eq:2.53}
|\beta_{1}-1|+|\beta_{2}-1|=O(1)\alpha_{NP},
\end{eqnarray}
where the bounds of $O(1)$ are independent on $\boldsymbol{\mu}$.
\end{proposition}

\begin{proof}
First, consider the following equations that are derived from \eqref{eq:2.51}:
\begin{eqnarray*}
\Phi(\boldsymbol{\beta}; U_{L})=\Phi^{(\boldsymbol{\mu})}_{k}(\alpha_{k}; U_{L}, \boldsymbol{\mu})
\qquad \mbox{for $\boldsymbol{\beta}=(\beta_1,\beta_2)$}.
\end{eqnarray*}

Without loss of the generality, we consider only the case: $k=1$.
By \eqref{eq:2.36} and \eqref{eq:2.42}--\eqref{eq:2.43},
for $U_{L}\in D$, the above equation is equivalent to the following equations:
\begin{equation}\label{2.55a}
\begin{cases}
F_{1}(\beta_1,\beta_{2}, \alpha_1, \boldsymbol{\mu}, U_{L})=0, \\[3pt]
F_{2}(\beta_1,\beta_{2}, \alpha_1, \boldsymbol{\mu}, U_{L})=0,
\end{cases}
\end{equation}
where
\begin{eqnarray*}
\begin{cases}
F_{1}(\beta_1,\beta_{2}, \alpha_1, \boldsymbol{\mu}, U_{L}):=\beta_1\beta_2-\alpha_1, \\[3pt]
F_{2}(\beta_1,\beta_{2}, \alpha_1, \boldsymbol{\mu}, U_{L}):=\varphi_{1}(\beta_1)+\varphi_{2}(\beta_{2})
-\varphi^{(\boldsymbol{\mu})}_{1}(\alpha_1; U_{L}, \boldsymbol{\mu}).
\end{cases}
\end{eqnarray*}

When $\boldsymbol{\mu}=\boldsymbol{0}$, by Lemma \ref{lem:2.6}, system \eqref{2.55a}
has a unique solution $\beta_{1}=\alpha_1$ and $\beta_{2}=1$.
When $\boldsymbol{\mu}\neq\boldsymbol{0}$, by Lemma \ref{lem:2.6},
\begin{eqnarray*}
\det\bigg(\frac{\partial(F_{1}, F_{2})}{\partial(\beta_1,\beta_2)}\bigg)\bigg|_{\boldsymbol{\mu}\neq\boldsymbol{0}, \beta_{1}=\alpha_1, \beta_{2}=1}
=\varphi'_{2}(1)-\alpha_1\varphi'_{1}(\alpha_1)>C'_{\delta'_{0}}>0
\qquad\,\,\mbox{for $\alpha_{1}\in(\delta'_{0}, \frac{1}{\delta'_{0}})$}.
\end{eqnarray*}
Thus, by the implicit function theorem, system \eqref{2.55a} admits a unique solution:
$$
(\beta_{1}, \beta_2)=(\beta_{1}(\alpha_1,\boldsymbol{\mu}), \beta_{2}(\alpha_1,\boldsymbol{\mu}))\in C^{2}.
$$
In addition, by Lemma \ref{lem:2.6},  $\beta_{1}(1,\boldsymbol{\mu})=\beta_{2}(1,\boldsymbol{\mu})=1$,
$\beta_{1}(\alpha_1,0)=\alpha_1$, and $\beta_{2}(\alpha_1,0)=1$.
Then we can apply the Taylor expansion formula to obtain
\begin{eqnarray*}
&&\beta_{1}(\alpha_1, \boldsymbol{\mu})=\beta_{1}(\alpha_1,0)+\beta_{1}(1, \boldsymbol{\mu})-\beta_{1}(1,0)+O(1)|\alpha_1-1|\|\boldsymbol{\mu}\|
=\alpha_1+O(1)|\alpha_1-1|\|\boldsymbol{\mu}\|,\\[3pt]
&&\beta_{2}(\alpha_1, \boldsymbol{\mu})=\beta_{2}(\alpha_1,0)+\beta_{2}(1, \boldsymbol{\mu})-\beta_{2}(1,0)+O(1)|\alpha_1-1|\|\boldsymbol{\mu}\|
=1+O(1)|\alpha_1-1|\|\boldsymbol{\mu}\|,
\end{eqnarray*}
which are estimates \eqref{eq:2.52}.

Since, for $U_{L}, U_{R}\in D$, there exists $C>0$, independent of $\boldsymbol{\mu}$, such that
\begin{eqnarray*}
\frac{1}{C}\sum_{j=1,2}|\beta_{j}-1|\leq |\Phi(\boldsymbol{\beta}; U_{L})-U_{L}|\leq C\sum_{j=1,2}|\beta_{j}-1|,
\end{eqnarray*}
estimate \eqref{eq:2.53} follows immediately.
\end{proof}

Following the proof of Proposition \ref{prop:2.1} above,
we have the following corollary in a direct way, whose proof is omitted.
\begin{corollary}\label{coro:2.1}
Assume that two constant states $U_{L}, U_{R}\in D$ satisfy 
\begin{eqnarray}\label{eq:2.54}
U_{R}=\Phi(\boldsymbol{\beta}; U_{L}), \quad U_{R}=\Phi^{(\boldsymbol{\mu})}(\boldsymbol{\alpha}; U_{L}, \boldsymbol{\mu}) \qquad
\mbox{ for $\boldsymbol{\beta}=(\beta_1, \beta_2)$ and $\boldsymbol{\alpha}=(\alpha_1, \alpha_2)$},
\end{eqnarray}
where $\alpha_{1}, \alpha_{2}\in (\delta'_{0}, \frac{1}{\delta'_{0}})$, and constant $\delta'_{0}>0$ is given in {\rm  Lemma \ref{lem:2.7}}.
Then, for $\|\boldsymbol{\mu}\|\leq \|\boldsymbol{\bar{\mu}}'_{1}\|$,
\begin{eqnarray}\label{eq:2.55}
\beta_{j}=\alpha_{j}+O(1)\big(\sum_{k=1,2}|\alpha_{k}-1|\big)\|\boldsymbol{\mu}\|\qquad \mbox{for $j=1,2$},
\end{eqnarray}
where $\boldsymbol{\bar{\mu}}'_{1}=(\bar{\epsilon}'_{1}, \bar{\tau}'_{1})$ is given in {\rm Lemma \ref{lem:2.7}}.
Moreover, if $\tilde{U}_{R}\in D$, $|\tilde{U}_{R}-U_{R}|=\alpha_{NP}$, and
\begin{eqnarray}\label{eq:2.56}
 U_{R}=\Phi(\boldsymbol{\beta}; U_{L}), \quad \tilde{U}_{R}=\Phi^{(\boldsymbol{\mu})}(\boldsymbol{\alpha}; U_{L}, \boldsymbol{\mu}) 
 \qquad \mbox{ for $\boldsymbol{\beta}=(\beta_1, \beta_2)$ and
 $\boldsymbol{\alpha}=(\alpha_1, \alpha_2)$},
\end{eqnarray}
then
\begin{eqnarray}\label{eq:2.57}
\beta_{j}=\alpha_{j}+O(1)\big(\sum_{k=1,2}|\alpha_{k}-1|\big)\|\boldsymbol{\mu}\|+O(1)\alpha_{NP}
\qquad \mbox{for $j=1,2$},
\end{eqnarray}
where the bounds of $O(1)$ are independent of $\boldsymbol{\mu}$.
\end{corollary}

Next, we compare the Riemann solutions near the boundary with the following initial boundary value conditions:
\begin{align}\label{eq:2.58}
\begin{cases}
v^{(\boldsymbol{\mu})}_{b}=\sqrt{1-\tau^2B^{(\epsilon)}(\rho^{(\boldsymbol{\mu})}_{b}, v^{(\boldsymbol{\mu})}_{b},\epsilon)}b_{0}&\  
\mbox{on $\{x=\hat{x},\ y=\hat{y}+b_{0}(x-\hat{x})\}$},\\[3pt]
U(x,y)=U_{L} &\  \mbox{on $\{ x=\hat{x},\ y<\hat{y}\}$},
\end{cases}
\end{align}
where $U_{L}=(\rho_{L},v_{L})^{\top}$ and $b_0<0$.

\begin{lemma}\label{lem:2.8}
For any given constant state $U_{L}\in D$ with $D$ defined by {\rm Lemma \ref{lem:2.1}}, there exist small constants
$\bar{\epsilon}''_{1}\leq \min\{\bar{\epsilon}'_{0}, \bar{\epsilon}''_{0}\}$, $\bar{\tau}''_{1}\leq \min\{\bar{\tau}'_{0}, \bar{\tau}''_{0}\}$,
and  $\delta''_{0}\in (\delta_{0},\frac{1}{2})$ such that, for
$\|\boldsymbol{\mu}\|\leq \|\boldsymbol{\bar{\mu}}''_{1}\|$, the Riemann problem \eqref{eq:1.15} and \eqref{eq:2.58}
admits a unique solution $U^{(\boldsymbol{\mu})}_{b}=(\rho^{(\boldsymbol{\mu})}_{b}, v^{(\boldsymbol{\mu})}_{b})^{\top}$
connecting $U_{L}$ by the first-family wave curve with strength $\alpha_1\in(\delta''_{0},\frac{1}{\delta''_{0}})${\rm :}
\begin{eqnarray}\label{eq:2.59}
U^{(\boldsymbol{\mu})}_{b}=\Phi^{(\boldsymbol{\mu})}_{1}(\alpha_1;U_{L},\boldsymbol{\mu}),
\qquad U^{(\boldsymbol{\mu})}_{b}=(\rho^{(\boldsymbol{\mu})}_{b}, v^{(\boldsymbol{\mu})}_{b})^{\top}.
\end{eqnarray}
\end{lemma}

\begin{proof}
It suffices to show that the following equation:
\begin{eqnarray*}
v_{L}+\varphi^{(\boldsymbol{\mu})}_{1}(\alpha_1; U_{L},\boldsymbol{\mu})
=b_0\sqrt{1-\tau^2B^{(\epsilon)}(\Phi^{(\boldsymbol{\mu})}_{1}(\alpha_1;U_{L},\boldsymbol{\mu}),\epsilon)}
\end{eqnarray*}
has a unique solution $\alpha_1$ when $\|\boldsymbol{\mu}\|$ is small for $U_{L}\in D$ and $b_{0}<0$.

Let
\begin{eqnarray*}
F_{b}(\alpha_1; \boldsymbol{\mu}, b_{0}, U_{L})
=v_{L}+\varphi^{(\boldsymbol{\mu})}_{1}(\alpha_1; U_{L},\boldsymbol{\mu})
-b_0\sqrt{1-\tau^2B^{(\epsilon)}(\Phi^{(\boldsymbol{\mu})}_{1}(\alpha_1;U_{L},\boldsymbol{\mu}),\epsilon)}.
\end{eqnarray*}
When $\boldsymbol{\mu}=\boldsymbol{0}$, $F_{b}(\alpha_1; \boldsymbol{\mu}, b_{0}, U_{L})$ can be reduced as
\begin{eqnarray*}
F_{b}(\alpha_1; \boldsymbol{0}, b_{0}, U_{L})=\varphi_{1}(\alpha_1)+v_{L}-b_0.
\end{eqnarray*}

If $v_{L}\leq b_{0}$, then, by \eqref{eq:2.44},
equation $F_{b}(\alpha_1; \boldsymbol{0}, b_{0}, v_{L})=0$ has a unique solution $\alpha_1=e^{a_{\infty}(v_{L}-b_{0})}\in(0,1]$.

If $v_{L}> b_{0}$, it follows from $\lim_{\alpha_{1}\rightarrow\infty}\frac{(\alpha_1-1)\ln \alpha_1}{\alpha_1+1}=\infty$ that
$$
\lim_{\alpha_{1}\rightarrow\infty}F_{b}(\alpha_1; \boldsymbol{0}, b_{0}, U_{L})=-\infty.
$$
Moreover, for $\alpha_1\in [1,\infty)$,
\begin{eqnarray*}
F_{b}(1; \boldsymbol{0}, b_{0}, U_{L})=v_{L}-b_{0}>0
\end{eqnarray*}
and, by Lemma \ref{lem:2.6},
\begin{eqnarray*}
\frac{\partial F_{b}(\alpha_1; \boldsymbol{0}, b_{0}, U_{L})}{\partial \alpha_1}=\varphi'_{1}(\alpha_1)<0.
\end{eqnarray*}
Thus, $F_{b}(\alpha_1; \boldsymbol{0}, b_{0}, U_{L})=0$ has a unique solution $\alpha_1\in (1,\infty)$.
Therefore, if $\boldsymbol{\mu}=\boldsymbol{0}$, $F_{b}(\alpha_1; \boldsymbol{\mu}, b_{0}, U_{L})=0$ has a unique
solution $\alpha_1$ for $U_{L}\in D$ and $b_{0}<0$.

Next, notice that there exists a constant $\delta''_{0}>0$ such that, for $\alpha_{1}\in(\delta''_{0}, \frac{1}{\delta''_{0}})$,
\begin{eqnarray*}
\frac{\partial F_{b}}{\partial \alpha_1}\bigg|_{\boldsymbol{\mu}=\boldsymbol{0}}=\varphi'_{1}(\alpha_1)<-C_{\delta''_{0}}<0.
\end{eqnarray*}
Then, by the implicit function theorem, there exist small constants
$\bar{\epsilon}''_{1}\leq \min\{\bar{\epsilon}'_{0}, \bar{\epsilon}''_{0}\}$ and $\bar{\tau}''_{1}\leq \min\{\bar{\tau}'_{0}, \bar{\tau}''_{0}\}$
such that, for
$\|\boldsymbol{\mu}\|\leq \|\boldsymbol{\bar{\mu}}''_{1}\|$,
$F_{b}(\alpha_1; \boldsymbol{\mu}, b_{0}, U_{L})=0$ admits a unique solution
$\alpha_{1}\in(\delta''_{0}, \frac{1}{\delta''_{0}})$.
This completes the proof.
\end{proof}

Now we are ready to compare the Riemann solutions between system \eqref{eq:1.15} and system \eqref{eq:1.19} with a boundary.
\begin{proposition}\label{prop:2.2}
Let $U_{L}=(\rho_{L}, v_{L})^{\top}$, $U_{b}=(\rho_{b}, v_{b})^{\top}$,
and $U^{(\boldsymbol{\mu})}_{b}=(\rho^{(\boldsymbol{\mu})}_{b}, v^{(\boldsymbol{\mu})}_{b})^{\top}$ be the three
constant states in $D$ satisfying
\begin{eqnarray}\label{eq:2.60}
U_{b}=\Phi_{1}(\beta_1; U_{L}),\qquad U^{(\boldsymbol{\mu})}_{b}=\Phi^{(\boldsymbol{\mu})}_{1}(\alpha_1; U_{L}, \boldsymbol{\mu})
\quad \mbox{for $\alpha_1\in(\delta''_{0}, \frac{1}{\delta''_{0}})$},
\end{eqnarray}
and
\begin{eqnarray}\label{eq:2.61}
v_{b}=b_0, \qquad v^{(\boldsymbol{\mu})}_{b}=b_{0}\sqrt{1-\tau^2B^{(\epsilon)}(\rho^{(\boldsymbol{\mu})}_{b},v^{(\boldsymbol{\mu})}_{b},\epsilon)},
\end{eqnarray}
where $\delta''_{0}>0$ is given in {\rm Lemma \ref{lem:2.8}}.
Then, for $\|\boldsymbol{\mu}\|\leq \|\boldsymbol{\bar{\mu}}''_{1}\|$,
\begin{eqnarray}\label{eq:2.62}
\beta_{1}=\alpha_1+O(1)\big(1+|\alpha_1-1|\big)\|\boldsymbol{\mu}\|,
\end{eqnarray}
where $\bar{\epsilon}''_{1}$ and $\bar{\tau}''_{1}$ are given in {\rm Lemma \ref{lem:2.8}}
and the bound of $O(1)$ is independent of $\boldsymbol{\mu}$.
\end{proposition}

\begin{proof}
By \eqref{eq:2.36}, \eqref{eq:2.42}, and \eqref{eq:2.60}--\eqref{eq:2.61}, we have the following
relation for $\alpha_1$ and $\beta_1$:
\begin{eqnarray*}
v_{L}+\varphi^{(\boldsymbol{\mu})}_{1}(\alpha_1; U_{L}, \boldsymbol{\mu})
=\sqrt{1-\tau^2B^{(\epsilon)}(\rho^{(\boldsymbol{\mu})}_{b},v^{(\boldsymbol{\mu})}_{b},\epsilon)}\big(\varphi_{1}(\beta_1)+v_{L}\big).
\end{eqnarray*}
Let
\begin{eqnarray*}
\begin{split}
\mathcal{F}_{b}(\beta_1, \alpha_1, \boldsymbol{\mu}, U_{L})
:&=\sqrt{1-\tau^2B^{(\epsilon)}(\rho^{(\boldsymbol{\mu})}_{b},v^{(\boldsymbol{\mu})}_{b},\epsilon)}\,\varphi_{1}(\beta_1)
-\varphi^{(\boldsymbol{\mu})}_{1}(\alpha_1; U_{L},\boldsymbol{\mu})\\[3pt]
&\quad\, +\Big(\sqrt{1-\tau^2B^{(\epsilon)}(\rho^{(\boldsymbol{\mu})}_{b},v^{(\boldsymbol{\mu})}_{b},\epsilon)}-1\Big)v_{L}.
\end{split}
\end{eqnarray*}

For $\boldsymbol{\mu}=\boldsymbol{0}$, it is direct to see that
$\mathcal{F}_{b}(\beta_1, \alpha_1, \boldsymbol{0}, U_{L})=\varphi_1(\beta_1)-\varphi_1(\alpha_1)=0$ has a unique solution $\beta_1=\alpha_1$.
In addition, by Lemma \ref{lem:2.6}, for $\alpha_1\in(\delta''_{0},1)\cup(1, \frac{1}{\delta''_{0}})$,  we have
\begin{eqnarray}
\frac{\partial \mathcal{F}_{b}(\beta_1, \alpha_1, \boldsymbol{\mu}, U_{L})}{\partial \beta_1}\bigg|_{\boldsymbol{\mu}=\boldsymbol{0}, \beta_1=\alpha_1}
=-\varphi'_{1}(\alpha_1)-C_{\delta''_{0}}<0.
\end{eqnarray}
Then, by the implicit function theorem, for $\|\boldsymbol{\mu}\|\leq \|\boldsymbol{\bar{\mu}}''_{1}\|$,
there exists a unique solution $\beta_1=\beta_{1}(\alpha_1, \boldsymbol{\mu})\in C^2$ of the equation: $\mathcal{F}_{b}=0$.
Moreover, when $\alpha_1=1$, by Lemmas \ref{lem:2.5}--\ref{lem:2.6},
$\mathcal{F}_{b}(\beta_1, \alpha_1, \boldsymbol{\mu}, U_{L})=0$
can be reduced to
\begin{eqnarray*}
\sqrt{1-\tau^2B^{(\epsilon)}(\rho_{L},v_{L},\epsilon)}\,\varphi_{1}(\beta_1)
+\Big(\sqrt{1-\tau^2B^{(\epsilon)}(\rho_{L},v_{L},\epsilon)}-1\Big)v_{L}=0,
\end{eqnarray*}
so that
\begin{eqnarray*}
\beta_1(1, \boldsymbol{\mu})=1+O(1)\|\boldsymbol{\mu}\|,
\end{eqnarray*}
where the bound of $O(1)$ is independent of $\boldsymbol{\mu}$.

Finally, by the Taylor formula and the fact that $\beta_{1}(1,\boldsymbol{0})=1$, 
\begin{align*}
\beta_{1}(\alpha_1,\boldsymbol{\mu})&=\beta_{1}(\alpha_1,\boldsymbol{0})+\beta_{1}(1,\boldsymbol{\mu})-\beta_{1}(1,\boldsymbol{0})+O(1)|\alpha_1-1|\|\boldsymbol{\mu}\|\\[3pt]
&=\alpha_1+O(1)\big(1+|\alpha_1-1|\big)\|\boldsymbol{\mu}\|,
\end{align*}
where the bound of $O(1)$ is independent of $\boldsymbol{\mu}$.
\end{proof}

Based on Propositions \ref{prop:2.1}--\ref{prop:2.2}, we have
\begin{proposition}\label{prop:2.3}
Let $U_{L}=(\rho_{L}, v_{L})^{\top}$ and $U_{b}=(\rho_{b}, v_{b})^{\top}$ be two constant states in $D$ satisfying
\begin{eqnarray}
&&U_{b}=\Phi(\boldsymbol{\beta}; U_{L}),\quad U_{b}=\Phi^{(\boldsymbol{\mu})}_{1}(\alpha_1; U_{L}, \boldsymbol{\mu}) 
\qquad \mbox{ for $\boldsymbol{\beta}=(\beta_1, \beta_2)$},\label{eq:2.64}\\[3pt]
&&v_{b}=b_{0}\sqrt{1-\tau^2B^{(\epsilon)}(\rho_{b},v_{b},\epsilon)}.\label{eq:2.65}
\end{eqnarray}
Then, for $\alpha_1\in(\delta''_{0}, \frac{1}{\delta''_{0}})$ and $\|\boldsymbol{\mu}\|\leq \|\boldsymbol{\bar{\mu}}''_{1}\|$,
\begin{eqnarray}\label{eq:2.66}
\beta_{1}=\alpha_1+O(1)\big(1+|\alpha_1-1|\big)\|\boldsymbol{\mu}\|, \quad \beta_{2}=1+O(1)\big(1+|\alpha_1-1|\big)\|\boldsymbol{\mu}\|,
\end{eqnarray}
where the small constants $\delta''_{0}$, $\bar{\epsilon}''_{1}$, and $\bar{\tau}''_{1}$ are given in {\rm Lemmas \ref{lem:2.7}--\ref{lem:2.8}},
and the bound of $O(1)$ is independent of $\boldsymbol{\mu}$.
\end{proposition}

\section{Existence of Solutions of Problem \eqref{eq:1.15}--\eqref{eq:1.17} and Well-Posedness of Problem \eqref{eq:1.19}--\eqref{eq:1.21} with Large Data}
In this section, we construct the approximate solutions of the initial-boundary value problem \eqref{eq:1.15}--\eqref{eq:1.17}
and establish the well-posedness of the initial-boundary value problem \eqref{eq:1.19}--\eqref{eq:1.21} in $BV\cap L^1$,
which is the basis to establish the $L^1$-error estimate between the two respective entropy solutions of problem \eqref{eq:1.15}--\eqref{eq:1.17}
and problem \eqref{eq:1.19}--\eqref{eq:1.21}.

\subsection{Wave front-tracking scheme for  problem \eqref{eq:1.15}--\eqref{eq:1.17}}
Let $\nu\in\mathbb{N}_{+}$ be a given parameter.
As in \cite{Asakura, Bressan} (see also \cite{Dafermos2016}), for given initial data $U_{0}(y)={(\rho_{0}, v_{0})^{\top}}(y)$ with $y<0$,
we can construct a piecewise constant function $U^{\nu}_{0}(y)=(\rho^{\nu}_{0}, v^{\nu}_{0})^{\top}(y)$
such that
\begin{eqnarray}\label{eq:3.1}
\|U^{\nu}_{0}(\cdot)-U_{0}(\cdot)\|_{L^{1}({\Sigma_{0}})}\leq 2^{-\nu},
\qquad T.V.\{U^{\nu}_{0}(\cdot); {\Sigma_{0}}\}\leq T.V.\{U_{0}(\cdot); {\Sigma_{0}}\}.
\end{eqnarray}
Then the approximate solution $U^{(\boldsymbol{\mu}),\nu}(x,y)$ in $\Omega$ is constructed in the following way:

Let $y_{N}<y_{N-1}<\cdots<y_{1}<y_{0}=0$ be the location of the discontinuities of $U^{\nu}_{0}(y)$ at $x=0$.
At each point $(0,y_{k})$ for $1\leq k\leq N$, we solve the Riemann problem \eqref{eq:1.15} and \eqref{eq:2.44} with
$U_{L}=U^{\nu}_{0}(y_{k}-)$ and $U_{R}=U^{\nu}_{0}(y_{k}+)$.
At $(0,y_{0})$, we solve the Riemann problem \eqref{eq:1.15} and \eqref{eq:2.58} with $U_{L}=U^{\nu}_{0}(0-)$.
Then, by Lemmas \ref{lem:2.7}--\ref{lem:2.8}, the solutions of these two types of Riemann problem may consist of
shock waves $S^{(\boldsymbol{\mu})}$ or rarefaction waves $R^{(\boldsymbol{\mu})}$.
We further partition the rarefaction waves into several small central rarefaction fans (still denoted by) $R^{(\boldsymbol{\mu})}$
with strength less than $\nu^{-1}$, which propagate with the characteristic speeds.
Such a modified solution of the two Riemann problems is called an {\it Accurate Riemann Solver} (\emph{ARS}).
Putting all the modified solutions together,
we define an approximate solution $U^{(\boldsymbol{\mu}), \nu}(x,y)$.
It is piecewise constant and prolongs until a pair of neighbouring discontinuities
interacts at point $(\hat{x}, y)\in\Omega$ or a wave front hits boundary $\Gamma_{\textrm{w}}$ at point $(\hat{x}_1, b_{0}\hat{x}_1)$.
At this point, we continue to construct the approximate solution
by giving the \emph{ARS} of the Riemann problem \eqref{eq:1.15} and \eqref{eq:2.44} with initial data $U^{(\boldsymbol{\mu}), \nu}(\hat{x}-,y)$
or of the Riemann problem \eqref{eq:1.15} and \eqref{eq:2.58} with Riemann data $U^{(\boldsymbol{\mu}), \nu}(\hat{x}_{1}-,b_{0}\hat{x}_{1}-)$.
We repeat this construction as long as the number of the wave fronts does not tend to the infinity in a finite {\it time}.
Then, to avoid the case that the number of wave fronts blows up,
we introduce a {\it Simplified Riemann Solver} (\emph{SRS}), in which all the new waves are lumped
into a single non-physical wave $NP^{(\boldsymbol{\mu})}$
with a fixed speed $\hat{\lambda}$, which is larger than all the characteristics speeds.
To decide when the \emph{SRS} is used, we introduce a threshold parameter $\varrho>0$,
depending only on $\nu^{-1}$. When the strengths of the two approaching physical wave fronts $\alpha$ and $\beta$
satisfy that $|\alpha-1||\beta-1|>\varrho$, the
\emph{ARS} is used and, otherwise, the \emph{SRS} is used.

Moreover,
we may change some of the speeds of the wave fronts slightly with a quality less than $2^{-\nu}$,
in order to make sure that only two wave fronts interact or only one wave front hits boundary $\Gamma$ at each point.
The set of all the fronts are defined by $J(U^{(\boldsymbol{\mu})}):=S^{(\boldsymbol{\mu})}\cup R^{(\boldsymbol{\mu})}\cup NP^{(\boldsymbol{\mu})}$.
Then, applying the path decomposition position method developed in \cite{Asakura}
and following the arguments in \cite{Kuang-Xiang-Zhang-1},
we obtain the following results for the approximate solution $U^{(\boldsymbol{\mu}),\nu}(x,y)$ of
problem \eqref{eq:1.15}--\eqref{eq:1.17}.

\begin{proposition}\label{prop:3.1}
Assume that $\rho_*\leq\rho_{0}\leq\rho^{*}$ for some constants $\rho^{*}>\rho_{*}>0$.
Then there exists both a constant vector $\boldsymbol{\bar{\mu}}^{*}_{0}=(\bar{\epsilon}^{*}_{0}, (\bar{\tau}^{*}_{0})^2)$
and a constant $\bar{C}_0>0$ with $\bar{\epsilon}^{*}_{0}>0$ and $\bar{\tau}^{*}_{0}>0$ depending only on $(a_{\infty}, \rho^{*}, \rho_{*})$
such that, for $\|\boldsymbol{\mu}\|\leq \|\boldsymbol{\bar{\mu}}^{*}_{0}\|$, if $(\rho_{0}-1, v_{0})\in (L^{1}\cap BV)(\Sigma_{0})$
and
\begin{eqnarray}\label{eq:3.2}
\|\boldsymbol{\mu}\|\big(T.V.\{U_{0}(\cdot);\ \Sigma_{0}\}+|b_{0}|\big)<\bar{C}_{0},
\end{eqnarray}
the approximate solution $U^{(\boldsymbol{\mu})}(x,y)$ constructed above
can be defined for all $(x,y)\in \Omega$ and satisfies
\begin{align}
&\sup_{x>0}\|U^{(\boldsymbol{\mu}),\nu}(x,\cdot)\|_{L^{\infty}(-\infty,\,b_{0}x)}
+\sup_{x>0}T.V.\{U^{(\boldsymbol{\mu}),\nu}(x,\cdot); (-\infty, b_{0}x)\}<\bar{C}_{1},\label{eq:3.3}\\
&\|U^{(\boldsymbol{\mu}),\nu}(x_1,\cdot+b_{0}x_1)-U^{(\boldsymbol{\mu}),\nu}(x_2,\cdot+b_{0}x_2)\|_{L^{1}(-\infty, 0)}
 <\bar{C}_{2}|x_{1}-x_{2}|.\label{eq:3.4}
\end{align}
The strength of each rarefaction wave-front and the total strength of the non-physical front are small{\rm :}
\begin{eqnarray}\label{eq:3.5}
\max_{\alpha \in R^{(\boldsymbol{\mu})}}|\alpha-1|<\bar{C}_{3}\nu^{-1}, \qquad\,
\sum_{\alpha \in NP^{(\boldsymbol{\mu})}}\alpha
<\bar{C}_{4}2^{-\nu}.
\end{eqnarray}
Moreover, there exists a subsequence $\{\nu_{i}\}^{\infty}_{i=1}$ with $\nu_{i}\rightarrow\infty$ as $i\rightarrow\infty$ such that
\begin{eqnarray}\label{eq:3.6}
 U^{(\boldsymbol{\mu}),\nu_{i}}\rightarrow U^{(\boldsymbol{\mu})}\qquad\,\,\mbox{ in $L^{1}_{\rm loc}(\Omega_{\rm{w}})$},
\end{eqnarray}
and $U^{(\boldsymbol{\mu})}\in (BV_{\rm loc}\cap L^{1}_{\rm loc})(\Omega_{\rm{w}})$  is an entropy solution of
problem \eqref{eq:1.15}--\eqref{eq:1.17}.
Here the positive constants $\bar{C}_{k}$, $k=1,2,3,4$, depend only on $(a_{\infty}, \rho^{*}, \rho_{*})$, but independent of $(\boldsymbol{\mu},\nu)$.
\end{proposition}

\subsection{Well-posedness of the initial-boundary value problem \eqref{eq:1.19}--\eqref{eq:1.21}}
In this subsection, we consider the initial-boundary value problem \eqref{eq:1.19}--\eqref{eq:1.21},
construct the semigroup, and establish the existence and $L^1$-stability of the solutions.
First, we consider the following initial-boundary value problem on $\bar{\Omega}=\{(x,y)\,:\, x>0, y<0\}$:
\begin{eqnarray}\label{eq:3.7}
\begin{cases}
\partial_{x}\rho+\partial_{x}(\rho v)=0 &\qquad \mbox{in $\bar{\Omega}$}, \\[3pt]
\partial_{x}+\partial_{x}\big(\frac{1}{2}v^2+\frac{\ln \rho}{a^{2}_{\infty}}\big)=0&\qquad \mbox{in $\bar{\Omega}$},
\end{cases}
\end{eqnarray}
with initial data
\begin{eqnarray}\label{eq:3.8}
(\rho,v)=(\bar{\rho}_0, \bar{v}_0)(y) \qquad \mbox{on $\bar{\Sigma}_{0}=\{(x,y)\,:\, x=0, y<0\}$},
\end{eqnarray}
and boundary condition
\begin{eqnarray}\label{eq:3.9}
v=0 \qquad \mbox{ on ${\bar{\Gamma}}=\{(x,y)\,:\, x>0, y=0\}$}.
\end{eqnarray}

Let $U(x,y)=(\rho, v)^{\top}(x,y)$, $\bar{U}_{0}(y)=(\bar{\rho}_{0}, \bar{v}_{0})^{\top}(y)$,
and $\bar{U}_{\infty}=(\bar{\rho}_\infty, \bar{v}_{\infty})^\top$ with $\bar{\rho}_\infty>0$ and $\bar{v}_\infty>0$.
Following the results in \cite{Colombo-Risebro, Nishida}, we have

\begin{lemma}\label{lem:3.1}
Assume that $0<\rho_{*}<\bar{\rho}_{0}(y)<\rho^{*}<\infty$ and $\bar{U}_{0}(y)-\bar{U}_{\infty}\in(BV\cap L^{1})({\bar{\Sigma}_{0}})$.
Then there is a constant $\bar{C}'_{0}>0$ such that, if
\begin{eqnarray}\label{eq:3.10}
T.V.\{\bar{U}_{0}(y); {\bar{\Sigma}_{0}}\}+|\bar{v}_{0}(0-)|<\bar{C}'_{0},
\end{eqnarray}
there exist a domain $\bar{\mathcal{D}}\subseteq BV((-\infty,0))$,
an $L^{1}$-Lipschitz semigroup $\bar{\mathcal{S}}_{x}:(0,\infty)\times \bar{\mathcal{D}}\mapsto \bar{\mathcal{D}}$,
and a Lipschitz constant $\bar{L}>0$ so that

\begin{enumerate}
\item[(i)] $\bar{\mathcal{D}}$ contains the $L^{1}$-closure of the set of those
functions $U(\cdot,y):(-\infty,0)\mapsto \bar{\Omega}$ satisfying $U-\bar{U}_{\infty}\in \big(L^{1}\cap BV\big)((-\infty, 0));$

\item[(ii)] $U(x)=\bar{\mathcal{S}}_{x}(\bar{U}_{0})$ is the entropy solution of the initial-boundary value
  problem \eqref{eq:3.7}--\eqref{eq:3.9} and satisfies
\begin{align}\label{eq:3.11}
\|\bar{\mathcal{S}}_{x_1}\bar{U}_{1,0}-\bar{\mathcal{S}}_{x_2}\bar{U}_{2,0}\|_{L^{1}((-\infty,0))}
\leq \bar{L}\big(\|\bar{U}_{1,0}-\bar{U}_{2,0}\|_{L^{1}( \bar{\Sigma_{0}})}+|x_1-x_{2}|\big);
\end{align}

\item[(iii)] If $\bar{U}_{0}$ is a  piecewise constant function, then, for $x>0$ sufficiently small, $\bar{\mathcal{S}}_{x}(U_{0})$ {coincides}
 with the solution of the initial-boundary value problem
\eqref{eq:3.7}--\eqref{eq:3.9} by piecing together the Riemann solutions at the all jumps of $\bar{U}_{0}$.
\end{enumerate}
\end{lemma}

We now turn to the initial-boundary value problem \eqref{eq:1.19}--\eqref{eq:1.21}.
\begin{proposition}\label{prop:3.2}
Suppose that $0<\rho_{*}<\rho_{0}(y)<\rho^{*}<\infty$ and $U_{0}(y)-U_{\infty}\in(BV\cap L^{1})(\Sigma_{0})$ with $U_\infty=(1,0)^\top$.
Then there is a constant $\bar{C}''_{0}>0$ such that,
\begin{eqnarray}\label{eq:3.12}
T.V.\{U_{0}(y); \Sigma_{0}\}+|b_0|<\bar{C}''_{0},
\end{eqnarray}
there exist a domain $\mathcal{D}\subset BV((-\infty, b_{0}x))$, an $L^1$-Lipschitz semigroup $\mathcal{S}_x:(0,\infty)\times \mathcal{D}\mapsto \mathcal{D}$,
and a Lipschitz constant $L>0$ so that
\begin{enumerate}
\item[(i)] $\mathcal{D}$ contains the $L^1$-closure of the set of functions
$U:(-\infty, b_{0}x)\mapsto \Omega_{\textrm{w}}$ satisfying $U-U_{\infty}\in \big(L^1\cap BV\big)((-\infty, b_{0}x));$

\item[(ii)] $U(x,\cdot)=\mathcal{S}_xU_0(\cdot)$
is the entropy solution of problem \eqref{eq:1.19}--\eqref{eq:1.21} and
\begin{align}\label{eq:3.13}
\|\mathcal{S}_{x_1}(U_{1,0}(\cdot))-\mathcal{S}_{x_2}(U_{2,0}(\cdot))\|_{L^{1}((-\infty,b_{0}x))}\leq L\|U_{1,0}(\cdot)-U_{2,0}(\cdot)\|_{L^{1}(\Sigma_{0})};
\end{align}

\item[(iii)] If $U(\tilde{x},\cdot)$ is a piecewise constant function, then, for $x>\tilde{x}$ sufficiently small, $\mathcal{S}_{x}U(\tilde{x},\cdot)$ coincides
with the solution of the initial-boundary value problem \eqref{eq:1.19}--\eqref{eq:1.21} by piecing together all the Riemann solutions at the all jumps of $U(x)$.
\end{enumerate}
\end{proposition}

\begin{proof}
Let
\begin{eqnarray}\label{eq:3.14}
(\hat{x},\,\hat{y}):= (x,\,y-b_{0}x).
\end{eqnarray}
Define
\begin{eqnarray}\label{eq:3.15}
\hat{\rho}(\hat{x},\hat{y}):=\rho(\hat{x}, \hat{y}+b_{0}\hat{x}),\quad  \hat{v}(\hat{x},\hat{y}):=v(\hat{x}, \hat{y}+b_{0}\hat{x})-b_{0}.
\end{eqnarray}
Then 
$(\hat{\rho}, \hat{v})$ satisfies the initial-boundary value problem \eqref{eq:3.7}--\eqref{eq:3.9} in the $(\hat{x},\hat{y})$-plane
with the initial data:
\begin{eqnarray*}
(\hat{\rho},\hat{v})(0,\hat{y})=(\rho_{0}(\hat{y}), v_{0}(\hat{y})-b_0).
\end{eqnarray*}
Thus, by applying Lemma \ref{lem:3.1},
there exist a domain $\hat{\mathcal{D}}\subset BV((-\infty,0))$, an $L^1$-Lipschitz
semigroup $\hat{\mathcal{S}}_{\hat{x}}:[0,\infty)\times \hat{\mathcal{D}}\mapsto \hat{\mathcal{D}}$,
and a Lipschitz constant $\hat{L}>0$ so that facts \rm(i)--\rm(iii) in Lemma \ref{lem:3.1} hold.
We define the inverse transformation of \eqref{eq:3.14}--\eqref{eq:3.15}:
\begin{eqnarray*}
x=\hat{x},\quad y=\hat{y}+b_{0}\hat{x},\quad \rho(x,y)=\hat{\rho}(x, y-b_{0}x ),\quad v(x,y)=\hat{v}(x,y-b_{0}x)-b_{0}.
\end{eqnarray*}
Then substituting them into Lemma \ref{lem:3.1}, we obtain \rm(i)--\rm(iii).
This completes the proof.
\end{proof}

By Proposition \ref{prop:3.2} and \cite{Bressan}, we can also derive the following semigroup formula:
\begin{proposition}\label{prop:3.3}
Let $V(x,y):[0,\infty)\mapsto \mathcal{D}$ be a Lipschitz continuous map with a finite number of wave fronts for some $x>0$ and $V(0,y)=V_{0}(y)$.
Let $\mathcal{S}$ be a semigroup obtained by {\rm Proposition \ref{prop:3.2}}. Then
\begin{align}\label{eq:3.16}
&\|\mathcal{S}_{x}(V_{0}(\cdot))-V(x,\cdot)\|_{L^{1}(-\infty,\,b_{0}x)}\nonumber\\[3pt]
&\leq L\int^{x}_{0}\lim\inf_{h\rightarrow 0+}\frac{\|\mathcal{S}_{h}(V(s,\cdot))-V(s+h,\cdot)\|_{L^{1}(-\infty,\,b_{0}(s+h))}}{h}{\rm d}s,
\end{align}
where $L$ and $\mathcal{D}$ are given in {\rm Proposition \ref{prop:3.2}}.
\end{proposition}

\section{Proof of Theorem \ref{thm:1.1}}
In this section, we prove the main result of this paper, \emph{i.e.}, Theorem \ref{thm:1.1}.
To complete the proof of Theorem \ref{thm:1.1}, we need first to obtain the convergence rate estimate \eqref{eq:1.22} that is based on the local $L^1$-difference
between the entropy solutions $U^{(\boldsymbol{\mu})}$ and $U$, and then show that it is optimal with respect to $\boldsymbol{\mu}$ by constructing a simple example.

\subsection{Local $L^1$ error estimates between the entropy solutions $U^{(\boldsymbol{\mu})}$ and $U$}
In this subsection, we give some lemmas regarding the local $L^1$-difference between the entropy solutions $U^{(\boldsymbol{\mu})}$ and $U$
corresponding to problem \eqref{eq:1.15}--\eqref{eq:1.17} and
problem \eqref{eq:1.19}--\eqref{eq:1.21} with a boundary, respectively.

\begin{lemma}\label{lem:4.1}
Suppose that $U^{(\boldsymbol{\mu}),\nu}$ is an approximate solution of problem \eqref{eq:1.15}--\eqref{eq:1.17} constructed as in {\rm \S 3.1}
and satisfies {\rm Proposition \ref{prop:3.1}} with a front at point $(\hat{x},y_{\mathcal{I}})\in\Omega$.
For a given $k=1,2$, denote
\begin{equation}\label{eq:4.1}
U^{(\boldsymbol{\mu}),\nu}(\hat{x}+h,y)=
\begin{cases}
U_{R}=(\rho_{R},v_{R})^{\top}\qquad &\mbox{for $y>y_{\mathcal{I}}+\dot{y}_{\alpha_k}h$},\\[3pt]
U_{L}=(\rho_{L},v_{L})^{\top}\qquad &\mbox{for $y<y_{\mathcal{I}}+\dot{y}_{\alpha_k}h$},
\end{cases}
\end{equation}
where $U_{L}:=U^{(\boldsymbol{\mu}),\nu}(\hat{x}, y_{\mathcal{I}}-)$, $U_{R}:=U^{(\boldsymbol{\mu}),\nu}(\hat{x}, y_{\mathcal{I}}+)$, $h>0$,
and $|\dot{y}_{\alpha_k}|<\hat{\lambda}$.
Let $\mathcal{S}$ be a uniformly Lipschitz continuous semigroup obtained in {\rm Proposition \ref{prop:3.3}}.
If $\|\boldsymbol{\mu}\|\leq \|\boldsymbol{\bar{\mu}}^{*}_{0}\|$ with $\boldsymbol{\bar{\mu}}^{*}_{0}$
given in {\rm Proposition \ref{prop:3.1}} and $h>0$ is sufficiently small, then

\begin{enumerate}
\item[(i)] when $U_{L}$ and $U_{R}$ are connected by a $k$-th shock wave-font $\alpha_{k}\in S^{(\boldsymbol{\mu})}_{k}$
and $|\dot{y}_{\alpha_k}-\sigma^{(\boldsymbol{\mu})}_{k}(\alpha_{k})|<2^{-\nu}$ for $\sigma^{(\boldsymbol{\mu})}_{k}(\alpha_{k})$ as the speed of $k$-th shock wave,
\begin{eqnarray}\label{eq:4.2}
\|\mathcal{S}_{h}(U^{(\boldsymbol{\mu}),\nu}(\hat{x},\cdot))-U^{(\boldsymbol{\mu}),\nu}(\hat{x}+h,\cdot)\|_{L^1(y_{\mathcal{I}}-\eta, y_{\mathcal{I}}+\eta)}
\leq C_{\mathcal{I}}\big(\|\boldsymbol{\mu}\|+2^{-\nu}\big)|\alpha_{k}-1|h;
\end{eqnarray}

\item[(ii)] when $U_{L}$ and $U_{R}$ are connected by a $k$-th rarefaction front $\alpha_{k}\in R^{(\boldsymbol{\mu})}_{k}$
with $|\dot{y}_{\alpha_k}-\lambda^{(\boldsymbol{\mu})}_{k}(U_{R},\boldsymbol{\mu})|<2^{-\nu}$ for $\lambda^{(\boldsymbol{\mu})}_{k}(U_{R},\boldsymbol{\mu})$ as the speed of the rarefaction wave,
\begin{eqnarray}\label{eq:4.3}
\begin{split}
&\|\mathcal{S}_{h}(U^{(\boldsymbol{\mu}),\nu}(\hat{x},\cdot))-U^{(\boldsymbol{\mu}),\nu}(\hat{x}+h,\cdot)\|_{L^1(y_{\mathcal{I}}-\eta, y_{\mathcal{I}}+\eta)}\\[3pt]
&\leq C_{\mathcal{I}}\big(\|\boldsymbol{\mu}\|+2^{-\nu}+(1+\|\boldsymbol{\mu}\|)\nu^{-1}\big)|\alpha_{k}-1|h;
\end{split}
\end{eqnarray}

\item[(iii)] when $U_{L}$ and $U_{R}$ are connected by a non-physical wave $\alpha_{NP}\in NP^{(\boldsymbol{\mu})}$,
\begin{eqnarray}\label{eq:4.4}
\|\mathcal{S}_{h}(U^{(\boldsymbol{\mu}),\nu}(\hat{x},\cdot))-U^{(\boldsymbol{\mu}),\nu}(\hat{x}+h,\cdot)\|_{L^1(y_{\mathcal{I}}-\eta, y_{\mathcal{I}}+\eta)}
\leq C_{\mathcal{I}} \alpha_{NP}\,h,
\end{eqnarray}
where constant $C_{\mathcal{I}}>0$ is independent of  $(\boldsymbol{\mu},h)$, and constant $\eta$ satisfies $\eta>\hat{\lambda}h$.
\end{enumerate}
\end{lemma}

\begin{proof}  We divide the proof into three steps.

\smallskip
1. Without loss of the generality, we consider the case: $k=1$ only, since the case: $k=2$ can be dealt with in the same way.

By Proposition \ref{prop:3.3}, we know that, for sufficiently small $h>0$, $\mathcal{S}_{h}(U^{(\boldsymbol{\mu}),\nu}(\hat{x},\cdot))$
is the Riemann solution of system \eqref{eq:1.19}, consisting of three constant states $U_{L}$, $U_{M}$, and $U_{R}$ with
$U_{M}=(\rho_{M}, v_{M})^{\top}$. The three states are separated by the elementary waves $\beta_1$ and $\beta_2$.
Then we have
\begin{eqnarray}\label{eq:4.5}
\Phi(\boldsymbol{\beta};U_{L})=\Phi^{(\boldsymbol{\mu})}_{1}(\alpha_1; U_{L}, \boldsymbol{\mu})\qquad \mbox{for $\boldsymbol{\beta}=(\beta_1,\beta_2)$}.
\end{eqnarray}

By Proposition \ref{prop:2.1}, if $\|\boldsymbol{\mu}\|$ is sufficiently small, equation \eqref{eq:4.5} admits a unique
solution  $\boldsymbol{\beta}=(\beta_1,\beta_2)$ so that
\begin{eqnarray}\label{eq:4.6}
\beta_1=\alpha_1+O(1)\|\boldsymbol{\mu}\||\alpha_{1}-1|, \qquad\,\,  \beta_2=1+O(1)\|\boldsymbol{\mu}\||\alpha_{1}-1|.
\end{eqnarray}
Since $\alpha_1\in S^{(\boldsymbol{\mu})}_{1}$, $\beta_1$ is also a 1-shock wave, while $\beta_2$ may
be either a 2-shock wave or 2-rarefaction wave (see Fig. \ref{fig4.1} or Fig. \ref{fig4.2} below).

Let us first consider the case that both $\beta_1$ and  $\beta_2$ are shock waves.
As shown in Fig. \ref{fig4.1}, when $\eta>\hat{\lambda}h$, and $h>0$ is sufficiently small,
interval $(y_{I}-\eta, y_{I}+\eta)$ can be divided into two sub-intervals $I$ and $II$ by fronts $\alpha_1$, $\beta_1$,
and  $\beta_2$.

\begin{figure}[ht]
\begin{minipage}[t]{0.49\textwidth}
\centering
\begin{tikzpicture}[scale=0.6]
\draw [line width=0.04cm](-4,2.5)--(-4,-4.0);
\draw [line width=0.04cm](0.5,2.5)--(0.5,-4.0);
\draw [thin][<->](-3.9,2.1)--(0.4,2.1);

\draw [thick](-4,0)--(0.5, 1.0);
\draw [thick][red](-4,0)--(0.5, -1.2);
\draw [thick](-4,0)--(0.5, -3.1);

\draw [thick][blue](0.5,-1.2)--(0.5, 1.0);
\draw [thick][green](0.5,-3.1)--(0.5, -1.2);

\node at (2.6, 2) {$$};
\node at (-1.8, 2.7){$h$};
\node at (-5.1, 0) {$(\hat{x}, y_{\mathcal{I}})$};
\node at (-4, -4.9) {$x=\hat{x}$};
\node at (0.5, -4.9) {$x=\hat{x}+h$};

\node at (0.0, 1.3) {$\beta_{2}$};
\node at (0.0, -0.5) {$\alpha_1$};
\node at (0.0, -2.3) {$\beta_{1}$};

\node at (1.2, 0){$II$};
\node at (1.0, -2.4){$I$};

\node at (-1.8, 1.2){$U_{R}$};
\node at (-1.8, -0.1){$U_{M}$};
\node at (-1.8, -2.6){$U_{L}$};
\end{tikzpicture}
\caption{Comparison of the Riemann solvers for $\alpha_1\in S^{(\boldsymbol{\mu})}_1$ and $\beta_2$ being a $2$-shock wave}\label{fig4.1}
\end{minipage}
\begin{minipage}[t]{0.49\textwidth}
\centering
\begin{tikzpicture}[scale=0.6]
\draw [line width=0.04cm](-4,2.5)--(-4,-4.0);
\draw [line width=0.04cm](0.5,2.5)--(0.5,-4.0);
\draw [thin][<->](-3.9,2.1)--(0.4,2.1);

\draw [thin](-4,0)--(0.5, 1.2);
\draw [thin](-4,0)--(0.5, 1.0);
\draw [thin](-4,0)--(0.5, 0.8);
\draw [thin](-4,0)--(0.5, 0.6);
\draw [thick][red](-4,0)--(0.5, -1.2);
\draw [thick](-4,0)--(0.5, -3.1);

\draw [thick][green](0.5,0.6)--(0.5, 1.2);
\draw [thick][blue](0.5,-1.2)--(0.5, 0.6);
\draw [thick][green](0.5,-3.1)--(0.5, -1.2);

\node at (2.6, 2) {$$};
\node at (-1.8, 2.7){$h$};
\node at (-5.1, 0) {$(\hat{x}, y_{\mathcal{I}})$};
\node at (-4, -4.9) {$x=\hat{x}$};
\node at (0.5, -4.9) {$x=\hat{x}+h$};

\node at (-0.1, 1.5) {$\beta_{2}$};
\node at (-0.1, -0.5) {$\alpha_1$};
\node at (-0.1, -2.2) {$\beta_{1}$};

\node at (1.2, 0.8){$III$};
\node at (1.2, -0.2){$II$};
\node at (1.0, -2.4){$I$};

\node at (-1.6, 1.3){$U_{R}$};
\node at (-1.6, -0.2){$U_{M}$};
\node at (-1.6, -2.6){$U_{L}$};

\end{tikzpicture}
\caption{Comparison of the Riemann solvers for $\alpha_1\in S^{(\boldsymbol{\mu})}_1$ and $\beta_2$ being a $2$-rarefaction wave}\label{fig4.2}
\end{minipage}
\end{figure}

Denote the speeds of $\beta_1$ and $\beta_2$ by $\sigma_{1}(\beta_1)$ and $\sigma_{1}(\beta_2)$. By the \emph{Rankine-Hugoniot} (R-H) conditions:
\begin{eqnarray*}
&&\sigma_1(\beta_1)=\frac{\rho_{M}v_{M}-\rho_{L}v_{L}}{\rho_{M}-\rho_{L}}=v_{L}+\frac{\beta_1\varphi_{1}(\beta_1)}{\beta_1-1},\\
&&\sigma_2(\beta_2)=\frac{\rho_{R}v_{R}-\rho_{M}v_{M}}{\rho_{R}-\rho_{M}}=v_{M}+\frac{\beta_2\varphi_{2}(\beta_2)}{\beta_2-1}.
\end{eqnarray*}
Moreover, for $\sigma^{(\boldsymbol{\mu})}_1(\alpha_1)$, when $\|\boldsymbol{\mu}\|$ is sufficiently small, it follows from Lemma \ref{lem:2.6} that
\begin{eqnarray*}
\begin{split}
\sigma^{(\boldsymbol{\mu})}_1(\alpha_1)&=\frac{\rho_{R}v_{R}-\rho_{L}v_{L}}{\rho_{R}\sqrt{1-\tau^2B^{(\epsilon)}(\rho_{R},v_{R},\epsilon)}-\rho_{L}\sqrt{1-\tau^2B^{(\epsilon)}(\rho_{L},v_{L},\epsilon)}}\\[3pt]
&=\frac{\alpha_1\varphi^{(\boldsymbol{\mu})}_{1}(\alpha_1; U_{L},\boldsymbol{\mu})+(\alpha_1-1)v_{L}}{\alpha_1\sqrt{1-\tau^2B^{(\epsilon)}(\rho_{L}\alpha_1,\varphi^{(\boldsymbol{\mu})}_{1}+v_{L},\epsilon)}-\sqrt{1-\tau^2B^{(\epsilon)}(\rho_{L},v_{L},\epsilon)}}\\[3pt]
&=\frac{\alpha_1\varphi_{1}(\alpha_1)}{\alpha_1-1}+v_{L}+O(1)\|\boldsymbol{\mu}\|.
\end{split}
\end{eqnarray*}
Therefore, using \eqref{eq:4.6} and Propositions \ref{prop:3.1}--\ref{prop:3.2}, we have
\begin{eqnarray}
&&\sigma^{(\boldsymbol{\mu})}_1(\alpha_1)-\sigma_1(\beta_1)=\frac{\alpha_1\varphi_{1}(\alpha_1)}{\alpha_1-1}+v_{L}+O(1)\|\boldsymbol{\mu}\|-v_{L}
-\frac{\beta_1\varphi_{1}(\beta_1)}{\beta_1-1}
=O(1)\|\boldsymbol{\mu}\|,\qquad\qquad\nonumber\\
&&|\sigma_2(\beta_2)-\sigma^{(\boldsymbol{\mu})}_1(\alpha_1)|\leq |\sigma_2(\beta_2)|
+|\sigma^{(\boldsymbol{\mu})}_1(\alpha_1)|<\infty.\nonumber
\end{eqnarray}
Then
\begin{eqnarray*}
&&|I|\leq |\sigma^{(\boldsymbol{\mu})}_1(\alpha_1)-\sigma_1(\beta_1)|h+|\dot{y}_{\alpha_1}-\sigma^{(\boldsymbol{\mu})}_1(\alpha_1)|h
  \leq \big(O(1)\|\boldsymbol{\mu}\|+2^{-\nu}\big)h,\\
&&|II|\leq |\sigma_2(\beta_2)-\sigma^{(\boldsymbol{\mu})}_1(\alpha_1)|h+|\dot{y}_{\alpha_1}-\sigma^{(\boldsymbol{\mu})}_1(\alpha_1)|h\leq \big(O(1)+2^{-\nu}\big)h.
\end{eqnarray*}

On the other hand, on interval $I$,
\begin{eqnarray*}
|\mathcal{S}_{h}U^{(\boldsymbol{\mu}),\nu}(\hat{x},y)-U^{(\boldsymbol{\mu}),\nu}(\hat{x}+h,y)|=|\rho_{R}-\rho_{L}|+|v_{R}-v_{L}|\leq O(1)|\alpha_1-1|,
\end{eqnarray*}
and, on interval $II$,
\begin{align*}
|\mathcal{S}_{h}U^{(\boldsymbol{\mu}),\nu}(\hat{x},y)-U^{(\boldsymbol{\mu}),\nu}(\hat{x}+h,y)|&=|\rho_{R}-\rho_{M}|+|v_{R}-v_{M}|\\[2pt]
&\leq O(1)|\beta_2-1|
\leq O(1)\|\boldsymbol{\mu}\||\alpha_1-1|.
\end{align*}

Base on these estimates, we finally obtain
\begin{align*}
\begin{split}
&\|\mathcal{S}_{h}(U^{(\boldsymbol{\mu}),\nu}(\hat{x},\cdot))-U^{(\boldsymbol{\mu}),\nu}(\hat{x}+h,\cdot)\|_{L^1(y_{\mathcal{I}}-\eta,y_{\mathcal{I}}+\eta)}\\[3pt]
& =\|\mathcal{S}_{h}(U^{(\boldsymbol{\mu}),\nu}(\hat{x},\cdot))-U^{(\boldsymbol{\mu}),\nu}(\hat{x}+h,\cdot)\|_{L^1(I\cup II)}\\[3pt]
& \leq O(1)|I||\alpha_1-1|+O(1)|II|\|\boldsymbol{\mu}\||\alpha_1-1|\\[3pt]
& \leq O(1)\big(\|\boldsymbol{\mu}\|+2^{-\nu}\big)|\alpha_1-1|h,
\end{split}
\end{align*}
which completes the proof of \rm(i) for the case that $\beta_2$ is a shock wave.

\smallskip
Next, we consider the case that $\beta_2$ is a rarefaction wave as shown in Fig. \ref{fig4.2}. In this case,
\begin{eqnarray}\label{eq:4.8}
\mathcal{S}_{h}(U^{(\boldsymbol{\mu}),\nu}(\hat{x},y))=
\begin{cases}
U_{R},&\quad\ \xi\in[\lambda_{2}(U_{R}),\frac{\eta}{h}),\\[3pt]
\Phi_{2}(\beta_{2}(\xi);U_{M}),&\quad\  \xi\in[\lambda_{2}(U_{M}),\lambda_{2}(U_{R})),\\[3pt]
U_{M}, &\quad\ \xi\in[\sigma_1(\beta_1),\lambda_{2}(U_{M})),\\[3pt]
U_{L},&\quad\ \xi\in(-\frac{\eta}{h}, \sigma_1(\beta_1)),
\end{cases}
\end{eqnarray}
where $\xi=\frac{y-y_{I}}{h}$, $\beta_{2}(\lambda_{2}(U_{M}))=1$, and $\beta_{2}(\lambda_{2}(U_{R}))=\beta_2$.

Following the same argument as done for the case that $\beta_2$ is a shock wave, we have 
\begin{eqnarray}\label{eq:4.9}
\|\mathcal{S}_{h}(U^{(\boldsymbol{\mu}),\nu}(\hat{x},\cdot))-U^{(\boldsymbol{\mu}),\nu})(\hat{x}+h,\cdot)\|_{L^1(I)}
\leq O(1)(\|\boldsymbol{\mu}\|+2^{-\nu})|\alpha_1-1|h.
\end{eqnarray}

For the length of interval $II$, it follows from Proposition \ref{prop:3.1} that 
\begin{align*}
|II|\leq |\dot{y}_{\alpha_1}-\xi_0|\leq \big(2^{-\nu}+|\sigma^{(\boldsymbol{\mu})}_{1}-\lambda_2(U_{M})|\big)h
\leq \big(2^{-\nu}+O(1)\big)h.
\end{align*}
Moreover, on $II$, by \eqref{eq:4.6} and Proposition \ref{prop:3.2}, we obtain
\begin{align*}
\begin{split}
|\mathcal{S}_{h}(U^{(\boldsymbol{\mu}),\nu}(\hat{x},y))-U^{(\boldsymbol{\mu}),\nu}(\hat{x}+h,y)|&=|\rho_{R}-\rho_{M}|+|v_{R}-v_{M}|\\[3pt]
&\leq O(1)|\beta_2-1|
\leq O(1)|\alpha_1-1|\|\boldsymbol{\mu}\|,
\end{split}
\end{align*}
so that
\begin{align}\label{eq:4.10}
\|\mathcal{S}_{h}(U^{(\boldsymbol{\mu}),\nu}(\hat{x},\cdot))-U^{(\boldsymbol{\mu}),\nu}(\hat{x}+h,\cdot)\|_{L^1(II)}
\leq O(1)\big(\|\boldsymbol{\mu}\|+2^{-\nu}\big)|\alpha_1-1|h.
\end{align}

On interval $III$, it is direct to see
\begin{align*}
|\mathcal{S}_{h}(U^{(\boldsymbol{\mu}),\nu}(\hat{x},y))-\,U^{(\boldsymbol{\mu}),\nu}(\hat{x}+h,y)|\,&=|\rho(\xi)-\rho_{R}|+|v(\xi)-v_{R}|\\[3pt]
&\leq O(1)|\xi_0-\xi_{1}|=O(1)|\lambda_{2}(U_{M})-\lambda_{2}(U_R)|\\[3pt]
&\leq O(1)|\beta_2-1|\leq O(1)|\alpha_1-1|\|\boldsymbol{\mu}\|,\\[3pt]
|III|\leq O(1)|\xi_0-\xi_{1}|\leq O(1)|\lambda_1(U_{M}&)-\lambda_{1}(U_R)|\leq O(1)h,
\end{align*}
so that
\begin{align}\label{eq:4.11}
\|\mathcal{S}_{h}(U^{(\boldsymbol{\mu}),\nu}(\hat{x},\cdot))-U^{(\boldsymbol{\mu}),\nu}(\hat{x}+h,\cdot)\|_{L^1(III)}
\leq O(1)\|\boldsymbol{\mu}\||\alpha_1-1|h.
\end{align}

Finally, combining estimates \eqref{eq:4.9}--\eqref{eq:4.11} together, we have
\begin{align*}
\begin{split}
&\|\mathcal{S}_{h}(U^{(\boldsymbol{\mu}),\nu}(\hat{x},y))-U^{(\boldsymbol{\mu}),\nu}(\hat{x}+h,y)\|_{L^1(y_{I}-\eta,y_{I}+\eta)}\\[3pt]
&=\|\mathcal{S}_{h}U^{(\boldsymbol{\mu}),\nu}(\hat{x},y)-U^{(\boldsymbol{\mu}),\nu}(\hat{x}+h,y)\|_{L^1(I\cup II\cup III)}\\[3pt]
&\leq O(1)\big(\|\boldsymbol{\mu}\|+2^{-\nu}\big)|\alpha_1-1|h,
\end{split}
\end{align*}
which gives estimate \eqref{eq:4.2} for the case that $\beta_2$ is a rarefaction wave.

\medskip
2. Without loss of the generality, similarly, we consider the case: $k=1$ only.
The Riemann solver $\mathcal{S}_{h}(U^{(\boldsymbol{\mu}),\nu}(\hat{x},y))$
consists of three constant states $U_{L}$, $U_{M}$, and $U_{R}$,
which are separated by the elementary waves $\beta_1$ and $\beta_2$
with equation \eqref{eq:4.5} and estimates \eqref{eq:4.6}.
Moreover, it follows from \eqref{eq:4.6} that $\beta_1$ is a $1$-rarefaction wave.
However, $\beta_2$ may be either a $2$-shock wave or a $2$-rarefaction wave (see Fig. \ref{fig4.3} or Fig. \ref{fig4.4} below).
\begin{figure}[ht]
\begin{minipage}[t]{0.49\textwidth}
\centering
\begin{tikzpicture}[scale=0.6]
\draw [line width=0.04cm](-4,2.5)--(-4,-4.0);
\draw [line width=0.04cm](0.5,2.5)--(0.5,-4.0);
\draw [thin][<->](-3.9,2.1)--(0.4,2.1);

\draw [thick](-4,0)--(0.5, 1.0);
\draw [thick][blue](-4,0)--(0.5, -1.2);
\draw [thick][blue](-4,0)--(0.5, -1.4);
\draw [thick][blue](-4,0)--(0.5, -1.6);
\draw [thick][blue](-4,0)--(0.5, -1.8);
\draw [thick][red](-4,0)--(0.5, -3.1);

\draw [thick][red](0.5,-1.2)--(0.5, 1.0);
\draw [line width=0.05cm](0.5,-1.9)--(0.5, -1.2);
\draw [thick][green](0.5,-3.1)--(0.5, -1.8);

\node at (2.6, 2) {$$};
\node at (-1.8, 2.7){$h$};
\node at (-5.1, 0) {$(\hat{x}, y_{\mathcal{I}})$};
\node at (-4, -4.9) {$x=\hat{x}$};
\node at (0.5, -4.9) {$x=\hat{x}+h$};

\node at (0.0, 1.4) {$\beta_{2}$};
\node at (0.0, -0.5) {$\beta_{1}$};
\node at (0.0, -2.3) {$\alpha_1$};

\node at (1.3, -0.2){$III$};
\node at (1.1, -1.5){$II$};
\node at (1.0, -2.5){$I$};

\node at (-1.8, 1.2){$U_{R}$};
\node at (-1.8, -0.1){$U_{M}$};
\node at (-1.8, -2.6){$U_{L}$};

\end{tikzpicture}
\caption{Comparison of the Riemann solvers for $\alpha_1\in R^{(\boldsymbol{\mu})}_1$ and $\beta_2$ being $2$-shock wave}\label{fig4.3}
\end{minipage}
\begin{minipage}[t]{0.49\textwidth}
\centering
\begin{tikzpicture}[scale=0.6]
\draw [line width=0.04cm](-4,2.5)--(-4,-4.0);
\draw [line width=0.04cm](0.5,2.5)--(0.5,-4.0);
\draw [thin][<->](-3.9,2.1)--(0.4,2.1);

\draw [thin](-4,0)--(0.5, 1.2);
\draw [thin](-4,0)--(0.5, 1.0);
\draw [thin](-4,0)--(0.5, 0.8);
\draw [thin](-4,0)--(0.5, 0.6);
\draw [thick][blue](-4,0)--(0.5, -1.2);
\draw [thick][blue](-4,0)--(0.5, -1.4);
\draw [thick][blue](-4,0)--(0.5, -1.6);
\draw [thick][blue](-4,0)--(0.5, -1.8);
\draw [thick][red](-4,0)--(0.5, -3.1);

\draw [thick][green](0.5,0.6)--(0.5, 1.2);
\draw [thick][red](0.5,-1.2)--(0.5, 0.6);
\draw [line width=0.05cm](0.5,-1.9)--(0.5, -1.2);
\draw [thick][green](0.5,-3.1)--(0.5, -1.8);

\node at (2.6, 2) {$$};
\node at (-1.8, 2.7){$h$};
\node at (-5.1, 0) {$(\hat{x}, y_{\mathcal{I}})$};
\node at (-4, -4.9) {$x=\hat{x}$};
\node at (0.5, -4.9) {$x=\hat{x}+h$};

\node at (0.0, 1.5) {$\beta_{2}$};
\node at (0.0, -0.5) {$\beta_{1}$};
\node at (0.0, -2.3) {$\alpha_1$};

\node at (1.2, 0.9){$IV$};
\node at (1.2, -0.3){$III$};
\node at (1.0, -1.5){$II$};
\node at (0.9, -2.6){$I$};

\node at (-1.6, 1.2){$U_{R}$};
\node at (-1.6, -0.2){$U_{M}$};
\node at (-1.6, -2.6){$U_{L}$};
\end{tikzpicture}
\caption{Comparison of the Riemann solvers for $\alpha_1\in R^{(\boldsymbol{\mu})}_1$ and $\beta_2$ being a $2$-rarefaction wave}\label{fig4.4}
\end{minipage}
\end{figure}

If $\beta_2$ is a shock wave, then
\begin{eqnarray}\label{eq:4.12}
\mathcal{S}_{h}(U^{(\boldsymbol{\mu}),\nu}(\hat{x},y))=
\begin{cases}
U_{R},&\qquad \xi\in[\sigma_{2}(\beta_2),\frac{\eta}{h}),\\[3pt]
U_{M}, &\qquad \xi\in[\lambda_{1}(U_{M}),\sigma_{2}(\beta_2)),\\[3pt]
\Phi_{1}(\beta_{1}(\xi);U_{L}),&\qquad\xi\in[\lambda_{1}(U_{L}),\lambda_{1}(U_{M})),\\[3pt]
U_{L},&\qquad \xi\in(-\frac{\eta}{h}, \lambda_1(U_L)),
\end{cases}
\end{eqnarray}
where $\xi=\frac{y-y_{\mathcal{I}}}{h}$, $\beta_{1}(\lambda_{1}(U_{L}))=1$, $\beta_{1}(\lambda_{1}(U_{M}))=\beta_1$,
and $\sigma_{2}(\beta_2)$ is the speed of $\beta_2$. Then interval $(y_{\mathcal{I}}-\eta, y_{\mathcal{I}}+\eta)$
is divided into three subintervals $I$, $II$, and $III$.

Using \eqref{eq:2.41}
and Proposition \ref{prop:3.1}, a direct computation leads to
\begin{align*}
\begin{split}
|I|&=|\lambda_1(U_L)-\dot{y}_{\alpha_1}|h\leq \big(|\lambda_1(U_L)-\lambda^{(\boldsymbol{\mu})}_{1}(U_{R}, \boldsymbol{\mu})|+2^{-\nu}\big)h\\[3pt]
&\leq \big(|\lambda_1(U_L)-\lambda_{1}(U_{R})|+O(1)\|\boldsymbol{\mu}\|+2^{-\nu}\big)h\\[3pt]
&\leq O(1)\big(|\alpha_1-1|+\|\boldsymbol{\mu}\|+2^{-\nu}\big)h\\[3pt]
&\leq O(1)\big(\|\boldsymbol{\mu}\|+\nu^{-1}+2^{-\nu}\big)h.
\end{split}
\end{align*}
Moreover, on $I$,
\begin{align*}
|\mathcal{S}_{h}(U^{(\boldsymbol{\mu}),\nu}(\hat{x},y))-U^{(\boldsymbol{\mu}),\nu}(\hat{x}+h,y)|&=|\rho_{R}-\rho_{L}|+|v_{R}-v_{L}|
\leq O(1)|\alpha_1-1|.
\end{align*}
Therefore, we have
\begin{eqnarray}\label{eq:4.13}
\|\mathcal{S}_{h}(U^{(\boldsymbol{\mu}),\nu}(\hat{x},\cdot))-U^{(\boldsymbol{\mu}),\nu}(\hat{x}+h,\cdot)\|_{L^1(I)}
\leq O(1)\big(\|\boldsymbol{\mu}\|+2^{-\nu}+\nu^{-1}\big)|\alpha_1-1|h.
\end{eqnarray}

Next, it follows from \eqref{eq:4.6} that
\begin{eqnarray*}
|II|=|\lambda_1(U_{M})-\lambda_1(U_{L})|h\leq O(1)|\beta_1-1|h=O(1)\big(1+\|\boldsymbol{\mu}\|\big)|\alpha_1-1|h,
\end{eqnarray*}
and, on  interval $II$,
\begin{align*}
\begin{split}
&|\mathcal{S}_{h}(U^{(\boldsymbol{\mu}),\nu}(\hat{x},y))-U^{(\boldsymbol{\mu}),\nu}(\hat{x}+h,y)|\\[3pt]
&=|\Phi^{(1)}_{1}(\beta_{1}(\xi);U_{L})-\rho_{R}|+|\Phi^{(2)}_{1}(\beta_{1}(\xi);U_{L})-v_{R}|\\[3pt]
&\leq |\Phi^{(1)}_{1}(\beta_{1}(\xi);U_{L})-\rho_{M}|+|\rho_{M}-\rho_{R}|+|\Phi^{(2)}_{1}(\beta_{1}(\xi);U_{L})-v_{M}|+|v_{M}-v_{R}|\\[3pt]
&\leq O(1)\big(|\xi-\xi_{1}|+|\beta_2-1|\big)\\[3pt]
&\leq O(1)\big(|\beta_1-1|+|\beta_2-1|\big)\\[3pt]
&\leq O(1)\big(1+\|\boldsymbol{\mu}\|\big)|\alpha_1-1|.
\end{split}
\end{align*}
Thus, by Proposition \ref{prop:3.1}, we obtain
\begin{align}\label{eq:4.14}
\|\mathcal{S}_{h}(U^{(\boldsymbol{\mu}),\nu}(\hat{x},\cdot))-U^{(\boldsymbol{\mu}),\nu}(\hat{x}+h,\cdot)\|_{L^1(II)}
&\leq O(1)\big(1+\|\boldsymbol{\mu}\|\big)^2|\alpha_1-1|^2h\nonumber\\[3pt]
&\leq O(1)\nu^{-1}|\alpha_1-1|h.
\end{align}

Finally, it follows from Proposition \ref{prop:3.1} and estimates \eqref{eq:4.6} that
\begin{eqnarray*}
|III|=|\lambda_1(U_{M})-\sigma_{2}(\beta_2)|h
=\Big|v_{M}-\frac{1}{a_{\infty}}-\frac{\beta_2 v_{R}-v_{M}}{\beta_2-1}\Big|h\leq O(1)h,
\end{eqnarray*}
and, on interval $III$,
\begin{eqnarray*}
\begin{split}
|\mathcal{S}_{h}(U^{(\boldsymbol{\mu}),\nu}(\hat{x},y))-U^{(\boldsymbol{\mu}),\nu}(\hat{x}+h,y)|&=|\rho_{M}-\rho_{R}|+|v_{M}-v_{R}|\\[3pt]
&=O(1)|\beta_2-1|
=O(1)|\alpha_1-1|\|\boldsymbol{\mu}\|.
\end{split}
\end{eqnarray*}
Then
\begin{eqnarray}\label{eq:4.15}
\|\mathcal{S}_{h}(U^{(\boldsymbol{\mu}),\nu}(\hat{x},\cdot))-U^{(\boldsymbol{\mu}),\nu}(\hat{x}+h,\cdot)\|_{L^1(III)}\leq O(1)|\alpha_1-1|\|\boldsymbol{\mu}\|h.
\end{eqnarray}

Combining estimates \eqref{eq:4.13}--\eqref{eq:4.15}, we finally obtain
\begin{eqnarray*}
\|\mathcal{S}_{h}(U^{(\boldsymbol{\mu}),\nu}(\hat{x},\cdot))-U^{(\boldsymbol{\mu}),\nu}(\hat{x}+h,\cdot)\|_{L^1(y_{\mathcal{I}}-\eta, y_{\mathcal{I}}+\eta)}
\leq O(1)\big(\|\boldsymbol{\mu}\|+\nu^{-1}+2^{-\nu}\big)|\alpha_1-1|h.
\end{eqnarray*}

\smallskip
When $\beta_2$ is a rarefaction wave, as shown in Fig. \ref{fig4.4}, interval $(y_{\mathcal{I}}-\eta, y_{\mathcal{I}}+\eta)$
is divided into four subintervals $I$, $II$, $III$, and $IV$.
Then 
\begin{eqnarray}\label{eq:4.16}
\mathcal{S}_{h}(U^{(\boldsymbol{\mu}),\nu}(\hat{x},y))=
\begin{cases}
U_{R},&\qquad \xi\in[\lambda_{2}(U_{R}),\frac{\eta}{h}),\\[3pt]
\Phi_{2}(\beta_2(\xi); U_{M}), &\qquad \xi \in[\lambda_2(U_{M}), \lambda_{2}(U_{R})),\\[3pt]
U_{M}, &\qquad \xi\in[\lambda_{1}(U_{M}),\lambda_2(U_{M})),\\[3pt]
\Phi_{1}(\beta_{1}(\xi);U_{L}),&\qquad\xi\in[\lambda_{1}(U_{L}),\lambda_{1}(U_{M})),\\[3pt]
U_{L},&\qquad \xi\in(-\frac{\eta}{h}, \lambda_1(U_L)),
\end{cases}
\end{eqnarray}
where $\xi=\frac{y-y_{\mathcal{I}}}{h}$, $\beta_{1}(\lambda_{1}(U_{L}))=1$, $\beta_{1}(\lambda_{1}(U_{M}))=\beta_1$,
$\beta_2(\lambda_2(U_{M}))=1$, and $\beta_{2}(\lambda_3(U_{R}))=\beta_2$.
First, by the same argument for the case that $\beta_2$ is a shock above,
\begin{eqnarray}\label{eq:4.17}
\|\mathcal{S}_{h}(U^{(\boldsymbol{\mu}),\nu}(\hat{x},\cdot))-U^{(\boldsymbol{\mu}),\nu}(\hat{x}+h,\cdot)\|_{L^1(I\cup II \cup III)}
\leq O(1)\big(\|\boldsymbol{\mu}\|+\nu^{-1}+2^{-\nu}\big)|\alpha_1-1|h.\qquad
\end{eqnarray}

Thus, it suffices to
consider the estimate of $\|\mathcal{S}_{h}(U^{(\boldsymbol{\mu}),\nu}(\hat{x},\cdot))-U^{(\boldsymbol{\mu}),\nu}(\hat{x}+h,\cdot)\|_{L^1(IV)}$.
By Proposition \ref{prop:3.1}, we directly have
\begin{eqnarray*}
|IV|=|\lambda_{2}(U_{R})-\lambda_{2}(U_{M})|h\leq O(1)h,
\end{eqnarray*}
and, on $IV$,
\begin{eqnarray*}
\begin{split}
|\mathcal{S}_{h}(U^{(\boldsymbol{\mu}),\nu}(\hat{x},y))-U^{(\boldsymbol{\mu}),\nu}(\hat{x}+h,y)|
&=|\Phi^{(1)}_{2}(\beta_2(\xi); U_{M})-\rho_{R}|+|\Phi^{(2)}_{2}(\beta_2(\xi); U_{M})-v_{R}|\\[3pt]
&\leq O(1)|\xi_2-\xi_{3}|
\leq O(1)|\beta_2-1|
\leq O(1)|\alpha_1-1|\|\boldsymbol{\mu}\|,
\end{split}
\end{eqnarray*}
so that
\begin{eqnarray}\label{eq:4.18}
\|\mathcal{S}_{h}(U^{(\boldsymbol{\mu}),\nu}(\hat{x},y))-U^{(\boldsymbol{\mu}),\nu}(\hat{x}+h,y)\|_{L^1(IV)}\leq O(1)|\alpha_1-1|\|\boldsymbol{\mu}\|h.
\end{eqnarray}

Then combining estimate \eqref{eq:4.17} with estimate \eqref{eq:4.18} yields
\begin{eqnarray*}
\begin{split}
&\|\mathcal{S}_{h}(U^{(\boldsymbol{\mu}),\nu}(\hat{x},\cdot))-U^{(\boldsymbol{\mu}),\nu}(\hat{x}+h,\cdot)\|_{L^1(y_{\mathcal{I}}-\eta, y_{\mathcal{I}}+\eta)}\\[3pt]
&
\leq \|\mathcal{S}_{h}(U^{(\boldsymbol{\mu}),\nu}(\hat{x},\cdot))-U^{(\boldsymbol{\mu}),\nu}(\hat{x}+h,\cdot)\|_{L^1(I\cup II\cup III\cup IV)}\\[3pt]
&
\leq O(1)\big(\|\boldsymbol{\mu}\|+\nu^{-1}+2^{-\nu}\big)|\alpha_1-1|h.
\end{split}
\end{eqnarray*}

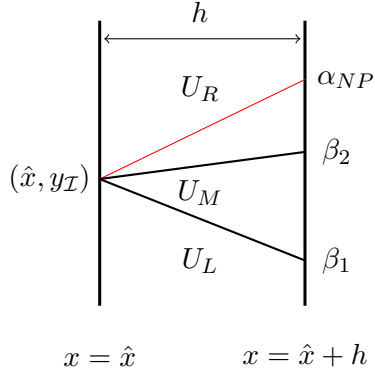
\begin{figure}[ht]
\begin{tikzpicture}[scale=0.6]
\draw [line width=0.04cm](-4,2.5)--(-4,-3.8);
\draw [line width=0.04cm](0.5,2.5)--(0.5,-3.8);
\draw [thin][<->](-3.9,2.1)--(0.4,2.1);

\draw [thin][red](-4,-1)--(0.5, 1.2);
\draw [thick](-4,-1)--(0.5, -0.4);
\draw [thick](-4,-1)--(0.5, -2.8);

\node at (2.6, 2) {$$};
\node at (-1.8, 2.7){$h$};
\node at (-5.1, -1.0) {$(\hat{x}, y_{\mathcal{I}})$};
\node at (-4, -4.9) {$x=\hat{x}$};
\node at (0.5, -4.9) {$x=\hat{x}+h$};

\node at (1.4, 1.2) {$\alpha_{NP}$};
\node at (1.2, -0.4) {$\beta_{2}$};
\node at (1.2, -2.8) {$\beta_{1}$};

\node at (-1.8, 1.0){$U_{R}$};
\node at (-1.8, -1.3){$U_{M}$};
\node at (-1.8, -2.8){$U_{L}$};

\end{tikzpicture}
\caption{Comparison of the Riemann solvers for $\alpha_{NP}\in NP^{(\boldsymbol{\mu})}$}\label{fig4.5}
\end{figure}

3. When the front in $U^{(\nu),\nu}(\hat{x}+h,\cdot)$ is a non-physical wave $\alpha_{NP}$, as shown in Fig. \ref{fig4.5},
$|U_{R}-U_{L}|=\alpha_{NP}$, and the Riemann solution $\mathcal{S}_{h}(U^{(\boldsymbol{\mu}),\nu}(\hat{x},\cdot))$
consists of two waves $\beta_1$ and $\beta_2$ satisfying
\begin{eqnarray}
U_{R}=\Phi(\beta_1,\beta_2; U_{L}).
\end{eqnarray}

Applying Proposition \ref{prop:2.1} leads  to
\begin{eqnarray}\label{eq:4.20}
|\beta_1-1|+|\beta_2-1|=O(1)\alpha_{NP}.
\end{eqnarray}

Let $U_{M}$ be the middle state of $\mathcal{S}_{h}(U^{(\boldsymbol{\mu}),\nu}(\hat{x},\cdot))$.
Then it follows from Propositions \ref{prop:3.1}--\ref{prop:3.2}
and estimate \eqref{eq:4.20} that
\begin{eqnarray*}
\begin{split}
 \|\mathcal{S}_{h}(U^{(\boldsymbol{\mu}),\nu}(\hat{x},\cdot))-U^{(\boldsymbol{\mu}),\nu}(\hat{x}+h,\cdot)\|_{L^1(y_{\mathcal{I}}-\eta, y_{\mathcal{I}}+\eta)}
& \leq O(1)\big(|U_{M}-U_{L}|+|U_{M}-U_{R}|\big)h\\[3pt]
&\leq O(1)\big(|\beta_1-1|+|\beta_2-1|\big)h\\[3pt]
& \leq O(1)\alpha_{NP}\,h.
\end{split}
\end{eqnarray*}
\end{proof}

Based on Lemma \ref{lem:4.1}, it is direct to derive the following corollary for the case that $U^{(\boldsymbol{\mu}),\nu}(\hat{x}+h,y)$
contains more than one discontinuities:

\begin{corollary}\label{coro:4.1}
Let $U^{(\boldsymbol{\mu}),\nu}$ be an approximate solution of problem \eqref{eq:1.15}--\eqref{eq:1.17} constructed in {\rm \S 3.1}
with a jump at point $(\hat{x},y)\in\Omega$. Let $\mathcal{S}$ be the uniformly Lipschitz continuous semigroup obtained
in {\rm Proposition \ref{prop:3.3}}.
Denote
\begin{equation}\label{eq:4.21}
U^{(\boldsymbol{\mu}),\nu}(\hat{x}+h,y)=
\begin{cases}
U_{R}=(\rho_{R},v_{R})^{\top}\quad &\mbox{for $y>y_{\mathcal{I}}+\hat{\lambda}h$},\\[3pt]
\hat{U}_{R}=(\hat{\rho}_{R},\hat{v}_{R})^{\top}\quad &\mbox{for $y_{\mathcal{I}}+\dot{y}_{\alpha_2}h<y<y_{\mathcal{I}}+\hat{\lambda}h$},\\[3pt]
U_{M}=(\rho_{M},v_{M})^{\top}\quad &\mbox{for $y_{\mathcal{I}}+\dot{y}_{\alpha_1}h<y<y_{\mathcal{I}}+\dot{y}_{\alpha_2}h$},\\[3pt]
U_{L}=(\rho_{L},v_{L})^{\top}\quad &\mbox{for $y<y_{\mathcal{I}}+\dot{y}_{\alpha_1}h$},
\end{cases}
\end{equation}
where $U_{L}=U^{(\boldsymbol{\mu}),\nu}(\hat{x}, y_{\mathcal{I}}-)$, $U_{R}=U^{(\boldsymbol{\mu}),\nu}(\hat{x}, y_{\mathcal{I}}+)$,
$|\dot{y}_{\alpha_k}|<\hat{\lambda}$ with $k=1,2$,
and $U_{L}, \hat{U}_{R}$, and $U_{R}$ satisfy \eqref{eq:3.3}.
For $\|\boldsymbol{\mu}\|\leq \|\boldsymbol{\bar{\mu}}^{*}_{0}\|$ with $\boldsymbol{\bar{\mu}}^{*}_{0}$ given in {\rm Proposition \ref{prop:3.1}}
and, for sufficiently small $h>0$, if $U_{L}$ and $U_{M}$ are connected by a $1$-shock wave $\alpha_1\in S^{(\boldsymbol{\mu})}_1$
with $|\dot{y}_{\alpha_1}-\sigma^{(\boldsymbol{\mu})}_{1}(\alpha_1)|<2^{-\nu}$ {\rm (}or a $1$-rarefaction front $\alpha_1\in R^{(\boldsymbol{\mu})}_1$
with $|\dot{y}_{\alpha_1}-\lambda_{1}(U_{M}, \boldsymbol{\mu})|<2^{-\nu}${\rm )},
if $U_{M}$ and $\hat{U}_{R}$ are connected by a $2$-shock wave $\alpha_2\in S^{(\boldsymbol{\mu})}_2$
with $|\dot{y}_{\alpha_2}-\sigma^{(\boldsymbol{\mu})}_{2}(\alpha_2)|<2^{-\nu}$
{\rm (}or a $2$-rarefaction front $\alpha_2\in R^{(\boldsymbol{\mu})}_2$ with $|\dot{y}_{\alpha_2}-\lambda_{2}(\hat{U}_{R}, \boldsymbol{\mu})|<2^{-\nu}${\rm )},
and if $\hat{U}_{R}$ and $U_{R}$ are connected by a non-physical wave front $\alpha_{\mathcal{NP}}\in NP^{(\boldsymbol{\mu})}$, then
\begin{eqnarray}\label{eq:4.22}
&&\|\mathcal{S}_{h}(U^{(\boldsymbol{\mu}),\nu}(\hat{x},\cdot))-U^{(\boldsymbol{\mu}),\nu}(\hat{x}+h,\cdot)\|_{L^1(y_{\mathcal{I}}-\eta, y_{\mathcal{I}}+\eta)}\nonumber\\[3pt]
&&\leq C_{\mathcal{II}}\big(\|\boldsymbol{\mu}\|+2^{-\nu}+\nu^{-1}\big)\Big(\sum_{k=1,2}|\alpha_{k}-1|+ \alpha_{NP}\Big)h,
\end{eqnarray}
where constant $\eta>0$ satisfies $\eta>\hat{\lambda}h$, and constant $C_{\mathcal{II}}>0$ is independent of $\boldsymbol{\mu}$ and $h$.
\end{corollary}

Next, we consider the comparison of the approximate solution $U^{(\boldsymbol{\mu}),\nu}(\hat{x}+h,y)$ of
problem \eqref{eq:1.15}--\eqref{eq:1.17} and the entropy solution $\mathcal{S}_{h}(U^{(\boldsymbol{\mu}),\nu}(\hat{x},\cdot))$
of problem \eqref{eq:1.19}--\eqref{eq:1.21} near boundary $\Gamma_{\textrm{w}}$
with $(\hat{x},b_0\hat{x})$ as a discontinuity point on it.
Denote
\begin{eqnarray}
&&U_{L}=(\rho_{L}, v_{L})^{\top}:=U^{(\boldsymbol{\mu}),\nu}(\hat{x},b_0\hat{x}-),
  \quad U_{b}=(\rho_{b}, v_{b})^{\top}:=U^{(\boldsymbol{\mu}),\nu}(\hat{x},b_0\hat{x}),\label{eq:4.23}\\[3pt]
&&U^{(\boldsymbol{\mu})}_{b}:=(\rho^{(\boldsymbol{\mu})}_{b}, v^{(\boldsymbol{\mu})}_{b})^{\top}=U^{(\boldsymbol{\mu}),\nu}(x,b_0x)
\qquad\,\, \mbox{ for $x\in(\hat{x},\hat{x}+h),\ h>0$},\label{eq:4.24}\\[3pt]
&&U^{(\boldsymbol{\mu}),\nu}_{b}(\hat{x},y):=
\begin{cases}
U_{b} \quad &\mbox{for $y=b_0\hat{x}$},\\[3pt]
U_{L} \quad &\mbox{for $y<b_0\hat{x}$}.
\end{cases}\label{eq:4.25}
\end{eqnarray}
Then we have the following lemma:
\begin{lemma}\label{lem:4.2}
Let $U^{(\boldsymbol{\mu}),\nu}_{b}(\hat{x},y)$ be defined by \eqref{eq:4.25} with $v_{b}=b_{0}$.
Let $\mathcal{S}$ be a uniform Lipschitz continuous semigroup obtained in {\rm Proposition \ref{eq:3.2}}.
Define
\begin{equation}\label{eq:4.26}
U^{(\boldsymbol{\mu}),\nu}_{b}(\hat{x}+h,y):=
\begin{cases}
U^{(\boldsymbol{\mu})}_{b}\quad &\mbox{for $b_0\hat{x}+\dot{y}_{\alpha_1}h<y<b_{0}h$},\\[3pt]
U_{L}\quad &\mbox{for $y<b_0\hat{x}+\dot{y}_{\alpha_1}h$},
\end{cases}
\end{equation}
where $\dot{y}_{\alpha_1}\in (-\hat{\lambda}, b_0)$, and $U^{(\boldsymbol{\mu})}_{b}$ satisfies
\begin{equation}\label{eq:4.27}
v^{(\boldsymbol{\mu})}_{b}=b_0\sqrt{1-\tau^2B^{(\epsilon)}(\rho^{(\boldsymbol{\mu})}_b, v^{(\boldsymbol{\mu})}_b, \epsilon)}.
\end{equation}
For $\|\boldsymbol{\mu}\|\leq \|\bar{\boldsymbol{\mu}}^{*}_0\|$ with $\boldsymbol{\bar{\mu}}^{*}_{0}$ given in {\rm Proposition \ref{prop:3.1}}
and for sufficiently small $h>0$,

\begin{enumerate}
\item[(i)] if $U^{(\boldsymbol{\mu})}_{b}$ and $U_{L}$ are connected by a $1$-shock wave $\alpha_1\in S^{(\boldsymbol{\mu})}$
with $|\dot{y}_{\alpha_1}-\sigma^{(\boldsymbol{\mu})}_{1}(\alpha_1)|<2^{-\nu}$ for $\sigma^{(\boldsymbol{\mu})}_{1}(\alpha_1)$ as the speed of $\alpha_1$, then
\begin{eqnarray}\label{eq:4.28}
&&\|\mathcal{S}_{h}(U^{(\boldsymbol{\mu}),\nu}_b(\hat{x},\cdot))-U^{(\boldsymbol{\mu}),\nu}_b(\hat{x}+h,\cdot)\|_{L^1(b_0\hat{x}-\eta, b_{0}(\hat{x}+h))}\nonumber\\[3pt]
&&\,
\leq C_{b}(\|\boldsymbol{\mu}\|+2^{-\nu})(|\alpha_1-1|+1)h;
\end{eqnarray}

\item[(ii)] if $U^{(\boldsymbol{\mu})}_{b}$ and $U_{L}$ are connected by a $1$-rarefaction wave $\alpha_1\in R^{(\boldsymbol{\mu})}$
with $|\dot{y}_{\alpha_1}-\lambda^{(\boldsymbol{\mu})}_{1}(U^{(\boldsymbol{\mu})}_{b}, \boldsymbol{\mu})|<2^{-\nu}$
for $\lambda^{(\boldsymbol{\mu})}_{1}(U^{(\boldsymbol{\mu})}_{b}, \boldsymbol{\mu})$ as the speed of $\alpha_1$, then
\begin{eqnarray}\label{eq:4.29}
&&\|\mathcal{S}_{h}(U^{(\boldsymbol{\mu}),\nu}_b(\hat{x},\cdot))-U^{(\boldsymbol{\mu}),\nu}_b(\hat{x}+h,\cdot)\|_{L^1(b_0\hat{x}-\eta, b_{0}(\hat{x}+h))}\nonumber\\[3pt]
&&\,\leq C_{b}(\|\boldsymbol{\mu}\|+\nu^{-1}+2^{-\nu})(|\alpha_1-1|+1)h,
\end{eqnarray}
\end{enumerate}
where $\eta$ satisfies $\eta<-\hat{\lambda}h$, and $C_{b}>0$ is independent of $\boldsymbol{\mu}$, $\nu$ and $h$.
\end{lemma}

\begin{proof} We divide the proof into two steps accordingly.

\smallskip
1. We know that $U^{(\boldsymbol{\mu})}_{b}$, $U_{L}$, and $\alpha_1$ satisfy
\begin{eqnarray*}
U^{(\boldsymbol{\mu})}_{b}=\Phi^{(\boldsymbol{\mu})}_{1}(\alpha_1; U_{L},\boldsymbol{\mu}).
\end{eqnarray*}
$\mathcal{S}_{h}(U^{(\boldsymbol{\mu}),\nu}_{b}(\hat{x},y))$ near the boundary $\Gamma$ satisfies
\begin{eqnarray*}
U_{b}=\Phi_{1}(\beta_1; U_{L}).
\end{eqnarray*}
Then 
\begin{eqnarray}\label{eq:4.30}
v_{L}+\Phi^{(\boldsymbol{\mu}),(2)}_{1}(\alpha_1; U_{L},\boldsymbol{\mu})
=\sqrt{1-\tau^2B^{(\epsilon)}(\Phi^{(\boldsymbol{\mu}),(1)}_{1},\Phi^{(\boldsymbol{\mu}),(2)}_{1},\boldsymbol{\mu })}\big(\Phi^{(2)}_{1}(\beta_1; U_{L})+v_{L}\big),
\end{eqnarray}
where $\Phi^{(\boldsymbol{\mu}),(k)}_{1}$ is the $k$-th component of $\Phi^{(\boldsymbol{\mu})}_{1}$ for $k=1,2$,
and  $\Phi^{(2)}_{1}$ is the $2$-nd component of $\Phi_{1}$.
Thus, by Proposition \ref{prop:2.2}, when $\|\boldsymbol{\mu}\|\leq \|\boldsymbol{\mu}^{*}_0\|$, 
equation \eqref{eq:4.30} admits a unique solution $\beta_1=\beta_1(\alpha_1,\boldsymbol{\mu})\in C^2$ such that
\begin{eqnarray}\label{eq:4.31}
\beta_1=\alpha_1+O(1)\big(1+|\alpha_1-1|\big)\|\boldsymbol{\mu}\|.
\end{eqnarray}
Hence, $\beta_1$ is also a $1$-shock wave.

\begin{figure}[ht]
\begin{minipage}[t]{0.49\textwidth}
\centering
\begin{tikzpicture}[scale=1.0]
\draw [thick](-5.0,2.5)--(-2,2);
\draw [thin][red](-5.0,2.5)--(-2,0.8);
\draw [thin](-5.0,2.5)--(-2,-0.4);

\draw [line width=0.05cm](-5.0,2.5)--(-5.0,-1.5);
\draw [line width=0.05cm](-2,2)--(-2,-1.5);
\draw [line width=0.05cm][green](-2,2)--(-2,0.8);
\draw [line width=0.05cm][blue](-2,0.8)--(-2,-0.4);

\draw [thin](-5.0,2.5)--(-4.7, 2.8);
\draw [thin](-4.7,2.45)--(-4.4, 2.75);
\draw [thin](-4.4,2.40)--(-4.1, 2.70);
\draw [thin](-4.1,2.35)--(-3.8, 2.65);
\draw [thin](-3.8, 2.30)--(-3.5,2.60);
\draw [thin](-3.5, 2.25)--(-3.2,2.55);

\draw [thin](-3.2, 2.20)--(-2.9,2.50);
\draw [thin](-2.9, 2.15)--(-2.6,2.45);
\draw [thin](-2.6, 2.10)--(-2.3,2.40);
\draw [thin](-2.3, 2.05)--(-2.0,2.35);
\draw [thin](-2.0, 2.00)--(-1.7,2.30);

\node at (-1.6, 1.4) {$I_{b2}$};
\node at (-1.6, 0.1) {$I_{b1}$};
\node at (-1.6, 0.7) {$\alpha_{1}$};
\node at (-1.6, -0.7) {$\beta_{1}$};
\node at (-2.5, 1.6) {$U^{(\boldsymbol{\mu})}_{b}$};
\node at (-2.8, 0.8) {$U_{b}$};
\node at (-2.8, -0.4) {$U_{L}$};

\node at (1, 2) {$$};

\node at (-5.8, 2.5) {$(\hat{x},b_0\hat{x})$};
\node at (-5.0, -1.9) {$\hat{x}$};
\node at (-2.0, -1.9) {$\hat{x}+h$};

\end{tikzpicture}
\caption{Comparison of the Riemann solvers for $\alpha_1\in S^{(\boldsymbol{\mu})}_1$ and $\beta_1$ is a $1$-shock wave}\label{fig4.6}
\end{minipage}
\begin{minipage}[t]{0.49\textwidth}
\centering
\begin{tikzpicture}[scale=1.0]
\draw [thick](-5.0,2.5)--(-2,2);
\draw [thin](-5.0,2.5)--(-2,0.9);
\draw [thin](-5.0,2.5)--(-2,0.75);
\draw [thin](-5.0,2.5)--(-2,0.6);
\draw [thin](-5.0,2.5)--(-2,0.45);

\draw [thin][red](-5.0,2.5)--(-2,-0.6);

\draw [line width=0.05cm](-5.0,2.5)--(-5.0,-1.5);
\draw [line width=0.05cm](-2,2)--(-2,-1.5);
\draw [line width=0.05cm][green](-2,0.9)--(-2,0.45);
\draw [line width=0.05cm][blue](-2,0.45)--(-2,-0.6);

\draw [thin](-5.0,2.5)--(-4.7, 2.8);
\draw [thin](-4.7,2.45)--(-4.4, 2.75);
\draw [thin](-4.4,2.40)--(-4.1, 2.70);
\draw [thin](-4.1,2.35)--(-3.8, 2.65);
\draw [thin](-3.8, 2.30)--(-3.5,2.60);
\draw [thin](-3.5, 2.25)--(-3.2,2.55);
\draw [thin](-3.2, 2.20)--(-2.9,2.50);
\draw [thin](-2.9, 2.15)--(-2.6,2.45);
\draw [thin](-2.6, 2.10)--(-2.3,2.40);
\draw [thin](-2.3, 2.05)--(-2.0,2.35);
\draw [thin](-2.0, 2.00)--(-1.7,2.30);

\node at (-1.6, 1.4) {$I_{b3}$};
\node at (-1.6, 0.7) {$I_{b2}$};
\node at (-1.6, -0.1) {$I_{b1}$};
\node at (-2.25, 1.4) {$\beta_{1}$};
\node at (-1.6, -0.7) {$\alpha_{1}$};
\node at (-2.8, 1.75) {$U^{(\boldsymbol{\mu})}_{b}$};
\node at (-2.8, 0.6) {$U_{b}$};
\node at (-2.8, -0.4) {$U_{L}$};

\node at (1, 2) {$$};

\node at (-5.8, 2.5) {$(\hat{x},b_0\hat{x})$};
\node at (-5.0, -1.9) {$\hat{x}$};
\node at (-2.0, -1.9) {$\hat{x}+h$};
\end{tikzpicture}
\caption{Comparison of the Riemann solvers for $\alpha_1\in R^{(\boldsymbol{\mu})}_1$ and $\beta_1$ is a $1$-rarefaction wave}\label{fig4.7}
\end{minipage}
\end{figure}

To estimate $\|\mathcal{S}_{h}(U^{(\boldsymbol{\mu}),\nu}_{b}(\hat{x},\cdot))-U^{(\boldsymbol{\mu}),\nu}_{b}(\hat{h}+h,\cdot)\|_{L^1(b_0\hat{x}-\eta, b_0(\hat{x}+h))}$,
it suffices to consider it on intervals $I_{b1}$ and $I_{b2}$ as shown in Fig. \ref{fig4.6}.
Let $\sigma_{1}(\beta_1)$ be the speed of $\beta_1$. Then
\begin{eqnarray*}
\sigma_1(\beta_1)=\frac{\rho_{b}v_{b}-\rho_{L}v_{L}}{\rho_{b}-\rho_{L}}=\frac{\beta_1\varphi_{1}(\beta_1)}{\beta_1-1}+v_{L}.
\end{eqnarray*}
Note that
\begin{eqnarray*}
\begin{split}
\sigma^{(\boldsymbol{\mu})}_1(\alpha_1)&=\frac{\rho^{(\boldsymbol{\mu})}_{b}v^{(\boldsymbol{\mu})}_{b}-\rho_{L}v_{L}}{\rho^{(\boldsymbol{\mu})}_{b}
\sqrt{1-\tau^2B^{(\epsilon)}(\rho^{(\boldsymbol{\mu})}_{b},v^{(\boldsymbol{\mu})}_{b}, \epsilon )}-\rho_{L}\sqrt{1-\tau^2B^{(\epsilon)}(\rho_{L},v_{L}, \epsilon )}}\\[3pt]
&=\frac{\rho^{(\boldsymbol{\mu})}_{b}v^{(\boldsymbol{\mu})}_{b}-\rho_{L}v_{L}}{\rho^{(\boldsymbol{\mu})}_{b}-\rho_{L}}+O(1)\|\boldsymbol{\mu}\|\\[3pt]
&=\frac{\alpha_1\varphi^{(\boldsymbol{\mu  })}_{1}(\alpha_1; U_{L},\boldsymbol{\mu})}{\alpha_1-1}+v_{L}+O(1)\|\boldsymbol{\mu}\|.
\end{split}
\end{eqnarray*}
Then, by Lemma \ref{lem:2.5} and Proposition \ref{prop:3.1}, we have
\begin{eqnarray*}
\begin{split}
|I_{b1}|&\leq |\sigma^{(\boldsymbol{\mu})}_{1}(\alpha_1)-\sigma_{1}(\beta_1)|h+2^{-\nu}h\\[3pt]
&\leq \Big(\Big|\frac{\alpha_1\varphi^{(\boldsymbol{\mu  })}_{1}(\alpha_1; U_{L},\boldsymbol{\mu})}{\alpha_1-1}
  -\frac{\beta_1\varphi_{1}(\beta_1)}{\beta_1-1}\Big|+O(1)\|\boldsymbol{\mu}\|+2^{-\nu}\Big)h\\[3pt]
&\leq O(1)\big(\|\boldsymbol{\mu}\|+2^{-\nu}\big)h.
\end{split}
\end{eqnarray*}

On interval $I_{b1}$, it follows from \eqref{eq:4.31} and Proposition \ref{prop:3.1} that
\begin{eqnarray*}
\begin{split}
|\mathcal{S}_{h}(U^{(\boldsymbol{\mu}),\nu}_{b}(\hat{x}, y))-U^{(\boldsymbol{\mu}),\nu}_{b}(\hat{x}+h, y)|&=|\rho_{b}-\rho_{L}|+|v_{b}-v_{L}|\\[3pt]
&=O(1)|\beta_1-1|
=O(1)\big(|\alpha_1-1|+1\big),
\end{split}
\end{eqnarray*}
so that
\begin{eqnarray}\label{eq:4.32}
\|\mathcal{S}_{h}(U^{(\boldsymbol{\mu}),\nu}_{b}(\hat{x}, \cdot))-U^{(\boldsymbol{\mu}),\nu}_{b}(\hat{x}+h, \cdot)\|_{L^1(I_{b1})}
\leq O(1)\big(|\alpha_1-1|+1\big)\big(\|\boldsymbol{\mu}\|+2^{-\nu}\big)h.
\end{eqnarray}

Similarly, on $I_{b2}$, we have
\begin{eqnarray*}
\begin{split}
|\mathcal{S}_{h}(U^{(\boldsymbol{\mu}),\nu}_{b}(\hat{x}, y))-U^{(\boldsymbol{\mu}),\nu}_{b}(\hat{x}+h, y)|
&=|\rho^{(\boldsymbol{\mu})}_{b}-\rho_{b}|+|v^{(\boldsymbol{\mu})}_{b}-v_{b}|\\[3pt]
&=O(1)|\beta_1-\alpha_1|+O(1)\|\boldsymbol{\mu}\|\\[3pt]
&=O(1)\big(|\alpha_1-1|+1\big)\|\boldsymbol{\mu}\|,\\[3pt]
|I_{b2}|\leq |\sigma^{(\boldsymbol{\mu})}_{1}(\alpha_1)-b_0|h+2^{-\nu}h\leq (&O(1)+2^{-\nu})h,
\end{split}
\end{eqnarray*}
so that
\begin{eqnarray}\label{eq:4.33}
\|\mathcal{S}_{h}(U^{(\boldsymbol{\mu}),\nu}_{b}(\hat{x}, \cdot))-U^{(\boldsymbol{\mu}),\nu}_{b}(\hat{x}+h, \cdot)\|_{L^1(I_{b2})}
\leq O(1)\big(|\alpha_1-1|+1\big)\|\boldsymbol{\mu}\|h.
\end{eqnarray}

Finally, with estimates \eqref{eq:4.32}--\eqref{eq:4.33}, we arrive at
\begin{align*}
\begin{split}
& \|\mathcal{S}_{h}(U^{(\boldsymbol{\mu}),\nu}_{b}(\hat{x}, \cdot))-U^{(\boldsymbol{\mu}),\nu}_{b}(\hat{x}+h, \cdot)\|_{L^1(b_{0}\hat{x}-\eta, b_{0}(\hat{x}+h))}\\[3pt]
&
\leq \sum_{j=1,2}\|\mathcal{S}_{h}(U^{(\boldsymbol{\mu}),\nu}_{b}(\hat{x}, \cdot))-U^{(\boldsymbol{\mu}),\nu}_{b}(\hat{x}+h, \cdot)\|_{L^1(I_{bj})}\\[3pt]
&
\leq O(1)\big(|\alpha_1-1|+1\big)\big(\|\boldsymbol{\mu}\|+2^{-\nu}\big)h,
\end{split}
\end{align*}
which gives estimate \eqref{eq:4.28}.

\smallskip
2. In this case, we know that relation \eqref{eq:4.30} and estimate \eqref{eq:4.31} hold.
Then $\beta_1$ is a rarefaction wave, and
\begin{eqnarray}\label{eq:4.34}
\mathcal{S}_{h}(U^{(\boldsymbol{\mu}),\nu}(\hat{x},y))=
\begin{cases}
U_{b}&\qquad \mbox{for $\xi\in(\lambda_{1}(U_{b}),b_0)$},\\[3pt]
\Phi_{1}(\beta_{1}(\xi);U_{L})&\qquad\mbox{for $\xi\in(\lambda_{1}(U_{L}),\lambda_{1}(U_{b})]$},\\[3pt]
U_{L},&\qquad\mbox{for $\xi\in(-\frac{\eta}{h}, \lambda_1(U_L)]$},
\end{cases}
\end{eqnarray}
where $\xi=\frac{y-b_{0}\hat{x}}{h}$, $\beta_1(\lambda_1(U_{L}))=1$, and $\beta_{1}(\lambda_1(U_{b}))=\beta_1$.

To show estimate \eqref{eq:4.29}, we only consider the case as shown in Fig. \ref{fig4.7}, since the other case can be treated in the same way.
Interval $(b_{0}\hat{x}-\eta, b_{0}(\hat{x}+h))$ is divided into three subintervals, \emph{i.e.}, $I_{bj}$ with $1\leq j\leq 3$.
Using \eqref{eq:2.41}, Proposition \ref{prop:3.1}, and Lemma \ref{lem:2.5}, we have
\begin{align*}
\begin{split}
|I_{b1}|&\leq \big(|\lambda^{(\boldsymbol{\mu})}_{1}(U^{(\boldsymbol{\mu})}_{b},\boldsymbol{\mu})-\lambda_1(U_{L})|+2^{-\nu}\big)h\\[3pt]
&\leq O(1)\big(\|\boldsymbol{\mu}\|+|\alpha_1-1|+2^{-\nu}\big)h\\[3pt]
&\leq O(1)\big(\|\boldsymbol{\mu}\|+\nu^{-1}+2^{-\nu}\big)h,
\end{split}
\end{align*}
and, on interval $I_{b1}$,
\begin{align*}
|\mathcal{S}_{h}(U^{(\boldsymbol{\mu}),\nu}_{b}(\hat{x}, y))-U^{(\boldsymbol{\mu}),\nu}_{b}(\hat{x}+h, y)|
=|\rho^{(\boldsymbol{\mu})}_{b}-\rho_{L}|+|v^{(\boldsymbol{\mu})}_{b}-v_{L}|
\leq O(1)|\alpha_1-1|,
\end{align*}
so that
\begin{align}\label{eq:4.35}
& \|\mathcal{S}_{h}(U^{(\boldsymbol{\mu}),\nu}_{b}(\hat{x}, \cdot))-U^{(\boldsymbol{\mu}),\nu}_{b}(\hat{x}+h, \cdot)\|_{L^1(I_{b1})}
\leq O(1)|\alpha_1-1|\big(\|\boldsymbol{\mu}\|+\nu^{-1}+2^{-\nu}\big)h.
\end{align}

\smallskip
Next, it follows from \eqref{eq:4.31} and \eqref{eq:4.34} that
\begin{align*}
|I_{b2}|&=|\lambda_1(U_{b})-\lambda_1(U_{L})|h=O(1)|U_{b}-U_{L}|h=O(1)|\alpha_1-1|h,
\end{align*}
and, on $I_{b2}$,
\begin{align*}
|\mathcal{S}_{h}(U^{(\boldsymbol{\mu}),\nu}_{b}(\hat{x}, y))-U^{(\boldsymbol{\mu}),\nu}_{b}(\hat{x}+h, y)|&=|\Phi_{1}(\beta_1(\xi); U_{L})-U^{(\boldsymbol{\mu})}_{b}|\\[3pt]
&\leq |\Phi_{1}(\beta_1(\xi); U_{L})-U_{L}|+|U_{L}-U^{(\boldsymbol{\mu})}_{b}|\\[3pt]
&\leq O(1)|\xi_0-\xi_{1}|+O(1)|\alpha_1-1|\\[3pt]
&\leq O(1)|\alpha_1-1|,
\end{align*}
so that
\begin{align}\label{eq:4.36}
\|\mathcal{S}_{h}(U^{(\boldsymbol{\mu}),\nu}_{b}(\hat{x}, \cdot))-U^{(\boldsymbol{\mu}),\nu}_{b}(\hat{x}+h, \cdot)\|_{L^1(I_{b2})}
\leq O(1)|\alpha_1-1|^{2}h \leq O(1)|\alpha_1-1| \nu^{-1} h.
\end{align}

On interval $I_{b3}$, by Lemma \ref{lem:2.6} and estimate \eqref{eq:4.31}, we obtain
\begin{align*}
\begin{split}
|\mathcal{S}_{h}(U^{(\boldsymbol{\mu}),\nu}_{b}(\hat{x}, y))-U^{(\boldsymbol{\mu}),\nu}_{b}(\hat{x}+h, y)|
&=|\rho^{(\boldsymbol{\mu})}_{b}-\rho_{b}|+|v^{(\boldsymbol{\mu})}_{b}-v_{b}|\\[3pt]
&\leq O(1)\big(|\beta_1-\alpha_1|+\|\boldsymbol{\mu}\|\big)\\[3pt]
&\leq O(1)\big(|\alpha_1-1|+1\big)\|\boldsymbol{\mu}\|,
\end{split}
\end{align*}
and $|I_{b3}|=|\lambda_1(U_{b})-b_{0}|h\leq O(1)h$, so that
\begin{align}\label{eq:4.37}
\|\mathcal{S}_{h}(U^{(\boldsymbol{\mu}),\nu}_{b}(\hat{x}, \cdot))-U^{(\boldsymbol{\mu}),\nu}_{b}(\hat{x}+h, \cdot)\|_{L^1(I_{b3})}
&=|\rho^{(\boldsymbol{\mu})}_{b}-\rho_{b}|+|v^{(\boldsymbol{\mu})}_{b}-v_{b}|\nonumber\\[3pt]
&\leq O(1)\big(|\beta_1-\alpha_1|+\|\boldsymbol{\mu}\|\big)\nonumber\\[3pt]
&\leq O(1)\big(|\alpha_1-1|+1\big)\|\boldsymbol{\mu}\|h.
\end{align}

Finally, combining estimates \eqref{eq:4.35}--\eqref{eq:4.37}, we obtain
\begin{align*}
\begin{split}
&\|\mathcal{S}_{h}(U^{(\boldsymbol{\mu}),\nu}_{b}(\hat{x}, \cdot))-U^{(\boldsymbol{\mu}),\nu}_{b}(\hat{x}+h, \cdot)\|_{L^1(b_{0}\hat{x}-\eta, b_{0}(\hat{x}+h))}\\[3pt]
&\leq \sum^{3}_{j=1}\|\mathcal{S}_{h}(U^{(\boldsymbol{\mu}),\nu}_{b}(\hat{x}, \cdot))-U^{(\boldsymbol{\mu}),\nu}_{b}(\hat{x}+h, \cdot)\|_{L^1(I_{bj})}\\[3pt]
&\leq O(1)\big(|\alpha_1-1|+1\big)\big(\|\boldsymbol{\mu}\|+\nu^{-1}+2^{-\nu}\big)h.
\end{split}
\end{align*}
This completes the proof.
\end{proof}

Next, we consider the comparison between the Riemann solutions with a boundary but away from the reflection points.
Using Proposition \ref{prop:2.3} and following the procedure of the proof
in Lemma \ref{lem:4.1}, we have

\begin{lemma}\label{lem:4.3}
Let $U^{(\boldsymbol{\mu}),\nu}_{b}(\hat{x},y)$ be a piecewise constant function defined by \eqref{eq:4.25} with $y\neq b_{0}\hat{x}$.
Let $\mathcal{S}$ be a uniform Lipschitz continuous semigroup obtained by {\rm Proposition \ref{eq:3.2}}. Define
\begin{equation}\label{eq:4.38}
U^{(\boldsymbol{\mu}),\nu}_{b}(\hat{x}+h,y):=
\begin{cases}
U_{b}\quad &\mbox{for $b_0\hat{x}+\dot{y}_{\alpha_1}h<y<b_{0}(\hat{x}+h)$},\\[3pt]
U_{L}\quad &\mbox{for $y<b_0\hat{x}+\dot{y}_{\alpha_1}h$},
\end{cases}
\end{equation}
where $\dot{y}_{\alpha_1}\in (-\hat{\lambda}, b_0)$, and $U_{b}$ is defined by \eqref{eq:4.23} with 
\begin{equation}\label{eq:4.39}
v_{b}=b_0\sqrt{1-\tau^2B^{(\epsilon)}(\rho_b, v_b, \epsilon)}.
\end{equation}
For $\|\boldsymbol{\mu}\|\leq \|\bar{\boldsymbol{\mu}}^{*}_0\|$ with $\boldsymbol{\bar{\mu}}^{*}_{0}$ given in {\rm Proposition \ref{prop:3.1}}
and for sufficiently small $h>0$,

\begin{enumerate}
\item[(i)]
if $U_{b}$ and $U_{L}$ are connected by a $1$-shock wave $\alpha_1\in S^{(\boldsymbol{\mu})}$
with $|\dot{y}_{\alpha_1}-\sigma^{(\boldsymbol{\mu})}_{1}(\alpha_1)|<2^{-\nu}$
for $\sigma^{(\boldsymbol{\mu})}_{1}(\alpha_1)$ as the exact speed of $\alpha_1$,
then
\begin{eqnarray}\label{eq:4.40}
&&
\|\mathcal{S}_{h}(U^{(\boldsymbol{\mu}),\nu}_{b}(\hat{x},\cdot))-U^{(\boldsymbol{\mu}),\nu}_{b}(\hat{x}+h,\cdot)\|_{L^1(b_0\hat{x}-\eta, b_{0}(\hat{x}+h))}
\nonumber\\[3pt]
&&\,
\leq C_{bb}(\|\boldsymbol{\mu}\|+2^{-\nu})(|\alpha_1-1|+1)h;
\end{eqnarray}

\item[(ii)] if $U_{b}$ and $U_{L}$ are connected by a $1$-rarefaction wave front $\alpha_1\in R^{(\boldsymbol{\mu})}$
with $|\dot{y}_{\alpha_1}-\lambda^{(\boldsymbol{\mu})}_{1}(U^{(\boldsymbol{\mu})}_{b}, \boldsymbol{\mu})|<2^{-\nu}$
for $\lambda^{(\boldsymbol{\mu})}_{1}(U^{(\boldsymbol{\mu})}_{b}, \boldsymbol{\mu})$ as the exact speed of $\alpha_1$, then
\begin{eqnarray}\label{eq:4.41}
&&\|\mathcal{S}_{h}(U^{(\boldsymbol{\mu}),\nu}_{b}(\hat{x},\cdot))-U^{(\boldsymbol{\mu}),\nu}_{b}(\hat{x}+h,\cdot)\|_{L^1(b_0\hat{x}-\eta, b_{0}(\hat{x}+h))}
\nonumber\\[3pt]
&&\,
\leq C_{bb}(\|\boldsymbol{\mu}\|+\nu^{-1}+2^{-\nu})(|\alpha_1-1|+1)h,
\end{eqnarray}
\end{enumerate}
where $\eta<-\hat{\lambda}h$ and constant $C_{bb}>0$ is independent on $\boldsymbol{\mu}$, $\nu$ and $h$.
\end{lemma}

\subsection{Proof of Theorem \ref{thm:1.1} for the convergence rate estimate \eqref{eq:1.22}}
Now, we are ready to prove estimate \eqref{eq:1.22} that is completed by the following two steps:

\smallskip
1. \emph{Estimate for $\|\mathcal{S}_{h}(U^{(\boldsymbol{\mu}),\nu}(\hat{x}, \cdot))-U^{(\boldsymbol{\mu}),\nu}(\hat{x}+h, \cdot)\|_{L^1(-\infty,\,b_{0}(\hat{x}+h))}$.}
Let $b_0\hat{x}=y_0>y_1>\cdot\cdot\cdot >y_{N}$ (or $b_0\hat{x}>y_0>y_1>\cdot\cdot\cdot >y_{N}$) be the jumps of $U^{(\boldsymbol{\mu}),\nu}(\hat{x},\cdot)$
on line $x=\hat{x}$.
Suppose that there is no wave interaction on the stripe between $x=\hat{x}$ and $x=\hat{x}+h$,
and there is no reflection on boundary $(x,b_{0}x)$ for $x\in (\hat{x},\hat{x}+h)$.
Let $S^{(\boldsymbol{\mu})}$ (or $R^{(\boldsymbol{\mu})}$) be the set of indices $\alpha_i$ with $i\in \{1,2,\cdots, N \}$
such that $U^{(\boldsymbol{\mu}),\nu}(\hat{x},y_{\alpha}+)$ and $U^{(\boldsymbol{\mu}),\nu}(\hat{x},y_{\alpha}-)$ are connected by a shock
wave front (or a rarefaction wave front) with strength $\alpha_i$.
Let $NP^{(\boldsymbol{\mu})}$ be the set of indices $\alpha_{NP}$ such that $U^{(\boldsymbol{\mu}),\nu}(\hat{x},y_{\alpha}+)$
and $U^{(\boldsymbol{\mu}),\nu}(\hat{x},y_{\alpha}-)$ are connected by a non-physical front with strength $\alpha_{NP}$.

By Lemmas \ref{lem:4.1}--\ref{lem:4.3}
and Proposition \ref{prop:3.1}, for sufficiently small $h>0$, if $\|\boldsymbol{\mu}\|\leq \|\boldsymbol{\mu}^{*}_{0}\|$, then
\begin{align}\label{eq:4.42}
&\|\mathcal{S}_{h}(U^{(\boldsymbol{\mu}),\nu}(\hat{x},\cdot))-U^{(\boldsymbol{\mu}),\nu}(\hat{x}+h,\cdot)\|_{L^1(-\infty,\,b_0(\hat{x}+h))}\nonumber\\[3pt]
& \leq\sum_{\alpha\in S^{(\boldsymbol{\mu})}\cup R^{(\boldsymbol{\mu})}\cup NP^{(\boldsymbol{\mu})}}
\|\mathcal{S}_{h}(U^{(\boldsymbol{\mu}),\nu}(\hat{x},y))-U^{(\boldsymbol{\mu}),\nu}(\hat{x}+h,y)\|_{L^1(y_{\alpha}-\eta,\,y_{\alpha}+\eta )}\nonumber\\[3pt]
& \leq C\big(\|\boldsymbol{\mu}\|+2^{-\nu}+\nu^{-1}\big)\Big(\sum_{\alpha\in S^{(\boldsymbol{\mu})}\cup R^{(\boldsymbol{\mu})}}|\alpha-1|+1\Big)h
+O(1)\Big(\sum_{\alpha_{NP}\in NP^{(\boldsymbol{\mu})}}\alpha_{NP}\Big)h\nonumber\\[3pt]
& \leq C\big(\|\boldsymbol{\mu}\|+2^{-\nu}+\nu^{-1}\big)\big(T.V.\{U^{(\boldsymbol{\nu}),\nu}(\hat{x},\cdot)\}+1\big)h
+O(1)2^{-\nu}h\nonumber\\[3pt]
&\leq C\big(\|\boldsymbol{\mu}\|+2^{-\nu}+\nu^{-1}\big)h,
\end{align}
where $C$ is independent of $(\mu,h)$, and $\eta=\frac{1}{2}\min_{1\leq j\leq N}\{y_{j-1}-y_{j}\}$.

\smallskip
2. \emph{Estimate on $\|U^{(\boldsymbol{\mu})}(x,\cdot)-U(x,\cdot)\|_{L^1(-\infty,\,b_{0}x)}$.} Let $U^{(\boldsymbol{\mu}),\nu}(x,y)$
be an approximate solution of \eqref{eq:1.15}--\eqref{eq:1.17} with initial data $U^{\nu}_{0}$ satisfying \eqref{eq:3.1}.
Let $\mathcal{S}$ be the uniformly Lipschitz semigroup given by Proposition \ref{prop:3.2}.
Then, by the triangle inequality, we have
\begin{align}\label{eq:4.43}
 \|U^{(\boldsymbol{\mu})}(x,\cdot)-U(x,\cdot)\|_{L^1(-\infty, b_{0}x)}
 &\leq \|U^{(\boldsymbol{\mu})}(x,\cdot)-U^{(\boldsymbol{\mu}),\nu}(x,\cdot)\|_{L^1(-\infty,\,b_{0}x)}\nonumber\\[3pt]
&\quad +\|U^{(\boldsymbol{\mu}),\nu}(x,\cdot)-\mathcal{S}_{x}(U^{\nu}_{0}(\cdot))\|_{L^1(-\infty,\,b_{0}x)}\nonumber\\[3pt]
&\quad +\|\mathcal{S}_{x}(U^{\nu}_{0}(\cdot))-U(x,\cdot)\|_{L^1(-\infty,\,b_{0}x)}\nonumber\\[3pt]
&=:J_1+J_2+J_3.
\end{align}

For $J_1$, by Proposition \ref{prop:3.1}, we can choose a subsequence (still denoted as) $\{U^{(\boldsymbol{\mu}),\nu}\}_{\nu}$
such that $U^{(\boldsymbol{\mu}),\nu}\rightarrow U^{(\boldsymbol{\mu})}$
in $L^{1}_{\rm loc}(\Omega)$ as $\nu\rightarrow \infty$.
Then $J_1\rightarrow0$ as $\nu\rightarrow \infty$.

Next, for $J_{3}$, by the Lipschitz property of $\mathcal{S}$ and \eqref{eq:3.1}, we have
\begin{align}\label{eq:4.44}
J_3&\leq \|\mathcal{S}_{x}(U^{\nu}_{0}(\cdot))-\mathcal{S}_{x}(U_0(\cdot))\|_{L^1(-\infty, b_{0}x)}\nonumber\\[3pt]
&\leq L\|U^{\nu}_{0}(\cdot)-U_0(\cdot)\|_{L^1(-\infty, b_{0}x)}\rightarrow 0 \qquad \mbox{as $\nu\rightarrow \infty$}.
\end{align}

For $J_{2}$, thanks to Proposition \ref{prop:3.3} and \emph{Step 1}, we obtain that, when $\|\boldsymbol{\mu}\|\leq \|\boldsymbol{\mu}^{*}_0\|$,
\begin{align}\label{eq:4.45}
J_2
&\leq L\int^{x}_{0}\liminf_{h\rightarrow0+}\frac{\|\mathcal{S}_{h}(U^{(\boldsymbol{\mu}),\nu}(\hat{x},\cdot))-U^{(\boldsymbol{\mu}),\nu}(\hat{x}+h,\cdot)\|_{L^1(-\infty,\,b_0(\hat{x}+h))}}{h}
  \,{\rm d}\hat{x}\nonumber\\[3pt]
&\leq O(1)x\big(2^{-\nu}+\nu^{-1}+\|\boldsymbol{\mu}\|\big).
\end{align}
Then it follows from \eqref{eq:4.42}--\eqref{eq:4.45} that we can choose a constant vector $\boldsymbol{\mu}_0=(\epsilon_0,\tau^2_{0})$
with $\epsilon_0>0, \tau_0>0$, and a constant $C_1>0$,
independent of $(\boldsymbol{\mu},\nu, x)$ such that, as $\nu\rightarrow \infty$,
\begin{eqnarray*}
\|U^{(\boldsymbol{\mu})}(x,\cdot)-U(x,\cdot)\|_{L^{1}(-\infty, b_{0}x)}\leq C_1x\|\boldsymbol{\mu}\|,
\end{eqnarray*}
which gives estimate \eqref{eq:1.22}.

\vspace{-3mm}
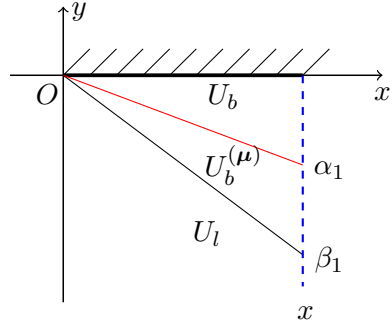
\begin{figure}[ht]
\begin{center}
\begin{tikzpicture}[scale=0.7]
\draw [line width=0.02cm][->] (-5,0) --(2.0,0);
\draw [line width=0.02cm][->] (-4.0,-4.3) --(-4.0,1.3);

\draw [line width=0.05cm](-4.0,0)--(0.5, 0);
\draw [thin](-4,0)--(0.5, -3.4);
\draw [thin][red](-4,0)--(0.5, -1.7);
\draw [thick][dashed][blue](0.5,0)--(0.5,-4.0);

\draw [thin](-4,0)--(-3.5, 0.5);
\draw [thin](-3.5,0)--(-3.0, 0.5);
\draw [thin](-3.0,0)--(-2.5, 0.5);
\draw [thin](-2.5,0)--(-2.0, 0.5);
\draw [thin](-2.0,0)--(-1.5, 0.5);
\draw [thin](-1.5,0)--(-1.0, 0.5);
\draw [thin](-1.0,0)--(-0.5, 0.5);
\draw [thin](-0.5,0)--(0.0, 0.5);
\draw [thin](0.0,0)--(0.5, 0.5);
\draw [thin](0.5,0)--(1.0, 0.5);

\node at (2.6, 2) {$$};
\node at (0.55, -4.5) {$x$};
\node at (1.0, -1.8) {$\alpha_1$};
\node at (1.0, -3.5) {$\beta_{1}$};

\node at (2.0, -0.35){$x$};
\node at (-3.7, 1.2){$y$};
\node at (-4.3, -0.35){$O$};
\node at (-0.8, -1.7){$U^{(\boldsymbol{\mu})}_{b}$};
\node at (-1.0, -0.4){$U_{b}$};
\node at (-1.3, -3.0){$U_{l}$};

\end{tikzpicture}
\end{center}
\caption{An example for the optimal convergence rate}\label{fig4.8}
\end{figure}

\subsection{Proof of Theorem \ref{thm:1.1} for the optimal convergence rate \eqref{eq:1.22}}
In this subsection, we further show that our convergence rate with respect to $\boldsymbol{\mu}$ in estimate \eqref{eq:1.22} is optimal.
To achieve this, it suffices to calculate an accurate convergence rate of a special Riemann solution in the following:
As shown in Fig.~\ref{fig4.8}, $\,b_{0}\equiv0$. We consider the Riemann problem for system \eqref{eq:1.19} with the following Riemann data:
\begin{eqnarray}\label{eq:4.46}
U|_{x=0}=\begin{cases}
v_{b}=0 \qquad  &\mbox{for $y=0$},\\[3pt]
U_{l}=(\rho_{l}, v_{l})^{\top} \qquad  &\mbox{for $y<0$},
\end{cases}
\end{eqnarray}
where $\rho_{l}=1$ and $v_{l}=\delta>0$.
Set  $\underline{U}:=U_{l}\big|_{\delta=0}=(1,0)^{\top}$.
Then, by \eqref{eq:2.42}, the following relation holds:
\begin{eqnarray}\label{eq:4.47}
0=\varphi_{1}(\alpha_1)+\delta.
\end{eqnarray}

For $\delta=0$, by \eqref{eq:4.47} and Lemma \ref{lem:2.6}, we see that $\alpha_1=1$.
Since $\varphi'_1(1)=-a_{\infty}^{-1}<0$ from Lemma \ref{lem:2.6}, by the implicit function theorem,
equation \eqref{eq:4.47} admits a unique solution $\alpha_1=\alpha_1(\delta)$ for $\delta>0$ sufficiently small.
Moreover, by direct computation, we obtain from \eqref{eq:4.47} that
\begin{eqnarray}\label{eq:4.48}
\alpha'_{1}(0)=-\frac{1}{\varphi'_1(1)}=a_{\infty}.
\end{eqnarray}
Therefore, by the Taylor formula, we have
\begin{eqnarray}\label{eq:4.49}
\alpha_1(\delta)=\alpha_{1}(0)+\alpha'_{1}(0)\delta+O(1)\delta^{2}=1+a_{\infty}\delta+O(1)\delta^{2},
\end{eqnarray}
where the bound of $O(1)$ depends only on $\underline{U}$. It implies that $\alpha_1$ is a shock wave.

Denote by $\rho_{b}$ the density on boundary $y=0$. Then, by \eqref{eq:2.42}, $\rho_{b}$ satisfies
\begin{eqnarray}\label{eq:4.50}
\rho_{b}=\alpha_1\rho_{l}=1+a_{\infty}\delta+O(1)\delta^{2}.
\end{eqnarray}
Therefore, the Riemann problem \eqref{eq:1.19} and \eqref{eq:4.46} admits a unique solution that consists of only one shock wave $\alpha_1$
issuing from point $(0, 0)$ and belonging to the $1$-st family with $U_l=(1,\delta)$ and $U_{b}:=(\rho_{b}, v_b)^{\top}$
as its \emph{left}-state and \emph{right}-state for some $\delta>0$.

Now, we turn to the Riemann solution of problem \eqref{eq:1.15} with $U_l=(1,\delta)^{\top}$ as the \emph{left}-state
and $v^{(\boldsymbol{\mu})}_b$ as the velocity on the boundary.
Let $\beta_1$ be the elementary wave in the Riemann solution. Then
\begin{eqnarray}\label{eq:4.51}
v^{(\boldsymbol{\mu})}_b-v_l=\varphi^{(\boldsymbol{\mu})}_{1}(\beta_1, U_l,\boldsymbol{\mu}),\qquad \rho^{(\boldsymbol{\mu})}_b=\rho_l\beta_1.
\end{eqnarray}
It follows from $b_0=0$ and the boundary condition \eqref{eq:4.27} in Lemma \ref{lem:4.2} that $v^{(\boldsymbol{\mu})}_b\equiv0$.
Thus, combining \eqref{eq:4.47} with \eqref{eq:4.51} yields
\begin{eqnarray}\label{eq:4.52}
\varphi^{(\boldsymbol{\mu})}_{1}(\beta_1, U_l,\boldsymbol{\mu})=\varphi_{1}(\alpha_1).
\end{eqnarray}

By Lemma \ref{lem:2.6}, we have
$$
\frac{\partial \varphi^{(\boldsymbol{\mu})}_{1}(\beta_1, U_l,\boldsymbol{\mu})}{\partial\beta_1}\big|_{\beta_1=1,\boldsymbol{\mu}=\boldsymbol{0}}=-a_{\infty}^{-1}<0.
$$
Thus, by applying the implicit function theorem again, equations \eqref{eq:4.51} admit a unique solution
$\beta_1$ that is a function of $(\alpha_1,\boldsymbol{\mu})$, \emph{i.e.}, $\beta_1=\beta_{1}(\alpha_1,\boldsymbol{\mu} )$.
Let $\rho^{(\boldsymbol{\mu})}_b$ be the density state on boundary $y=0$.
Then, by \eqref{eq:2.38}, $\rho^{(\boldsymbol{\mu})}_b$ satisfies $\rho^{(\boldsymbol{\mu})}_b=\beta_1\rho_{l}=\delta\beta_1$.

Moreover, by \eqref{eq:4.51}, we see that
\begin{eqnarray}\label{eq:4.52b}
\beta_1(1,\boldsymbol{\mu})=0, \qquad  \beta_1(\alpha_{1},\boldsymbol{0})=\alpha_1.
\end{eqnarray}
Then, by the Taylor formula and \eqref{eq:4.49}, we have
\begin{eqnarray}\label{4.53}
\begin{split}
\beta_1(\alpha_1,\boldsymbol{\mu})&=\beta_1(\alpha_{1},\boldsymbol{0})+\beta_1(1,\boldsymbol{\mu})+O(1)|\alpha_{1}-1|\|\boldsymbol{\mu}\|\\[3pt]
&=\alpha_{1}+O(1)|\alpha_{1}-1|\|\boldsymbol{\mu}\|
=1+a_{\infty}\delta+O(1)\|\boldsymbol{\mu}\|\delta,
\end{split}
\end{eqnarray}
where $\|\boldsymbol{\mu}\|=\epsilon+\tau^2$.
Thus, for $\delta>0$ sufficiently small, the Riemann problem \eqref{eq:1.15} with Riemann data $U_l=(1,\delta)^{\top}$ and $v^{(\boldsymbol{\mu})}_b=0$
admits a unique solution that also consists of only one shock wave $\beta_1$ issuing from point $(0, 0)$ and belonging to the $1$-st family with $U_l=(1,\delta)$
and $U^{(\boldsymbol{\mu})}_{b}:=(\rho^{(\boldsymbol{\mu})}_{b}, v^{(\boldsymbol{\mu})}_b)^{\top}$
as its \emph{left}-state and \emph{right}-state for some $\delta>0$. Moreover, $\beta_1=1$ is equivalent to $\delta=0$.

Next, we compute $\frac{\partial^2 \beta_{1}}{\partial \alpha_1\partial \epsilon}\Big|_{\alpha_1=1, \boldsymbol{\mu}=\boldsymbol{0}}$
and $\frac{\partial^2 \beta_{1}}{\partial \alpha_1\partial \tau^2}\Big|_{\alpha_1=1, \boldsymbol{\mu}=\boldsymbol{0}}$.
We first take the derivative on \eqref{eq:4.52b} with respect to $\alpha_1$ to deduce that
\begin{eqnarray}\label{eq:4.54}
\frac{\partial\varphi^{(\boldsymbol{\mu})}_{1}(\beta_1, U_l,\boldsymbol{\mu})}{\partial \beta_1}\frac{\partial \beta_1}{\partial \alpha_1}
=\varphi'_{1}(\alpha_1).
\end{eqnarray}
Taking $\alpha_1=1$ and $\boldsymbol{\mu}=\boldsymbol{0}$ in \eqref{eq:4.54}, by Lemma \ref{lem:2.6}, we obtain
\begin{eqnarray}\label{eq:4.55}
\frac{\partial \beta_1}{\partial \alpha_1}\Big|_{\alpha_1=1, \boldsymbol{\mu}=\boldsymbol{0}}=\frac{\varphi'_{1}(1)}
{\frac{\partial\varphi^{(\boldsymbol{\mu})}_{1}(\beta_1, U_l,\boldsymbol{\mu})}{\partial \beta_1}\Big|_{\alpha_1=1, \boldsymbol{\mu}=\boldsymbol{0}}}=1.
\end{eqnarray}

Based on this, we further take the derivatives on both sides of \eqref{eq:4.54} with respect to $\epsilon$ and $\tau^2$ to obtain
\begin{eqnarray*}
&&\frac{\partial\varphi^{(\boldsymbol{\mu})}_{1}}{\partial \beta_1}\frac{\partial^2 \beta_1}{\partial \alpha_1\partial\epsilon}
+\frac{\partial^2\varphi^{(\boldsymbol{\mu})}_{1}}{\partial \beta^2_1}\frac{\partial \beta_1}{\partial\epsilon}\frac{\partial \beta_1}{\partial \alpha_1}
+\frac{\partial^2\varphi^{(\boldsymbol{\mu})}_{1}}{\partial \beta_1\partial \epsilon}\frac{\partial \beta_1}{\partial \alpha_1}
=0,\\[3pt]
&&\frac{\partial\varphi^{(\boldsymbol{\mu})}_{1}}{\partial \beta_1}\frac{\partial^2 \beta_1}{\partial \alpha_1\partial\tau^2}
+\frac{\partial^2\varphi^{(\boldsymbol{\mu})}_{1}}{\partial \beta^2_1}\frac{\partial \beta_1}{\partial\tau^2}\frac{\partial \beta_1}{\partial \alpha_1}
+\frac{\partial^2\varphi^{(\boldsymbol{\mu})}_{1}}{\partial \beta_1\partial \tau^2}\frac{\partial \beta_1}{\partial \alpha_1}
=0.
\end{eqnarray*}
These imply that
\begin{eqnarray}\label{eq:4.57}
\begin{split}
\frac{\partial^2 \beta_1}{\partial \alpha_1\partial\epsilon}=-\frac{\Big(\frac{\partial^2\varphi^{(\boldsymbol{\mu})}_{1}}{\partial \beta^2_1}\frac{\partial \beta_1}{\partial\epsilon}
+\frac{\partial^2\varphi^{(\boldsymbol{\mu})}_{1}}{\partial \beta_1\partial \epsilon}\Big)\frac{\partial \beta_1}{\partial \alpha_1}}{\frac{\partial\varphi^{(\boldsymbol{\mu})}_{1}}{\partial \beta_1}},
\qquad
\frac{\partial^2 \beta_1}{\partial \alpha_1\partial\tau^2}=-\frac{\Big(\frac{\partial^2\varphi^{(\boldsymbol{\mu})}_{1}}{\partial \beta^2_1}\frac{\partial \beta_1}{\partial\tau^2}
+\frac{\partial^2\varphi^{(\boldsymbol{\mu})}_{1}}{\partial \beta_1\partial \tau^2}\Big)\frac{\partial \beta_1}{\partial \alpha_1}}{\frac{\partial\varphi^{(\boldsymbol{\mu})}_{1}}{\partial \beta_1}}.
\end{split}
\end{eqnarray}

In the following, we are devoted to the estimates of all the terms on both the right-hand sides
of $\frac{\partial^2 \beta_1}{\partial \alpha_1\partial \epsilon}$ and $\frac{\partial^2 \beta_1}{\partial \alpha_1\partial \tau^2}$ in \eqref{eq:4.57}, respectively.
Since $\beta_1$ is a shock wave, states $v^{(\boldsymbol{\mu})}-v_{l}$ and $U_l$ satisfy equation \eqref{eq:2.14b}:
\begin{align}\label{eq:4.58}
\mathcal{H}^{(\boldsymbol{\mu})}_{S}:=&\, (\varphi^{(\boldsymbol{\mu})}_{1})^2-\frac{2(\beta^{\epsilon}_1-1)(\beta_1-1)}{a^2_{\infty}\epsilon (\beta_1+1)}\nonumber\\[3pt]
&\,-\Big(\sqrt{(1-\tau^2\delta^2)(1-\tau^2\mathcal{B}^{(\epsilon)})}-1\Big)
\Big((\delta+\varphi^{(\boldsymbol{\mu})}_{1})\delta+\frac{2\beta_1\big(\beta^{\epsilon}_1-1\big)}{a^{2}_{\infty}\epsilon(\beta_1+1)}\Big)
-\tau^2\delta^2\mathcal{B}^{(\epsilon)}\nonumber\\[3pt]
=&\,0,
\end{align}
where $\mathcal{B}^{(\epsilon)}$ is given by
\begin{eqnarray}\label{eq:4.59}
\mathcal{B}^{(\epsilon)}=\frac{2(\beta^{\epsilon}_1-1)}{a^{2}_{\infty}\epsilon}+\big(\delta+\varphi^{(\boldsymbol{\mu})}_{1}\big)^2.
\end{eqnarray}

Using the Taylor formula and estimates (i)--(ii) in Lemma \ref{lem:A1} and applying Lemma \ref{lem:A2}, we obtain
\begin{align}\label{eq:4.60}
\begin{split}
\frac{\partial \mathcal{H}^{(\boldsymbol{\mu})}_{S}}{\partial \varphi^{(\boldsymbol{\mu})}_{1}}
&=2\varphi^{(\boldsymbol{\mu})}_{1}-\Big(\big(1-\tau^2\delta^2\big)^{\frac{1}{2}}-(1-\tau^2\mathcal{B}^{(\epsilon)})^{-\frac{1}{2}}\Big)\delta\\[3pt]
&\quad+2(\delta+\varphi^{(\boldsymbol{\mu})}_{1})\Big(\frac{1-\tau^2\delta^2}{1-\tau^2\mathcal{B}^{(\epsilon)}}\Big)^{\frac{1}{2}}
\Big((\delta+\varphi^{(\boldsymbol{\mu})}_{1})\delta+\frac{2\beta_1\big(\beta^{\epsilon}_1-1\big)}{a^{2}_{\infty}\epsilon(\beta_1+1)}\Big)\tau^2
   -2(\delta+\varphi^{(\boldsymbol{\mu})}_{1})\tau^2\delta^2\\[3pt]
&=-2\delta+O(1)(\delta+\epsilon+\tau^2)\delta,\\[5pt]
\frac{\partial \mathcal{H}^{(\boldsymbol{\mu})}_{S}}{\partial \beta_1}
&=\frac{(\epsilon+1)\beta^{\epsilon}_1+\epsilon\beta^{\epsilon-1}_1-1}{a^2_{\infty}(\beta_1+1)^2\epsilon}
-\frac{2(1-\tau^2\delta^2)^{\frac{1}{2}}(1-\tau^2\mathcal{B}^{(\epsilon)})^{\frac{1}{2}}}{a^2_{\infty}(\beta_1+1)^2}
\frac{(\epsilon+1)\beta^{\epsilon}_1+\epsilon\beta^{\epsilon-1}_1-1}{\epsilon}\\[3pt]
&\quad+\Big(\frac{1-\tau^2\delta^2}{1-\tau^2\mathcal{B}^{(\epsilon)}}\Big)^{\frac{1}{2}}
\Big((\delta+ \varphi^{(\boldsymbol{\mu})}_{1})\delta+\frac{2\beta_1(\beta^{\epsilon}_{1}-1)}{a^2_{\infty}(\beta_1+1)\epsilon}\Big)\frac{\tau^2}{a^2_{\infty}}\beta^{\epsilon-1}_{1}
-\frac{2\tau^2\delta^2}{a^2_{\infty}}\beta^{\epsilon-1}_{1},
\end{split}
\end{align}
and
\begin{align}\label{eq:4.61}
\begin{split}
\frac{\partial \mathcal{H}^{(\boldsymbol{\mu})}_{S}}{\partial \epsilon}
&=-\bigg\{\frac{\beta_1-1}{a^{2}_{\infty}(\beta_1+1)}
-\Big(\frac{1-\tau^2\delta^2}{1-\tau^2\mathcal{B}^{(\epsilon)}}\Big)^{\frac{1}{2}}
\Big((\delta+ \varphi^{(\boldsymbol{\mu})}_{1})\delta+\frac{2\beta_1(\beta^{\epsilon}_{1}-1)}{a^2_{\infty}(\beta_1+1)\epsilon}\Big)\frac{\tau^2}{a^2_{\infty}}\\[3pt]
&\qquad\,\,\,\,+\frac{\big((1-\tau^2\delta^2)^{\frac{1}{2}}(1-\tau^2\mathcal{B}^{(\epsilon)})^{\frac{1}{2}}-1\big)\beta_1}{a^{2}_{\infty}(\beta_1+1)}-\frac{2}{a^2_{\infty}}\tau^2\delta^2\bigg\}
\,\frac{(\epsilon\ln\beta_1-1)\beta^{\epsilon}_1+1}{\epsilon^{2}}\\[3pt]
&=-\Big(\frac{1}{a_{\infty}}\delta+O(1)(\delta+\epsilon+\tau^2)\delta\Big)\Big(a^{2}_{\infty}\delta^2+O(1)(\delta+\epsilon+\tau^2)\delta^2\Big)\\[3pt]
&=-a_{\infty}\delta^3+O(1)(\delta+\epsilon+\tau^2)\delta^3,\\[3pt]
\frac{\partial \mathcal{H}^{(\boldsymbol{\mu})}_{S}}{\partial \tau^2}
&=\frac{1}{2}\bigg(\Big(\frac{1-\tau^2\mathcal{B}^{(\epsilon)}}{1-\tau^2\delta^2}\Big)^{\frac{1}{2}}\delta^2
   +\Big(\frac{1-\tau^2\delta^2}{1-\tau^2\mathcal{B}^{(\epsilon)}}\Big)^{\frac{1}{2}}\mathcal{B}^{(\epsilon)}\bigg)
\bigg((\delta+\varphi^{(\boldsymbol{\mu})}_{1})\delta+\frac{2\beta_1\big(\beta^{\epsilon}_1-1\big)}{a^{2}_{\infty}\epsilon(\beta_1+1)}\bigg)
   -\delta^{2}\mathcal{B}^{(\epsilon)}\\[3pt]
&=\frac{1}{2}\Big(\frac{2}{a_{\infty}}\delta+O(1)(\delta+\epsilon+\tau^2)\delta\Big)\Big(\frac{1}{a_{\infty}}\delta+O(1)(\delta+\epsilon+\tau^2)\delta\Big)\\[3pt]
&=\frac{2}{a_{\infty}}\delta^2+O(1)(\delta+\epsilon+\tau^2)\delta^2.
\end{split}
\end{align}

Now, taking the derivatives on both sides of equation \eqref{eq:4.58} with respect to $\epsilon$ and $\tau^2$ respectively, we have
\begin{eqnarray}\label{eq:4.62}
\frac{\partial \mathcal{H}^{(\boldsymbol{\mu})}_{S}}{\partial \varphi^{(\boldsymbol{\mu})}_{1}}\frac{\partial \varphi^{(\boldsymbol{\mu})}_{1}}{\partial \epsilon}
+\frac{\partial \mathcal{H}^{(\boldsymbol{\mu})}_{S}}{\partial \epsilon}=0,
\qquad \frac{\partial \mathcal{H}^{(\boldsymbol{\mu})}_{S}}{\partial \varphi^{(\boldsymbol{\mu})}_{1}}\frac{\partial \varphi^{(\boldsymbol{\mu})}_{1}}{\partial \tau^2}
+\frac{\partial \mathcal{H}^{(\boldsymbol{\mu})}_{S}}{\partial \tau^2}=0.
\end{eqnarray}
Then it follows from estimates \eqref{eq:4.60}--\eqref{eq:4.61} that
\begin{eqnarray*}
\begin{split}
\frac{\partial \varphi^{(\boldsymbol{\mu})}_{1}}{\partial \epsilon}&=-\frac{\frac{\partial \mathcal{H}^{(\boldsymbol{\mu})}_{S}}{\partial \epsilon}}{\frac{\partial \mathcal{H}^{(\boldsymbol{\mu})}_{S}}{\partial \varphi^{(\boldsymbol{\mu})}_{1}}}
=-\frac{-a_{\infty}\delta^3+O(1)(\delta+\epsilon+\tau^2)\delta^3}{-2\delta+O(1)(\delta+\epsilon+\tau^2)\delta}\\[3pt]
&=-\frac{a_{\infty}}{2}\delta^2+O(1)(\delta+\epsilon+\tau^2)\delta^2,\\[3pt]
\frac{\partial \varphi^{(\boldsymbol{\mu})}_{1}}{\partial \tau^2}&=-\frac{\frac{\partial \mathcal{H}^{(\boldsymbol{\mu})}_{S}}{\partial \tau^2}}{\frac{\partial \mathcal{H}^{(\boldsymbol{\mu})}_{S}}{\partial \varphi^{(\boldsymbol{\mu})}_{1}}}=-\frac{\frac{2}{a_{\infty}}\delta^2+O(1)(\delta+\epsilon+\tau^2)\delta+O(1)(\delta+\epsilon+\tau^2)\delta^2}{-2\delta+O(1)(\delta+\epsilon+\tau^2)\delta}\\[3pt]
&=\frac{1}{2a_{\infty}}\delta+O(1)(\delta+\epsilon+\tau^2)\delta,
\end{split}
\end{eqnarray*}
which leads to
\begin{eqnarray}\label{eq:4.63}
\frac{\partial \varphi^{(\boldsymbol{\mu})}_{1}}{\partial\epsilon}\Big|_{\alpha_1=1, \epsilon=\tau=0}=
\frac{\partial \varphi^{(\boldsymbol{\mu})}_{1}}{\partial \tau^2}\Big|_{\alpha_1=1, \epsilon=\tau=0}=0.
\end{eqnarray}
Therefore, we obtain from  \eqref{eq:4.51} and \eqref{eq:4.63} that
\begin{align}\label{eq:4.64}
\frac{\partial\beta_{1}}{\partial \epsilon}\Big|_{\alpha_1=1, \epsilon=\tau=0}=
-\frac{\frac{\partial\varphi^{(\boldsymbol{\mu})}_{1}}{\partial \epsilon}\Big|_{\alpha_1=1, \epsilon=\tau=0}}
{\frac{\partial\varphi^{(\boldsymbol{\mu})}_{1}}{\partial \beta_1}\Big|_{\alpha_1=1, \epsilon=\tau=0}}
=0,\qquad \frac{\partial\beta_{1}}{\partial \tau^2}\Big|_{\alpha_1=1, \epsilon=\tau=0}=
-\frac{\frac{\partial\varphi^{(\boldsymbol{\mu})}_{1}}{\partial \tau^2}\Big|_{\alpha_1=1, \epsilon=\tau=0}}
{\frac{\partial\varphi^{(\boldsymbol{\mu})}_{1}}{\partial \beta_1}\Big|_{\alpha_1=1, \epsilon=\tau=0}}=0.
\end{align}

Next, we further take the derivatives on equation \eqref{eq:4.62} with respect to $\varphi^{(\boldsymbol{\mu})}_{1}$, $\beta_1$, $\epsilon$, and $\tau^2$ respectively,
and combine with the estimates in Lemmas \ref{lem:A1}--\ref{lem:A2} to obtain
\begin{align}\label{eq:4.65}
\frac{\partial^2 \mathcal{H}^{(\boldsymbol{\mu})}_{S}}{\partial \varphi^{(\boldsymbol{\mu})^{2}}_{1}}
&=2+(\delta+\varphi^{(\boldsymbol{\mu})}_{1})(1-\tau^2\mathcal{B}^{(\epsilon)})^{-\frac{3}{2}}\tau^2\delta-2\tau^2\delta^2\nonumber\\[3pt]
&\quad+2(1-\tau^2\delta^2)^{\frac{1}{2}}(1-\tau^2\mathcal{B}^{(\epsilon)})^{-\frac{3}{2}}
  \bigg\{(\delta+\varphi^{(\boldsymbol{\mu})}_{1})\Big((\delta+\varphi^{(\boldsymbol{\mu})}_{1})\delta
+\frac{2\beta_1\big(\beta^{\epsilon}_1-1\big)}{a^{2}_{\infty}\epsilon(\beta_1+1)}\Big)\tau^2\nonumber\\[3pt]
&\qquad\qquad\qquad\qquad\qquad\qquad\quad\quad\,\,\,\,
  +(1-\tau^2\mathcal{B}^{(\epsilon)})\Big(2(\delta+\varphi^{(\boldsymbol{\mu})}_{1})\delta+\frac{2\beta_1\big(\beta^{\epsilon}_1-1\big)}{a^{2}_{\infty}\epsilon(\beta_1+1)}\Big)
\bigg\}\nonumber\\[3pt]
&=2+O(1)(\delta+\epsilon+\tau^2)\delta,
\end{align}
\begin{align}\label{eq:4.66}
\frac{\partial^2 \mathcal{H}^{(\boldsymbol{\mu})}_{S}}{\partial \beta_{1}\partial \varphi^{(\boldsymbol{\mu})}_{1}}
&=\bigg\{\frac{\big((\beta_1+1)\epsilon+1\big)\beta^{\epsilon}_1-1}{(\beta_1+1)^{2}\epsilon}
+\frac{\beta^{\epsilon-1}_1\delta}{\delta+\varphi^{(\boldsymbol{\mu})}_{1}}
+\frac{(\delta+\varphi^{(\boldsymbol{\mu})}_{1})\beta^{\epsilon-1}_1\tau^2}{2(1-\tau^2\mathcal{B}^{(\epsilon)})}
 \Big((\delta+\varphi^{(\boldsymbol{\mu})}_{1})\delta+\frac{2\beta_1\big(\beta^{\epsilon}_1-1\big)}{a^{2}_{\infty}(\beta_1+1)\epsilon}\Big)
\bigg\}\nonumber\\[3pt]
&\quad\,\, \times \Big(\frac{1-\tau^2\delta^2}{1-\tau^2\mathcal{B}^{(\epsilon)}}\Big)^{\frac{1}{2}}\frac{(\delta+\varphi^{(\boldsymbol{\mu})}_{1})\tau^2}{a^{2}_{\infty}}
 \nonumber\\[3pt]
&=O(1)\tau^2\delta,
\end{align}

\begin{align}\label{eq:4.67}
\frac{\partial^2 \mathcal{H}^{(\boldsymbol{\mu})}_{S}}{\partial \varphi^{(\boldsymbol{\mu})}_{1}\partial \epsilon}
&=\bigg\{\delta+2(\delta+\varphi^{(\boldsymbol{\mu})}_{1})(1-\tau^2\delta)^{\frac{1}{2}}\Big((\delta+\varphi^{(\boldsymbol{\mu})}_{1})\delta
+\frac{2\beta_1\big(\beta^{\epsilon}_1-1\big)}{a^{2}_{\infty}\epsilon(\beta_1+1)}\Big)\tau^2
\nonumber\\[3pt]
&\qquad +\frac{4(\delta+\varphi^{(\boldsymbol{\mu})}_{1})(1-\tau^2\delta)^{\frac{1}{2}}(1-\tau^2\mathcal{B}^{(\epsilon)})}{\beta_1+1}\bigg\}\, \frac{(1-\tau^2\mathcal{B}^{(\epsilon)})^{-\frac{3}{2}}\tau^2}{a^{2}_{\infty}}\frac{(\epsilon\ln\beta_1-1)\beta^{\epsilon}_1+1}{\epsilon^2}\nonumber\\[3pt]
&=O(1)\tau^2\delta^3,
\end{align}

\begin{align}\label{eq:4.68}
\frac{\partial^2 \mathcal{H}^{(\boldsymbol{\mu})}_{S}}{\partial \varphi^{(\boldsymbol{\mu})}_{1}\partial \tau^2}
&=\frac{1}{2}\Big((1-\tau^2\mathcal{B}^{(\epsilon)})^{-\frac{3}{2}}\mathcal{B}^{(\epsilon)}+(1-\tau^2\delta^2)^{-\frac{1}{2}}\delta^2\Big)\delta-2(\delta+\varphi^{(\boldsymbol{\mu})}_{1})\delta^2\nonumber\\[3pt]
&\quad +2(\delta+\varphi^{(\boldsymbol{\mu})}_{1})\Big((\delta+\varphi^{(\boldsymbol{\mu})}_{1})\delta+\frac{2\beta_1\big(\beta^{\epsilon}_1-1\big)}{a^{2}_{\infty}\epsilon(\beta_1+1)}\Big)
  \Big(\frac{1-\tau^2\delta^2}{1-\tau^2\mathcal{B}^{(\epsilon)}}\Big)^{\frac{1}{2}} \nonumber\\[3pt]
&\qquad \times
\Big(1+\frac{(1-\tau^2\delta^2)^{\frac{1}{2}}\mathcal{B}^{(\epsilon)}-(1-\tau^2\mathcal{B}^{(\epsilon)})^{\frac{1}{2}}\delta^2}{2(1-\tau^2\delta^2)(1-\tau^2\mathcal{B}^{(\epsilon)})}\Big)\tau^2
\nonumber\\[3pt]
&=O(1)(\delta+\epsilon+\tau^2)\delta,
\end{align}

\begin{align}\label{eq:4.69}
\frac{\partial^2 \mathcal{H}^{(\boldsymbol{\mu})}_{S}}{\partial \beta_{1}\partial \epsilon}
&=\frac{2}{a^2_{\infty}(\beta_1+1)^2}\frac{\big(\epsilon(\epsilon+1)\ln\beta_1-1\big)\beta^{\epsilon}_1+\epsilon^2\beta^{\epsilon-1}\ln\beta_1+1}{\epsilon^2}\nonumber\\[3pt]
&\quad-\frac{2\big(1-\tau^2\delta^2\big)^{\frac{1}{2}}\big(1-\tau^2\mathcal{B}^{(\epsilon)}\big)^{\frac{1}{2}}}{a^2_{\infty}(\beta_1+1)^2}
\frac{\big((\beta_1+1)\epsilon+1\big)\epsilon\beta^{\epsilon}_1\ln\beta_1-\beta^{\epsilon}+1}{\epsilon^2}\nonumber\\[3pt]
&\quad+\bigg\{\frac{2}{a^2_{\infty}(\beta_1+1)^2}\frac{\big[(\beta_1+1)\epsilon+1\big]\beta^{\epsilon}_1-1}{\epsilon}+\frac{\beta^{\epsilon}_1}{\beta_1+1}
+\Big((\delta+\varphi^{(\boldsymbol{\mu})}_{1})\delta+\frac{2\beta_1\big(\beta^{\epsilon}_1-1\big)}{a^{2}_{\infty}(\beta_1+1)\epsilon}\Big)
\frac{\beta^{\epsilon}_1\tau^2}{2}\bigg\}\nonumber\\[3pt]
&\qquad\,\,\times \frac{2\tau^2}{a^{4}_{\infty}}\Big(\frac{1-\tau^2\delta^2}{1-\tau^2\mathcal{B}^{(\epsilon)}}\Big)^{\frac{1}{2}}
\frac{\big(\epsilon\ln\beta_1-1\big)\beta^{\epsilon}_1+1}{\epsilon^2}\nonumber\\[3pt]
&\quad+\bigg\{\Big(\frac{1-\tau^2\delta^2}{1-\tau^2\mathcal{B}^{(\epsilon)}}\Big)^{\frac{1}{2}}
\Big((\delta+\varphi^{(\boldsymbol{\mu})}_{1})\delta+\frac{2\beta_1\big(\beta^{\epsilon}_1-1\big)}{a^{2}_{\infty}(\beta_1+1)\epsilon}\Big)-\delta^2\bigg\}
\frac{\beta^{\epsilon}_1\ln\beta_1}{a^2_{\infty}}\tau^2\nonumber\\[3pt]
&=\frac{1}{2a^2_{\infty}}\Big(\frac{3a_{\infty}}{2}\delta+O(1)(\delta+\epsilon+\tau^2)\delta\Big)
-\frac{1}{2a^2_{\infty}}\Big(2a_{\infty}\delta+O(1)(\delta+\epsilon+\tau^2)\delta\Big)+O(1)(\epsilon+\tau^2)\delta\nonumber\\[3pt]
&=-\frac{1}{4a_{\infty}}\delta+O(1)(\delta+\epsilon+\tau^2)\delta,
\end{align}
and
\begin{align}\label{eq:4.70}
\frac{\partial^2 \mathcal{H}^{(\boldsymbol{\mu})}_{S}}{\partial \beta_{1}\partial \tau^2}
&=\frac{(1-\tau^2\mathcal{B}^{(\epsilon)})\delta^2+(1-\tau^2\delta^2)\mathcal{B}^{(\epsilon)}}{(1-\tau^2\delta^2)^{\frac{1}{2}}(1-\tau^2\mathcal{B}^{(\epsilon)})^{\frac{1}{2}}}
\frac{\big((\beta_1+1)\epsilon+1\big)\beta^{\epsilon}_1-1}{a^2_{\infty}(\beta_1+1)^2\epsilon}\nonumber\\[3pt]
&\quad+\frac{\beta^{\epsilon-1}_1}{a^{2}_{\infty}}\Big(\frac{1-\tau^2\delta^2}{1-\tau^2\mathcal{B}^{(\epsilon)}}\Big)^{\frac{1}{2}}
\bigg(1+\frac{\big(\mathcal{B}^{(\epsilon)}-\delta^2\big)\tau^2}{2(1-\tau^2\delta^2)(1-\tau^2\mathcal{B}^{(\epsilon)})}\bigg)
\Big((\delta+\varphi^{(\boldsymbol{\mu})}_{1})\delta+\frac{2\beta_1\big(\beta^{\epsilon}_1-1\big)}{a^{2}_{\infty}(\beta_1+1)\epsilon}\Big)\nonumber\\[3pt]
&=O(1)(\delta+\epsilon+\tau^2)\delta.
\end{align}
Then we obtain from estimates \eqref{eq:4.65}--\eqref{eq:4.70} that
\begin{align*}
\frac{\partial^2\varphi^{(\boldsymbol{\mu})}_{1}}{\partial \beta_1 \partial \epsilon}
&=-\frac{\Big(\frac{\partial^2\mathcal{H}^{(\boldsymbol{\mu})}_{S}}{\partial\varphi^{(\boldsymbol{\mu})^{2}}_{1}}
\frac{\partial\varphi^{(\boldsymbol{\mu})}_{1}}{\partial \epsilon}
+\frac{\partial^2 \mathcal{H}^{(\boldsymbol{\mu})}_{S}}{\partial \varphi^{(\boldsymbol{\mu})}_{1}\partial\epsilon}\Big)\frac{\partial\varphi^{(\boldsymbol{\mu})}_{1}}{\partial\beta_1}
+\frac{\partial^2 \mathcal{H}^{(\boldsymbol{\mu})}_{S}}{\partial \beta_1 \partial \varphi^{(\boldsymbol{\mu})}_{1}}\frac{\partial\varphi^{(\boldsymbol{\mu})}_{1}}{\partial\epsilon}
+\frac{\partial^2 \mathcal{H}^{(\boldsymbol{\mu})}_{S}}{\partial \beta_{1}\partial \epsilon}}
{\frac{\partial\mathcal{H}^{(\boldsymbol{\mu})}_{S}}{\partial \varphi^{(\boldsymbol{\mu})}_{1}}}\\[3pt]
&=\Big(\frac{1}{a_{\infty}}+O(1)(\delta+\epsilon+\tau^2)\Big)\nonumber\\[3pt]
&\qquad\times \frac{\Big(\big(2+O(1)(\delta+\epsilon+\tau^2)\delta\big)\big(-\frac{a_{\infty}}{2}\delta^2+O(1)(\delta+\epsilon+\tau^2)\delta^2\big)
+O(1)\tau^2\delta^3\Big)}
{-2\delta+O(1)(\delta+\epsilon+\tau^2)\delta}\\[3pt]
&\quad-\frac{\Big(\frac{a_{\infty}}{2}\delta^2+O(1)(\delta+\epsilon+\tau^2)\delta^2\Big)O(1)\tau^2\delta}{2\delta+O(1)(\delta+\epsilon+\tau^2)\delta}
-\frac{\frac{1}{4a_{\infty}}\delta+O(1)(\delta+\epsilon+\tau^2)\delta}{2\delta+O(1)(\delta+\epsilon+\tau^2)\delta}\\[3pt]
&=-\frac{1}{8a_{\infty}}+O(1)\big(\delta+\epsilon+\tau^2\big),
\end{align*}
and
\begin{align*}
\frac{\partial^2\varphi^{(\boldsymbol{\mu})}_{1}}{\partial \beta_1 \partial \tau^2}
&=-\frac{\Big(\frac{\partial^2\mathcal{H}^{(\boldsymbol{\mu})}_{S}}{\partial\varphi^{(\boldsymbol{\mu})^{2}}_{1}}
\frac{\partial\varphi^{(\boldsymbol{\mu})}_{1}}{\partial \tau^2}
+\frac{\partial^2 \mathcal{H}^{(\boldsymbol{\mu})}_{S}}{\partial \varphi^{(\boldsymbol{\mu})}_{1}\partial\tau^2}\Big)\frac{\partial\varphi^{(\boldsymbol{\mu})}_{1}}{\partial\beta_1}
+\frac{\partial^2 \mathcal{H}^{(\boldsymbol{\mu})}_{S}}{\partial \beta_1 \partial \varphi^{(\boldsymbol{\mu})}_{1}}\frac{\partial\varphi^{(\boldsymbol{\mu})}_{1}}{\partial\tau^2}
+\frac{\partial^2 \mathcal{H}^{(\boldsymbol{\mu})}_{S}}{\partial \beta_{1}\partial \tau^2}}
{\frac{\partial\mathcal{H}^{(\boldsymbol{\mu})}_{S}}{\partial \varphi^{(\boldsymbol{\mu})}_{1}}}\\[3pt]
&=-\Big(\frac{1}{a_{\infty}}+O(1)(\delta+\epsilon+\tau^2)\Big)\nonumber\\[3pt]
&\qquad\,\,\times \frac{\Big(\big(2+O(1)(\delta+\epsilon+\tau^2)\delta\big)\big(\frac{1}{2a^2_{\infty}}+O(1)(\delta+\epsilon+\tau^2)\delta\big)
+O(1)(\delta+\epsilon+\tau^2)\delta\Big)}
{2\delta+O(1)(\delta+\epsilon+\tau^2)\delta}\\[3pt]
&\quad +\frac{O(1)(\delta+\epsilon+\tau^2)\delta}{2\delta+O(1)(\delta+\epsilon+\tau^2)\delta}\\[3pt]
&=-\frac{1}{2a^3_{\infty}}+O(1)(\delta+\epsilon+\tau^2),
\end{align*}
which show that
\begin{eqnarray}\label{eq:4.71}
\frac{\partial^2\varphi^{(\boldsymbol{\mu})}_{1}}{\partial \beta_1 \partial \epsilon}\Big|_{\alpha_1=1,\epsilon=\tau=0}=-\frac{1}{8a_{\infty}},\qquad
\frac{\partial^2\varphi^{(\boldsymbol{\mu})}_{1}}{\partial \beta_1 \partial \tau^2}\Big|_{\alpha_1=1,\epsilon=\tau=0}=-\frac{1}{2a^3_{\infty}}.
\end{eqnarray}

Using Lemma \ref{lem:2.6}, \eqref{eq:4.64}, and \eqref{eq:4.71}, we thus obtain
\begin{eqnarray}\label{eq:4.72}
\begin{split}
\frac{\partial^2 \beta_1}{\partial \alpha_1\partial\epsilon}\Big|_{\alpha_1=1,\epsilon=\tau=0}
&=-\frac{\Big(\frac{\partial^2\varphi^{(\boldsymbol{\mu})}_{1}}{\partial \beta^2_1}\frac{\partial \beta_1}{\partial\epsilon}
+\frac{\partial^2\varphi^{(\boldsymbol{\mu})}_{1}}{\partial \beta_1\partial \epsilon}\Big)\frac{\partial \beta_1}{\partial \alpha_1}\Big|_{\alpha_1=1,\epsilon=\tau=0}}
{\frac{\partial\varphi^{(\boldsymbol{\mu})}_{1}}{\partial \beta_1}\Big|_{\alpha_1=1,\epsilon=\tau=0}}
=-\frac{1}{8},\\[5pt]
\frac{\partial^2 \beta_1}{\partial \alpha_1\partial\tau^2}\Big|_{\alpha_1=1,\epsilon=\tau=0}
&=-\frac{\Big(\frac{\partial^2\varphi^{(\boldsymbol{\mu})}_{1}}{\partial \beta^2_1}\frac{\partial \beta_1}{\partial\tau^2}
+\frac{\partial^2\varphi^{(\boldsymbol{\mu})}_{1}}{\partial \beta_1\partial \tau^2}\Big)
\frac{\partial \beta_1}{\partial \alpha_1}\Big|_{\alpha_1=1,\epsilon=\tau=0}}{\frac{\partial\varphi^{(\boldsymbol{\mu})}_{1}}{\partial \beta_1}\Big|_{\alpha_1=1,\epsilon=\tau=0}}
=-\frac{1}{2a^2_{\infty}}.
\end{split}
\end{eqnarray}

Finally, combining the Taylor formula again with \eqref{eq:4.72}, we arrive at
\smallskip
\begin{align}\label{eq:4.73}
\beta_{1}(\alpha_1,\epsilon, \tau^2)&=\beta_1(\alpha_1,0,\tau^2)+\beta_1(1,\epsilon, \tau^2)-\beta_1(1,0,\tau^2)\nonumber\\[3pt]
&\quad + (\alpha_1-1)\epsilon\int^1_0\int^1_0\frac{\partial^2 \beta_1}{\partial \alpha_1\partial \epsilon}(\xi_1 (\alpha_1-1)+1, \xi_2\epsilon, \tau^2)\,{\rm d}\xi_1{\rm d}\xi_2\nonumber\\[3pt]
&=\beta_1(\alpha_1,0,0)+\beta_1(1,0,\tau^2)-\beta_1(1,0,0)\nonumber\\[3pt]
&\quad +(\alpha_1-1)\tau^2\int^1_0\int^1_0\frac{\partial^2 \beta_1}{\partial \alpha_1\partial \epsilon}(\xi_1 (\alpha_1-1)+1,0, \xi_3\tau^2)\,{\rm d}\xi_1{\rm d}\xi_3\nonumber\\[3pt]
&\quad +\beta_1(1,\epsilon, \tau^2)-\beta_1(1,0,\tau^2)\nonumber\\[3pt]
&\quad + (\alpha_1-1)\epsilon\int^1_0\int^1_0\frac{\partial^2 \beta_1}{\partial \alpha_1\partial \epsilon}(\xi_1 (\alpha_1-1)+1, \xi_2\epsilon, \tau^2)\,{\rm d}\xi_1{\rm d}\xi_2\nonumber\\[3pt]
&=\alpha_1+\frac{\partial^2 \beta_1}{\partial \alpha_1\partial \epsilon}\Big|_{\alpha_1=1,\epsilon=\tau^2=0}(\alpha_1-1)\epsilon+\frac{\partial^2 \beta_1}{\partial \alpha_1\partial \tau^2}\Big|_{\alpha_1=1,\epsilon=\tau^2=0}(\alpha_1-1)\tau^2\nonumber\\[3pt]
&\quad +O(1)\big(|\alpha_1-1|+\epsilon+\tau^2\big)|\alpha_1-1|(\epsilon+\tau^2)\nonumber\\[3pt]
&=\alpha_1-\frac{1}{8}(\alpha_1-1)\epsilon-\frac{1}{2a^2_{\infty}}(\alpha_1-1)\tau^2
+O(1)\big(|\alpha_1-1|+\epsilon+\tau^2\big)|\alpha_1-1|(\epsilon+\tau^2),
\end{align}
where the bounds of $O(1)$ depend only on $\underline{U}$.

Denoted by $\sigma_{1}(\alpha_1)$ and $\sigma^{(\boldsymbol{\mu})}_{1}(\beta_1)$ the speeds of $\alpha_1$ and $\beta_1$, respectively.
Then, for $\delta$, $\epsilon$, and $\tau$ are sufficiently small, by \eqref{eq:4.49}, we have
\begin{eqnarray}
&&\sigma_{1}(\alpha_1)=\frac{\rho_{b}v_{b}-\rho_{l}v_{l}}{\rho_{b}-\rho_{l}}=-\frac{\delta}{\alpha_1-1}=-\frac{1}{a_{\infty}}+O(1)\delta,\label{eq:4.74}\\[4pt]
&&\sigma^{(\tau)}_{1}(\beta_1)=\frac{\rho^{(\boldsymbol{\mu})}_{b}v^{(\boldsymbol{\mu})}_{b}-\rho_{l}v_{l}}{\rho^{(\boldsymbol{\mu})}_{b}\sqrt{1-\tau^2\mathcal{B}^{(\epsilon)}}-\rho_{l}\sqrt{1-\tau^2\delta}}
=-\frac{\delta}{\beta_1\sqrt{1-\tau^2\mathcal{B}^{(\epsilon)}}-\sqrt{1-\tau^2\delta}}\nonumber\\[3pt]
&&\qquad\qquad =-\frac{\delta}{\beta_1-1}+O(1)\tau^2\delta.\label{eq:4.75}
\end{eqnarray}
Notice from \eqref{eq:4.73} that
\begin{eqnarray*}
\beta_1-1&=(\alpha_1-1)\Big(1-\frac{1}{8}\epsilon-\frac{1}{2a^2_{\infty}}\tau^2
+O(1)\big(|\alpha_1-1|+\epsilon+\tau^2\big)(\epsilon+\tau^2)\Big).
\end{eqnarray*}
Therefore, we can further deduce from \eqref{eq:4.75} that
\begin{align}\label{eq:4.76}
\sigma^{(\tau)}_{1}(\beta_1)&=-\frac{\delta}{(\alpha_1-1)\Big(1-\frac{1}{8}\epsilon-\frac{1}{2a^2_{\infty}}\tau^2
+O(1)\big(|\alpha_1-1|+\epsilon+\tau^2\big)(\epsilon+\tau^2)\Big)}+O(1)\tau^2\delta\nonumber\\[3pt]
&=-\frac{\delta}{\alpha_1-1}\Big(1+\frac{1}{8}\epsilon+\frac{1}{2a^2_{\infty}}\tau^2+O(1)\big(\delta+\epsilon+\tau^2\big)(\epsilon+\tau^2)\Big)+O(1)\tau^2\delta.
\end{align}
Thus, by \eqref{eq:4.74} and \eqref{eq:4.76}, we have
\begin{align}\label{eq:4.77}
\sigma_{1}(\alpha_1)-\sigma^{(\tau)}_{1}(\beta_1)&=\frac{\delta}{\alpha_1-1}\Big(1+\frac{1}{8}\epsilon+\frac{1}{2a^2_{\infty}}\tau^2+O(1)\big(\delta+\epsilon+\tau^2\big)(\epsilon+\tau^2)\Big)
-\frac{\delta}{\alpha_1-1}+O(1)\tau^2\delta\nonumber\\[3pt]
&=\frac{\delta}{\alpha_1-1}\Big(\frac{3}{8}\epsilon+\frac{1}{2a^2_{\infty}}\tau^2\Big)+O(1)\big(\delta+\epsilon+\tau^2\big)(\epsilon+\tau^2)+O(1)\delta\tau^2\nonumber\\[3pt]
&=\frac{1}{8a_{\infty}}\epsilon+\frac{1}{2a^3_{\infty}}\tau^2+O(1)\big(\delta+\epsilon+\tau^2\big)(\epsilon+\tau^2).
\end{align}

On the other hand, using estimates \eqref{eq:4.49} and \eqref{eq:4.54}, we have
\begin{align}
|\rho^{(\boldsymbol{\mu})}_{{b}}-\rho_{l}|&=|\beta_1-1\big|\nonumber\\[3pt]
&=\Big|(\alpha_1-1)\big(1-\frac{1}{8}\epsilon-\frac{1}{2a^2_{\infty}}\tau^2
+O(1)(|\alpha_1-1|+\epsilon+\tau^2)(\epsilon+\tau^2)\big)\Big|\nonumber\\[3pt]
&=a_{\infty}\delta+O(1)\big(\delta+\epsilon+\tau^2\big)\delta,\label{eq:4.78}\\[4pt]
|\rho^{(\boldsymbol{\mu})}_{{b}}-\rho_{b}|&=|\beta_1-\alpha_1|\nonumber\\[3pt]
&=\Big|(\alpha_1-1)\big(-\frac{3}{8}\epsilon-\frac{1}{2a^2_{\infty}}\tau^2
+O(1)(|\alpha_1-1|+\epsilon+\tau^2)(\epsilon+\tau^2)\big)\Big|\nonumber\\[3pt]
&=\Big(\frac{1}{8}\epsilon+\frac{1}{2a^2_{\infty}}\tau^2\Big)a_{\infty}\delta+O(1)\big(\delta+\epsilon+\tau^2\big)(\epsilon+\tau^2)\delta.\label{eq:4.79}
\end{align}

With estimates \eqref{eq:4.74} and \eqref{eq:4.77}--\eqref{eq:4.79} in hand, we thus have
\begin{align}
&\big(|\rho^{(\boldsymbol{\mu})}_{{b}}-\rho_{l}|+|v^{(\boldsymbol{\mu})}_{b}-v_{l}|\big)\,
 \big(\sigma_1(\alpha_1)-\sigma^{(\boldsymbol{\mu})}_1(\beta_1)\big)\nonumber\\[3pt]
&\,\, =\Big(a_{\infty}\delta+O(1)\big(\delta+\epsilon+\tau^2\big)\delta+\delta\Big)
\Big(\frac{3}{8a_{\infty}}\epsilon+\frac{1}{2a^3_{\infty}}\tau^2+O(1)\big(\delta+\epsilon+\tau^2\big)(\epsilon+\tau^2)\Big)\nonumber\\[3pt]
&\,\, =\frac{a_{\infty}+1}{8a_{\infty}}\epsilon\delta +\frac{a_{\infty}+1}{2a^3_{\infty}}\tau^2\delta +O(1)\big(\delta+\epsilon+\tau^2\big)(\epsilon+\tau^2)\delta,\label{eq:4.80}\\[4pt]
&\big(|\rho_{b}-\rho^{(\boldsymbol{\mu})}_{b}|+|v_{b}-v^{(\boldsymbol{\mu})}_{b}|\big)\,\big(-\sigma_1(\alpha_1)\big)\nonumber\\[3pt]
&\,\,=\Big(\big(\frac{3}{8}\epsilon+\frac{1}{2a^2_{\infty}}\tau^2\big)a_{\infty}\delta+O(1)\big(\delta+\epsilon+\tau^2\big)(\epsilon+\tau^2)\delta\Big)
\Big(a_{\infty}\delta+O(1)\big(\delta+\epsilon+\tau^2\big)\delta\Big)\nonumber\\[3pt]
&\,\,=\frac{1}{8}\epsilon\delta +\frac{1}{2a^2_{\infty}}\tau^2\delta +O(1)\big(\delta+\epsilon+\tau^2\big)(\epsilon+\tau^2)\delta.\label{eq:4.81}
\end{align}

Finally, combining estimate \eqref{eq:4.80} with estimate \eqref{eq:4.81}, we conclude
\begin{align}\label{eq:4.82}
\big\|U^{(\boldsymbol{\mu})}-U\big\|_{L^1}
&=\int^{\sigma_1(\alpha_1)x}_{\sigma^{(\boldsymbol{\mu})}_1(\beta_1)x}\big|U^{(\boldsymbol{\mu})}-U\big|\,{\rm d}y
  +\int^{0}_{\sigma_1(\alpha_1)x}\big|U^{(\boldsymbol{\mu})}-U\big|\,{\rm d}y\nonumber\\[3pt]
&=\big(|\rho^{(\boldsymbol{\mu})}_{{b}}-\rho_{l}|+|v^{(\boldsymbol{\mu})}_{b}-v_{l}|\big)\,
\big(\sigma_1(\alpha_1)-\sigma^{(\boldsymbol{\mu})}_1(\beta_1)\big)x\nonumber\\[3pt]
&\quad +\big(|\rho_{b}-\rho^{(\boldsymbol{\mu})}_{b}|+|v_{b}-v^{(\boldsymbol{\mu})}_{b}|\big)\,\big(-\sigma_1(\alpha_1)\big)x\nonumber\\[3pt]
&=\frac{2a_{\infty}+1}{8a_{\infty}}\epsilon\delta x+\frac{2a_{\infty}+1}{2a^3_{\infty}}\tau^2\delta x +O(1)\big(\delta+\epsilon+\tau^2\big)(\epsilon+\tau^2)\delta x.
\end{align}
This implies that the convergence rate we have obtained in Theorem \ref{thm:1.1} is optimal.

\smallskip
\section{ Proof of Theorem \ref{coro:1.1}}

In this section, we give the proof of Theorem \ref{coro:1.1} to establish the convergence rate between
the entropy solutions $(\rho^{(\boldsymbol{\mu})}, u^{(\boldsymbol{\mu})}, v^{(\boldsymbol{\mu})})$
of problems \eqref{eq:1.6}--\eqref{eq:1.9} and the entropy solution $(\rho, u, v)$ of problem \eqref{eq:1.10}--\eqref{eq:1.12}.

For $\rho>0$, define
\begin{eqnarray}\label{eq-5.1}
\Psi(U, \boldsymbol{\mu})=\frac{B^{(\epsilon)}(\rho, v,\epsilon)}{\sqrt{1-\tau^2 B^{(\epsilon)}(\rho, v,\epsilon)}+1},\qquad
\Psi(U)=\frac{1}{2}v^2+\frac{\ln \rho}{a^2_{\infty}},
\end{eqnarray}
where $U=(\rho, v)^{\tau}$ and $B^{(\epsilon)}$ are given by \eqref{eq:1.14}.
Clearly, for fixed $U$, we have
\begin{eqnarray}\label{eq-5.2}
\Psi(U, \boldsymbol{0})=\Psi(U).
\end{eqnarray}

We first introduce two lemmas which are useful to the proof of Theorem \ref{coro:1.1}.
\begin{lemma}\label{lem-5.1}
For $\rho>0$,
\begin{eqnarray}\label{eq-5.3}
\partial_{\rho}B^{(\epsilon)}=\frac{2\rho^{\epsilon-1}}{a^2_{\infty}},\quad\,\, \partial_{v}B^{(\epsilon)}=2v,\quad\,\,
\partial_{\epsilon}B^{(\epsilon)}=\frac{2(\epsilon \rho^{\epsilon}\ln \rho-\rho^{\epsilon}+1)}{a^2_{\infty} \epsilon^2}.
\end{eqnarray}
\end{lemma}
Lemma \ref{lem-5.1} is obtained by direct calculation, so we omit the details.

\begin{lemma}\label{lem-5.2}
For $\rho>0$,
\begin{eqnarray}\label{eq-5.4}
\partial_{\rho}\Psi(U,\boldsymbol{\mu})=\frac{\rho^{\epsilon-1}}{a^2_{\infty}\sqrt{1-\tau^2B^{(\epsilon)}}},\quad\,\,
\partial_{v}\Psi(U,\boldsymbol{\mu})=\frac{v}{a^2_{\infty}\sqrt{1-\tau^2B^{(\epsilon)}}},
\end{eqnarray}
and
\begin{eqnarray}\label{eq-5.5}
\partial_{\epsilon}\Psi(U,\boldsymbol{\mu})=\frac{\epsilon\rho^{\epsilon}\ln \rho-\rho^{\epsilon}+1}{a^2_{\infty}\sqrt{1-\tau^2B^{(\epsilon)}}},\quad\,\,
\partial_{\tau^2}\Psi(U,\boldsymbol{\mu})=\frac{B^{(\epsilon)^2}}{2\sqrt{1-\tau^2B^{(\epsilon)}}\big(\sqrt{1-\tau^2B^{(\epsilon)}}+1\big)^2}.
\end{eqnarray}
\end{lemma}

\begin{proof}
By direct computation, we have
\begin{eqnarray*}
\begin{split}
\partial_{\rho}\Psi(U,\boldsymbol{\mu})&=\frac{\big(\sqrt{1-\tau^2B^{(\epsilon)}}+1\big)\partial_{\rho}B^{(\epsilon)}
+\frac{1}{2}\tau^2B^{(\epsilon)}(1-\tau^2B^{\epsilon})^{-\frac{1}{2}}\partial_{\rho}B^{(\epsilon)}}{\big(\sqrt{1-\tau^2B^{(\epsilon)}}+1\big)^2}\\[2pt]
&=\frac{\big(2+2\sqrt{1-\tau^2B^{(\epsilon)}}-\tau^2B^{(\epsilon)}\big)\partial_{\rho}B^{(\epsilon)}}{2\sqrt{1-\tau^2B^{(\epsilon)}}\big(\sqrt{1-\tau^2B^{(\epsilon)}}+1\big)^2}\\[2pt]
&=\frac{\partial_{\rho}B^{(\epsilon)}}{2\sqrt{1-\tau^2B^{(\epsilon)}}}.
\end{split}
\end{eqnarray*}
Then the expression of $\partial_{\rho}\Psi(U,\boldsymbol{\mu})$ follows from Lemma \ref{lem-5.1}.
The expressions of  $\partial_{v}\Psi(U,\boldsymbol{\mu})$ and $\partial_{\epsilon}\Psi(U,\boldsymbol{\mu})$ can be obtained
by similar arguments from Lemma \ref{lem-5.1}.

Finally, for $\partial_{\tau^2}\Psi(U,\boldsymbol{\mu})$, by \eqref{eq-5.1} and direct calculations, we have
\begin{eqnarray*}
\begin{split}
\partial_{\tau^2}\Psi(U,\boldsymbol{\mu})&=\frac{-B^{(\epsilon)}}{\big(\sqrt{1-\tau^2B^{(\epsilon)}}+1\big)^2}
\frac{-B^{(\epsilon)}}{2}\big(1-\tau^2B^{(\epsilon)}\big)^{-\frac{1}{2}}\\[2pt]
&=\frac{B^{(\epsilon)^2}}{2\sqrt{1-\tau^2B^{(\epsilon)}}\big(\sqrt{1-\tau^2B^{(\epsilon)}}+1\big)^2}.
\end{split}
\end{eqnarray*}

This completes the proof of the lemma.
\end{proof}

Now, we are ready to prove Theorem \ref{coro:1.1}.

\begin{proof}[Proof of {\rm Theorem \ref{coro:1.1}}]
$\,$ Let $(\rho^{(\boldsymbol{\mu})}, v^{(\boldsymbol{\mu})})$
be the entropy solution of problem \eqref{eq:1.15}--\eqref{eq:1.17} obtained by Proposition \ref{prop:3.1},
and $(\rho, v)$ be the entropy solution of problem \eqref{eq:1.19}--\eqref{eq:1.21} as given by Proposition \ref{prop:3.2}.
Then, by relations \eqref{eq:1.13} and \eqref{eq:1.18}, we obtain $u^{(\boldsymbol{\mu})}$ and $u$ from the solutions of problem \eqref{eq:1.6}--\eqref{eq:1.9}
and problem \eqref{eq:1.10}--\eqref{eq:1.12}, respectively, as
\begin{eqnarray}\label{eq-5.6}
u^{(\boldsymbol{\mu})}=\frac{1}{\tau^2}\Big(\sqrt{1-\tau^2 B^{(\epsilon)}(\rho^{(\boldsymbol{\mu})},v^{(\boldsymbol{\mu})},\epsilon)}-1\Big),\quad\,\,
u=-\frac{1}{2}v^2-\frac{\ln \rho}{a^2_{\infty}}.
\end{eqnarray}
Then
\begin{align}
u^{(\boldsymbol{\mu})}-u&=\frac{1}{\tau^2}\Big(\sqrt{1-\tau^2 B^{(\epsilon)}(\rho^{(\boldsymbol{\mu})},v^{(\boldsymbol{\mu})},\epsilon)}-1\Big)
  -\Big(-\frac{1}{2}v^2-\frac{\ln \rho}{a^2_{\infty}}\Big)\nonumber\\[2pt]
&=-\bigg(\frac{B^{(\epsilon)}(\rho^{(\boldsymbol{\mu})}, v^{(\boldsymbol{\mu})},\epsilon)}{\sqrt{1-\tau^2 B^{(\epsilon)}(\rho^{(\boldsymbol{\mu})}, v^{(\boldsymbol{\mu})},\epsilon)}+1}
   -\frac{1}{2}v^2-\frac{\ln \rho}{a^2_{\infty}}\bigg)\nonumber\\[2pt]
&=-\big(\Psi(U^{(\boldsymbol{\mu})}, \boldsymbol{\mu})-\Psi(U)\big)\nonumber\\[2pt]
&=-\big(\Psi(U^{(\boldsymbol{\mu})}, \boldsymbol{\mu})-\Psi(U,\boldsymbol{\mu})\big)
-\big(\Psi(U, \boldsymbol{\mu})-\Psi(U)\big),\label{eq-5.7}
\end{align}
where $U^{(\boldsymbol{\mu})}=(\rho^{(\boldsymbol{\mu})}, v^{(\boldsymbol{\mu})})^{\top}$ and $U=(\rho, v)^{\top}$.

Next, we estimate the two terms $\Psi(U^{(\boldsymbol{\mu})}, \boldsymbol{\mu})-\Psi(U,\boldsymbol{\mu})$ and $\Psi(U, \boldsymbol{\mu})-\Psi(U)$ one by one.
By Lemma \ref{lem-5.2}, we have
\begin{align}
&\big\|\Psi(U^{(\boldsymbol{\mu})}, \boldsymbol{\mu})-\Psi(U,\boldsymbol{\mu})\big\|_{L^{1}((-\infty, b_0 x))}\nonumber\\[2pt]
&=\bigg\|\int^{1}_{0}\nabla_{U}\Psi(U+t(U^{(\boldsymbol{\mu})}-U), \boldsymbol{\mu})\, \dd t \cdot\big(U^{(\boldsymbol{\mu})}-U\big) \bigg\|_{L^{1}((-\infty, b_0 x))}\nonumber\\[2pt]
&\leq \bigg\|\int^{1}_{0}\frac{\big(\rho+t(\rho^{(\boldsymbol{\mu})}-\rho)\big)^{\epsilon+1} \, \dd t}{a^2_{\infty}\sqrt{1-\tau^2 B^{(\epsilon)}(\rho+t(\rho^{(\boldsymbol{\mu})}-\rho), v+t(v^{(\boldsymbol{\mu})}-v),\epsilon)}+1} \bigg\|_{L^{\infty}(\Omega_{\mathrm{w}})}\, \big\|\rho^{(\boldsymbol{\mu})}-\rho\big\|_{L^{1}((-\infty, b_0 x))}\nonumber\\[2pt]
&\quad +\bigg\|\int^{1}_{0}\frac{\big(v+t(v^{(\boldsymbol{\mu})}-v)\big)^{\epsilon+1} \, \dd t}{a^2_{\infty}\sqrt{1-\tau^2 B^{(\epsilon)}(\rho+t(\rho^{(\boldsymbol{\mu})}-\rho), v+t(v^{(\boldsymbol{\mu})}-v),\epsilon)}+1} \bigg\|_{L^{\infty}(\Omega_{\mathrm{w}})}\, \big\|v^{(\boldsymbol{\mu})}-v\big\|_{L^{1}((-\infty, b_0 x))}.\label{eq-5.8}
\end{align}

Note that, by Propositions \ref{prop:3.1}--\ref{prop:3.2}, for $\tau>0$ and $\epsilon>0$ sufficiently small,
we can choose a constant $C_2>0$ depending only on
$a_{\infty}$, $\rho^*$, and $\rho_*$ such that
\begin{align}
&\bigg\|\int^{1}_{0}\frac{\big(\rho+t(\rho^{(\boldsymbol{\mu})}-\rho)\big)^{\epsilon+1} \, \dd t}{a^2_{\infty}\sqrt{1-\tau^2 B^{(\epsilon)}(\rho+t(\rho^{(\boldsymbol{\mu})}-\rho), v+t(v^{(\boldsymbol{\mu})}-v),\epsilon)}+1} \bigg\|_{L^{\infty}(\Omega_{\mathrm{w}})}\nonumber\\[2pt]
&\,\, +\bigg\|\int^{1}_{0}\frac{\big(v+t(v^{(\boldsymbol{\mu})}-v)\big)^{\epsilon+1} \, \dd t}{a^2_{\infty}\sqrt{1-\tau^2 B^{(\epsilon)}(\rho+t(\rho^{(\boldsymbol{\mu})}-\rho), v+t(v^{(\boldsymbol{\mu})}-v),\epsilon)}+1} \bigg\|_{L^{\infty}(\Omega_{\mathrm{w}})}<C_2.\label{eq-5.9}
\end{align}

Thus, using Theorem \ref{thm:1.1} and  estimates \eqref{eq-5.8}--\eqref{eq-5.9}, we conclude
\begin{eqnarray}\label{eq-5.10}
\big\|\Psi(U^{(\boldsymbol{\mu})}, \boldsymbol{\mu})-\Psi(U,\boldsymbol{\mu})\big\|_{L^{1}((-\infty, b_0 x))}\leq C_{3}x\|\boldsymbol{\mu}\|,
\end{eqnarray}
where $C_3>0$ only depends on $a_{\infty}$, $\rho^*$, and $\rho_*$.

Furthermore, by \eqref{eq-5.2} and Lemma \ref{lem-5.2}, we have
\begin{align}\label{eq-5.11}
&\big\|\Psi(U, \boldsymbol{\mu})-\Psi(U)\big\|_{L^{1}((-\infty, b_0 x))}=\big\|\Psi(U, \boldsymbol{\mu})-\Psi(U,0)\big\|_{L^{1}((-\infty, b_0 x))}\nonumber\\[3pt]
&=\bigg\|\int^{1}_{0}\nabla_{\boldsymbol{\mu}}\Psi(U, \theta\boldsymbol{\mu})\, \dd \theta \cdot\boldsymbol{\mu} \bigg\|_{L^{1}((-\infty, b_0 x))}\nonumber\\[2pt]
&\leq \bigg\|\int^{1}_{0}\frac{\big(\theta \epsilon \rho^{\theta\epsilon}\ln\rho-\rho^{\theta\epsilon}+1\big) \, \dd \theta}
{a^2_{\infty}\sqrt{1-\theta\tau^2 B^{(\epsilon)}(\rho, v,\theta\epsilon)}+1} \bigg\|_{L^{1}((-\infty,b_0 x))}\, \epsilon \nonumber\\[2pt]
&\quad\ +\bigg\|\int^{1}_{0}\frac{B^{(\epsilon)}(\rho, v,\theta\epsilon)\, \dd \theta}{2\sqrt{1-\theta\tau^2 B^{(\epsilon)}(\rho, v,\theta\epsilon)}
\big(\sqrt{1-\theta\tau^2 B^{(\epsilon)}(\rho, v,\theta\epsilon)}+1\big)^2} \bigg\|_{L^{1}((-\infty, b_0 x))}\,\tau^2.
\end{align}

For $\tau>0$ and $\epsilon>0$ sufficiently small, we can deduce from Proposition \ref{prop:3.2} that
\begin{align}\label{eq-5.12}
& \bigg\|\int^{1}_{0}\frac{\big(\theta \epsilon \rho^{\theta\epsilon}\ln\rho-\rho^{\theta\epsilon}+1\big) \, \dd \theta}
{a^2_{\infty}\sqrt{1-\theta\tau^2 B^{(\epsilon)}(\rho, v,\theta\epsilon)}+1} \bigg\|_{L^{1}((-\infty,b_0 x))}\nonumber\\[2pt]
&\,\,\, +\bigg\|\int^{1}_{0}\frac{B^{(\epsilon)}(\rho, v,\theta\epsilon)\, \dd \theta}{2\sqrt{1-\theta\tau^2 B^{(\epsilon)}(\rho, v,\theta\epsilon)}
\big(\sqrt{1-\theta\tau^2 B^{(\epsilon)}(\rho, v,\theta\epsilon)}+1\big)^2} \bigg\|_{L^{1}((-\infty, b_0 x))}\nonumber\\[2pt]
&\leq C_{4}\|(\rho-1, v)\|_{L^{1}((-\infty, b_0 x))},
\end{align}
where $C_4>0$ depends only on $a_{\infty}$, $\rho^*$, and $\rho_*$.

It follows from \eqref{eq-5.11}--\eqref{eq-5.12} that a constant $C_5>0$ can be chosen, depending only on $a_{\infty}$, $\rho^*$, and $\rho_*$, so that
\begin{align}\label{eq-5.13}
\begin{split}
\big\|\Psi(U, \boldsymbol{\mu})-\Psi(U)\big\|_{L^{1}((-\infty, b_0 x))}\leq C_5\|\boldsymbol{\mu}\|.
\end{split}
\end{align}

Then, combining estimates \eqref{eq-5.10} and \eqref{eq-5.13} altogether and employing equality \eqref{eq-5.7}, we obtain
\begin{eqnarray}\label{eq-5.14}
\|u^{(\boldsymbol{\mu})}-u\|_{L^{1}((-\infty,b_0 x))}\leq C_{6}(1+x)\|\boldsymbol{\mu}\|,
\end{eqnarray}
where $C_6>0$ depends only on $a_{\infty}$, $\rho^*$, and $\rho_*$.

Finally, combining estimate \eqref{eq-5.14} with  estimate \eqref{eq:1.22} in Theorem \ref{thm:1.1},
we conclude \eqref{eq:1.23}. This completes the proof of Theorem \ref{coro:1.1}.
\end{proof}

\appendix
\section{}

In this appendix, we give some basic estimates of the terms obtained from the derivatives of $\mathcal{H}^{(\boldsymbol{\mu})}$
which are used in proving the optimal convergence rate as stated in \S 4.3.

\begin{lemma}\label{lem:A1}
Let $\beta_1$ be given in \eqref{eq:4.51} which satisfies \eqref{eq:4.54}.
Then, for $\delta>0$, $\epsilon>0$, and $\tau>0$ sufficiently small,
the following estimates hold{\rm :}
\medskip
\begin{enumerate}
\item[(i)]~~ $\frac{\beta^{\epsilon}_{1}-1}{\epsilon}=a_{\infty}\delta+O(1)(\delta+\epsilon+\tau^2)\delta$,

\smallskip
\item[(ii)]~~ $\frac{((\beta_1+1)\epsilon+1)\beta^{\epsilon}_{1}-1}{\epsilon}=2(a_{\infty}+1)\delta+O(1)(\delta+\epsilon+\tau^2)\delta$,

\smallskip
\item[(iii)]~~ $\frac{(\epsilon\ln\beta_1-1)\beta^{\epsilon}_{1}+1}{\epsilon^2}=\frac{a^2_{\infty}}{2}\delta^2+O(1)(\delta+\epsilon+\tau^2)\delta^2$,

\smallskip
\item[(iv)]~~ $\frac{(\epsilon(\epsilon+1)\ln\beta_1-1)\beta^{\epsilon}_1+\epsilon^2\beta^{\epsilon-1}\ln\beta_1+1}{\epsilon^2}
=\frac{3 a_{\infty}}{2}\delta+O(1)(\delta+\epsilon+\tau^2)\delta$,

\smallskip
\item[(v)]~~$\frac{((\beta_1+1)\epsilon+1)\epsilon\beta^{\epsilon}_1\ln\beta_1-\beta^{\epsilon}+1}{\epsilon^2}
=2a_{\infty}\delta+O(1)(\delta+\epsilon+\tau^2)\delta$,
\end{enumerate}
\smallskip
where the bounds of $O(1)$ depend only on $\underline{U}$.
\end{lemma}

\begin{proof}
First, using estimate \eqref{eq:4.54} and the Taylor formula, for $\delta>0$ sufficiently small, we have
\begin{align}\label{eq:A2}
\ln\beta_1=\beta_1-1+O(1)(\beta_1-1)^{2}=a_{\infty}\delta+O(1)(\delta+\epsilon+\tau^2)\delta.
\end{align}
Then, using estimate \eqref{eq:A2} and applying the Taylor formula again, we obtain
\begin{align}\label{eq:A3}
\beta^{\epsilon}_1&=1+\epsilon\ln\beta_1+\epsilon^2(\ln\beta_1)^2\int^{1}_{0}(1-t)\beta^{\epsilon t}{\rm d}t\\[3pt]
&=1+a_{\infty}\delta+O(1)(\delta+\epsilon+\tau^2)\delta\epsilon.
\end{align}
Therefore, estimate (i) can be obtained from \eqref{eq:A3}. In the similar way, we can also show estimate (ii) with the help of \eqref{eq:A2}.

Next, we turn to consider estimate (iii). To this end, we set
\begin{eqnarray*}
\psi(\epsilon)=(\epsilon\ln\beta_1-1)\beta^{\epsilon}_{1}+1.
\end{eqnarray*}
Then a direct calculation shows that $\psi(0)=0$ and
\begin{eqnarray*}
\psi'(\epsilon)=(\ln\beta_1)^2\epsilon\beta^{\epsilon}_1, \quad\,\, \psi''(\epsilon)=(\ln\beta_1)^2\big(1+\epsilon\ln\beta_1\big)\beta^{\epsilon}_1,
\quad\,\, \psi'''(\epsilon)=(\ln\beta_1)^3\big(2+\epsilon\ln\beta_1\big)\beta^{\epsilon}_1,
\end{eqnarray*}
which satisfy
\begin{eqnarray*}
\psi'(0)=0,\qquad \psi''(0)=(\ln\beta_1)^2.
\end{eqnarray*}
Thus, by \eqref{eq:A3} and the Taylor formula, we obtain
\begin{align*}
\psi(\epsilon)&=\frac{1}{2}(\ln\beta_1)^2\epsilon^2+\frac{1}{2}(\ln\beta_1)^3\epsilon^3\int^1_{0}(1-t)^2(2+t\epsilon\ln\beta_1)\beta^{t\epsilon}_{1}\,{\rm d}t\\[3pt]
&=\frac{1}{2}a^2_{\infty}\delta^2\epsilon^2+O(1)(\delta+\epsilon+\tau^2)\delta^2\epsilon^2,
\end{align*}
which leads to estimate (iii). In the same way, we can also show estimates (iv)--(v).
\end{proof}

\begin{lemma}\label{lem:A2}
Let $\mathcal{B}^{(\epsilon)}$ be defined by \eqref{eq:4.59} with $\beta_1$ and $\varphi^{(\boldsymbol{\mu})}_1$ giving in \eqref{eq:4.51}
and satisfying \eqref{eq:4.54} for $\beta_1$.
Then, for $\delta>0$, $\epsilon>0$, and $\tau>0$ sufficiently small,
\begin{eqnarray}\label{eq:A4}
\mathcal{B}^{(\epsilon)}=\frac{2}{a_{\infty}}\delta+O(1)(\delta+\epsilon+\tau^2)\delta,
\end{eqnarray}
where the bounds of $O(1)$ depend only on $\underline{U}$.
\end{lemma}

\begin{proof}
Using the Taylor formula, Lemma \ref{lem:2.6}, and estimate \eqref{eq:4.54}, for $\delta>0$, $\epsilon>0$, and $\tau>0$ sufficiently small,
we have
\begin{align}\label{eq:A5}
\varphi^{(\boldsymbol{\mu})}_1&=\varphi^{(\boldsymbol{\mu})}_1\Big|_{\beta_1=1}
+\frac{\partial\varphi^{(\boldsymbol{\mu})}_1}{\partial \beta_1}\Big|_{\beta_1=1}(\beta_1-1)+O(1)(\beta_1-1)^{2}\nonumber\\[3pt]
&=\Big(\frac{\partial\varphi^{(\boldsymbol{\mu})}_1}{\partial \beta_1}\Big|_{\beta_1=1, \epsilon=\tau=0}+O(1)(\epsilon+\tau^2)\Big)
\big(a_{\infty}\delta+O(1)(\epsilon+\tau^2)\delta\big)\nonumber\\[3pt]
&\quad +O(1)\big(a_{\infty}\delta+O(1)(\epsilon+\tau^2)\delta\big)^{2}\nonumber\\[3pt]
&=-\delta+O(1)(\epsilon+\tau^2)\delta,
\end{align}
which implies
\begin{eqnarray}\label{eq:A6}
\varphi^{(\boldsymbol{\mu})}_1+\delta=O(1)(\epsilon+\tau^2)\delta.
\end{eqnarray}
Then combining the \eqref{eq:A6} with estimate (i) in Lemma \eqref{lem:A1} leads to estimate \eqref{eq:A4}.
\end{proof}

\bigskip
\noindent
{\bf Acknowledgements}.
The research of Gui-Qiang G. Chen was supported in part by the UK Engineering and Physical Sciences Research
Council Award EP/L015811/1, EP/V008854, and EP/V051121/1.
The research of Jie Kuang was supported in part by the NSFC Project 11801549, NSFC Project 11971024, 
NSFC Project 12271507 and the Multidisciplinary Interdisciplinary Cultivation Project No.S21S6401 from
Innovation Academy for Precision Measurement Science and Technology, Chinese Academy of Sciences.
The research of Wei Xiang was supported in part by the Research Grants Council of the HKSAR, China (Project No. CityU 11304820, CityU 11300021, CityU 11311722, and CityU 11305523).
The research of Yongqian Zhang was supported in part by the NSFC Project 12271507, NSFC Project 11421061, NSFC Project 11031001, NSFC Project 11121101,
the 111 Project B08018(China) and the Shanghai Natural Science Foundation 15ZR1403900.


\begin{thebibliography}{10}

\bibitem{Asakura} F.~Asakura,
{Wave-front tracking for the equations of isentropic gas dynamics},
{\it Q. Appl. Math.} 63 (2005), 20--33.


\bibitem{Anderson} J.~Anderson,
{\it Hypersonic and High-Temperature Gas Dynamics},
Second Edition, AIAA Education Series: Reston, 2006.


\bibitem{Bressan} A.~Bressan,
{\it Hyperbolic Systems of Conservation Laws. The One-Dimensional Cauchy Problem},
Oxford University Press: Oxford, 2000.


\bibitem{Chen-Christoforou-Zhang-1} G.-Q.~Chen, C.~Christoforou, and Y.~Zhang,
{Dependence of entropy solutions with large oscillations to the Euler equations on the nonlinear flux functions},
{\it Indiana Univ. Math. J.} 56 (2007), 2535--2568.


\bibitem{Chen-Christoforou-Zhang-2} G.-Q.~Chen, C.~Christoforou, and Y.~Zhang,
{Continuous dependence of entropy solutions to the Euler equations on the adiabatic exponent and Mach number},
{\it Arch. Ration. Mech. Anal.} 189 (2008), 97--130.

\bibitem{ChenGironSchulz}
G.-Q.~Chen, T.~P. Giron, and S.~M. Schulz,
The Morawetz problem for supersonic flow with cavitation,
arXiv preprint arXiv:2401.17524, 2024.


\bibitem{Chen-Kuang-Zhang} G.-Q.~Chen, J.~Kuang, and Y.~Zhang,
{Two-dimensional steady supersonic exothermically reacting Euler flow past Lipschitz bending walls},
{\it SIAM J. Math. Anal.} 49 (2017), 818--873.

\bibitem{Chen-Kuang-Zhang2} G.-Q.~Chen, J.~Kuang, and Y.~Zhang,
Stability of conical shocks in
three-dimensional steady supersonic isothermal flows past Lipschitz perturbed cones,
{\it SIAM J. Math. Anal.} 53 (2021), 2811--2862.

\bibitem{Chen-Kuang-Xiang-Zhang}
G.-Q.~Chen, J.~Kuang, W.~Xiang, and Y.~Zhang, Hypersonic similarity
for the steady compressible Euler flows over two-dimensional Lipschitz wedges,
Preprint, arXiv:2304.12925v1,2023.





\bibitem{Chen-Li} G.-Q.~Chen and T.-H.~Li,
{Well-posedness for two-dimensional steady supersonic Euler flows past a Lipschitz wedge},
{\it J. Differential Equations}, 244 (2008), 1521--1550.

\bibitem{ChenSlemrodWang} G.-Q.~Chen, M.~Slemrod, and D.~Wang,
Vanishing viscosity method for transonic flow,
{\it Arch. Ration. Mech. Anal.} 189 (2008), 159--188.


\bibitem{Chen-Xiang-Zhang} G.-Q.~Chen, W.~Xiang, and Y.~Zhang,
{Weakly nonlinear geometric optics for hyperbolic systems of conservation laws},
{\it Commun. Partial Differential Equations}, 38 (2015), 1936--1970.


\bibitem{Chen-Zhang-Zhu-1} G.-Q.~Chen, Y.~Zhang, and D.-W.~Zhu,
{Existence and stability of supersonic Euler flows past Lipschitz wedges},
{\it Arch. Ration. Mech. Anal.} 181 (2006), 261--310.



\bibitem{Colombo-Risebro} R.~M.~Colombo and N.~H.~Risebro,
{Continuous dependence in the large for some equations of gas dynamics},
{\it Commun. Partial Differential Equations}, 23 (1998), 1693--1718.


\bibitem{Courant-Friedrichs} R.~Courant and K.~Friedrichs,
{\it Supersonic Flow and Shock Waves},
Interscience Publishers Inc.: New York, 1948.

\bibitem{Dafermos2016}
C.~M. Dafermos.
\newblock {\it Hyperbolic Conservation Laws in Continuum Physics},
4th Edition,
\newblock Springer-Verlag, Berlin, 2016.


\bibitem{Glimm1965}
J.~Glimm.
\newblock Solutions in the large for nonlinear hyperbolic systems of equations.
\newblock {\em Comm. Pure Appl. Math.} 18 (1965), 697--715.


\bibitem{Kuang-Xiang-Zhang-1} J.~Kuang, W.~Xiang, and Y.~Zhang,
{Hypersonic similarity for the two dimensional steady potential flow with large data},
{\it Ann. Inst. H. Poincar\'{e} Anal. Non Lin\'{e}aire}, 37 (2020), 1379--1423.


\bibitem{Kuang-Xiang-Zhang-2} J.~Kuang, W.~Xiang, and Y.~Zhang,
{Convergence rate of hyperbolic similarity for steady potential flows over two-dimensional Lipschitz wedge},
{\it Calc. Var.} 62 (2023), art. no. 106, 49pp.


\bibitem{Landau} L.~Landau and E.~Lifschitz,
{\it Fluid Mechanics},
2nd Edition, Elsevier Ltd.: Singapore, 2004.


\bibitem{Nishida} T.~Nishida,
{Global solution for an initial-boundary value problem of a quasilinear hyperbolic system},
{\it Proc. Jap. Acad.} 44 (1968), 642--646.


\bibitem{Nishida-Smoller-1} T.~Nishida and J.~Smoller,
{Solutions in the large for some nonlinear hyperbolic conservation laws},
{\it Comm. Pure Appl. Math.} 26 (1973), 183--200.


\bibitem{Nishida-Smoller-2} T.~Nishida and J.~Smoller,
{Mixed problems for nonlinear conservation laws},
{\it J. Differential Equations},  23 (1977), 244--269.

\bibitem{Qu-Wang-Yuan} A.~Qu, L.~Wang, and H.~Yuan,
{Radon measure solutions for steady hypersonic-limit Euler flows passing two-dimensional finite non-symmetric obstacles and interactions of free concentration layers.}
{\it Comm. Math. Sci.} 19 (2021), 875--901.


\bibitem{Qu-Yuan-Zhao} A.~Qu, H.~Yuan, and Q.~Zhao,
{Hypersonic limit of two-dimensional steady compressible Euler flows passing a straight wedge},
{\it Z. Angew. Math. Mech.} 100 (2020):e201800225, 14pp.


\bibitem{smoller} J. Smoller,
{\it Shock Waves and Reaction-Diffusion Equations},
Second Edition, Springer-Verlag, Inc.: New York, 1994.



\bibitem{Tsien} H.-S.~Tsien,
{Similarity laws of hypersonic flows},
{\it J. Math. Phys.} 25 (1946), 247--251.


\bibitem{Smoller} J.~Smoller,
{\it Shock Waves and Reaction-Diffusion Equations},
Second Edition, Springer-Verlag, Inc.: New York, 1994.


\bibitem{Temple} J.~Temple,
{Solutions in the large for the nonlinear hyperbolic conservation laws of gas dynamics},
{\it J. Differential  Equations}, 41 (1981), 96--161.

\bibitem{Dyke} M. Van~Dyke,
{A study of hypersonic small disturbance theory},
NACA Report, 1194, April, 1954.



\bibitem{Zhang-1} Y.~Zhang,
{Global existence of steady supersonic potential flow past a curved wedge with piecewise smooth boundary},
{\it SIAM J. Math. Anal.} 31 (1999), 166--183.

\bibitem{Zhang-2} Y.~Zhang,
{Steady supersonic flow past an almost straight wedge with large vertex angle},
{\it J. Differential Equations}, 192 (2003), 1--46.
\end{thebibliography}
\end{document}